\documentclass[table,x11names]{amsart}
 
\newif\ifshort
\shorttrue
\ifshort
\usepackage[top=1.1in, bottom=1.1in, left=1.3in, right=1.3in]{geometry}
\fi

\usepackage{va, styell}
\newcommand\ocK{\overline{\cK}}

\usepackage{bm}
\usepackage{float}
\usepackage{tikz}
\usepackage{tikz-cd}

\usepackage[pdftex]{hyperref}
\hypersetup{colorlinks,citecolor=blue,linktocpage,hyperindex=true,backref=true}
\allowdisplaybreaks

\usepackage{etoolbox}

\makeatletter
\pretocmd{\section}{\addtocontents{toc}{\protect\addvspace{5\p@}}}{}{}
\makeatother

\usepackage{todonotes}
\date{March 31, 2023} 

\title{Compact moduli of K3 surfaces}
\author{Valery Alexeev}
\email{valery@uga.edu}
\address{Department of Mathematics, University of Georgia, Athens GA
  30602, USA}
\author{Philip Engel}
\email{philip.engel@uga.edu}
\address{Department of Mathematics, University of Georgia, Athens GA
  30602, USA}
\begin{document}
%\numberwithin{equation}{section}

\begin{abstract}
We construct geometric compactifications of the moduli space
$F_{2d}$ of polarized K3 surfaces, in any degree $2d$. 
Our construction is via KSBA theory, by considering canonical choices of divisor
$R\in |nL|$ on each polarized K3 surface $(X,L)\in F_{2d}$.
The main new notion is that of a {\it recognizable
divisor} $R$, a choice which can be consistently extended to all
central fibers of Kulikov models. We prove that any choice
of recognizable divisor leads to a semitoroidal
compactification of the period space, at least up to normalization.
Finally, we prove that the rational curve divisor is
recognizable for all degrees.
\end{abstract}

\maketitle

\ifshort 
\setcounter{tocdepth}{2}
\else
\setcounter{tocdepth}{2}  
\fi
\tableofcontents

%\listoffigures
% \listoftables

\section{Introduction}
\label{sec:introduction}

Let $F_{2d}$ be the coarse moduli space
of complex K3 surfaces $X$ having ADE singularities 
with an ample line bundle $L$ of degree $L^2=2d$.
A well known corollary of the Torelli theorem \cite{piateski-shapiro1971torelli} is
that $F_{2d}=\bD/\Gamma$ is the
quotient of a $19$-dimensional symmetric type IV domain $\bD$
by an arithmetic group $\Gamma\subset O(2,19)$. 
In this capacity, $F_{2d}$ admits a Baily-Borel $\oF_{2d}^\bb$
\cite{baily1966compactification-of-arithmetic}
and infinitely many toroidal 
$\oF_{2d}^\mathfrak{F}$
\cite{ash1975smooth-compactifications}
compactifications.
An admissible fan $\mathfrak{F}$ consists of
polyhedral decompositions of the positive cones of finitely many
hyperbolic signature lattices (Def.~\ref{def:fan}). Looijenga
\cite{looijenga2003compactifications-defined2} simultaneously
generalized the Baily-Borel and toroidal compactifications
to the \emph{semitoroidal} compactifications,
where $\mathfrak{F}$ may only be locally rational polyhedral
(Def.~\ref{def:semifan}).

Toroidal compactifications enjoy a number of geometric
properties by virtue of the fact that they are analytically-locally
modeled at the boundary by finite
quotients of toric varieties. But there are infinitely many,
with seemingly no one being distinguished.
An old and deep question is whether any toroidal, or semitoroidal,
compactifications can be understood as moduli spaces parameterizing
geometric objects---some generalized
``stable'' K3 surfaces, similar to the Deligne-Mumford's compactifications
$\oM_{g,n}$ of stable curves.

For the moduli space $A_g$ of principally polarized abelian varieties (PPAVs)
the answer is \emph{yes} by \cite{alexeev2002complete-moduli}. A PPAV
$(A,\lambda)\in A_g$ determines uniquely an
abelian torsor $A\acts X$
together with a theta divisor $\Theta\subset X$.
For  pairs $(X,\Theta)$ or, even better $(X,\epsilon \Theta)$,
there is a generalization of $\oM_{g,n}$. It
is the moduli space of KSBA stable pairs $(X,\epsilon \Theta)$ with slc
singularities 
\cite{kollar1988threefolds-and-deformations, kollar2023families-of-varieties},
\cite{alexeev1996log-canonical-singularities,
  alexeev2006higher-dimensional-analogues}. 
This moduli space is projective, and gives
a geometrically meaningful compactification~$\oA_g^\Theta$. 
Furthermore, the normalization of $\oA_g^\Theta$ is the toroidal compactification
$\oA_g^\mathfrak{F}$ associated to the 2nd Voronoi fan,
and so admits a purely period-theoretic definition.
Thus, among the infinitely many toroidal
compactifications, this one
has a clear geometric meaning.

To extend this construction to polarized K3 surfaces $(X,L)$, first 
one needs a \emph {canonical choice of polarizing divisor} (Def.~\ref{def:can-choice}),
an effective divisor $R\in |nL|$ in a fixed multiple of the polarization.
Given this choice, the general theory
produces a modular compactification $\oF_{2d}^R$ (Def.~\ref{stable-pair-comp})
via slc stable pairs $(X,\epsilon R)$,
see \cite{kollar2019moduli-of-polarized, kollar2023families-of-varieties}, 
\cite[Sec.~3]{alexeev2019stable-pair}.
The divisor is needed because for the general theory to work, the divisor
$K_X+\epsilon R$ must be ample. One can work with all divisors
in $|L|$, without making a canonical choice, e.g.
\cite{laza2016ksba-compactification},
but that gives a
larger moduli space $P_{2d}$ of dimension $20+d$.

At least two canonical choices for ample divisors on polarized
K3 surfaces have been identified. About 15 years ago, Sean Keel
proposed to consider, for a general polarized K3 surface $(X,L)$, the
sum $R^{\rm rc}=\sum C_i$ of rational curves in $|L|$. 
We call this the \emph{rational curve divisor} (Def.~\ref{naive-rc-def}).
One has $R^{\rm rc}\in |n_dL|$,
where the multiple $n_d$ is given by the Yau-Zaslow formula. For instance
$n _1=324$, $n_2=3200$, etc. 
The second choice, suggested to the authors by Claire Voisin,
is called the \emph{flex divisor} $R^{\rm flex}$ \cite{alexeev2021flex-divisor}.
It generalizes to all degrees the fixed locus $R\in |3L|$ of the involution
on a K3 surface of degree 2.

By the Kulikov-Persson-Pinkham theorem
\cite{kulikov1977degenerations-of-k3-surfaces,
persson1981degeneration-of-surfaces},
any one-parameter degeneration of K3 surfaces $X\to (C,0)$
admits a {\it Kulikov model}: a $K$-trivial model with smooth total
space and reduced normal crossings central fiber $X_0$ (Def.~\ref{def:kulikov}).
The key notion of this paper is a {\it recognizable
divisor} (Def.~\ref{recog-def}). Heuristically, it is a canonical choice of polarizing divisor
which can be extended to any such $X_0$. More precisely:
Given any Kulikov surface $X_0$ appearing as a one-parameter degeneration of
polarized K3 surfaces $(X_t,L_t)$, the limit of the canonically chosen
divisors $R_t\subset X_t$, $R_t\in |nL_t|$ is a unique curve
$R_0 = \lim_{t\to 0}R_t\subset X_0$.
Such a limit $R_0$ exists on any fixed
Kulikov model, but recognizability additionally
states that {\it $R_0$ is independent of how $X_0$ gets
smoothed}.

Our two main theorems are:

\begin{theorem*}[Thm.~\ref{main-thm}]
\label{thm:recognizable-semitoroidal}Suppose that $R$ is a recognizable choice of polarizing divisor.
There is a unique semifan $\mathfrak{F}_R$ for which
$\oF_{2d}^{\mathfrak{F}_R}$ is the normalization of $\oF_{2d}^R$.
\end{theorem*}
\begin{theorem*}[Thm.~\ref{thm:rc-divisor}]\label{thm:main-rc-divisor}
  The rational curve divisor $R^{\rm rc}$ is recognizable for $F_{2d}$.
\end{theorem*}

Theorem \ref{thm:recognizable-semitoroidal} holds more generally for moduli of
lattice-polarized K3 surfaces. Combined, Theorems \ref{thm:recognizable-semitoroidal}
and \ref{thm:main-rc-divisor} give an affirmative answer to the existence of a
compactification of K3 moduli
which is simultaneously geometric and period-theoretic:

\begin{corollary*}\label{cor:some-semitoroidal}
For all degrees $2d$, there is a KSBA compactification 
of $F_{2d}$ by slc stable pairs, whose normalization is 
semitoroidal. \end{corollary*} 

Theorem \ref{thm:recognizable-semitoroidal} is proven as follows: Kulikov 
models $X\to (C,0)$ with a given Picard-Lefschetz transformation, encoded by
a lattice vector $\lambda$, can be packaged into a $19$-dimensional family of
polarized surfaces $\cX\to S_\lambda$ we call {\it $\lambda$-families} (Def.~\ref{lambda-fam}).
The general fiber is smooth, and the discriminant $\Delta_\lambda\subset S_\lambda$
is a smooth divisor, isomorphic to $(\C^*)^{18}$ for Type III Kulikov models.
The discriminant parameterizes the equisingular, quasipolarized, smoothable deformations of $X_0$.

Recognizability implies that the divisor $R$ extends over the boundary
$\Delta_\lambda$ to give a family of pairs $(\cX,\cR)\to S_\lambda$ (Prop.~\ref{stronger-rec}).
We modify $\cX$ until $\cR\subset \cX$ is relatively nef and contains no singular
strata of any fiber (Prop.~\ref{divisor-lambda-rec},
Sec.~\ref{sec:global-fs}). Taking the relative canonical model shows that all degenerations
with a given Picard-Lefschetz transformation limit to stable pairs
 $(X_0,\epsilon R_0)$ of a fixed combinatorial type (Cor.~\ref{cor:slc-type-defined-by-lambda}).
 This fact, together with an argument involving resolution of indeterminacy (Lem.~\ref{toric-resolution})
 and quasi-affineness of the strata of the KSBA compactification (Thm.~\ref{stable-all-possible}), imply
 that there is a toroidal compactification
 $$\oF_{2d}^{\mathfrak{G}}\to (\oF_{2d}^R)^{\nu}$$ dominating
 the normalization of the KSBA compactification. The proof concludes
 with a new characterization: semitoroidal compactifications are exactly 
 those normal compactifications of $F_{2d}$ that are dominated
 by a toroidal compactification, and dominating Baily-Borel (Thm.~\ref{thm:semi-is-normal}).

Theorem \ref{thm:main-rc-divisor} is proven by borrowing ideas from
Gromov-Witten theory and degenerations of stable maps.
$R^{\rm rc}$ is recognizable because any limit $R_{i,0}\subset X_0$ of
a family of rational curves $R_{i,t}\subset R^{\rm rc}_t\subset X_t$ in a Kulikov model $X\to (C,0)$ enjoys
a geometric property which ensures its rigidity:
$R_{i,0}$ is the image of an admissible stable genus zero
map $f:T\to X_0$.
Using $K$-triviality of $X_0$ and adjunction,
we prove that $f(T)$ is locally constant on the
Kontsevich space of admissible stable maps (Lem.~\ref{rc-rec2}).

\subsection*{Relation to earlier work} The notion
of a recognizable divisor presented here arose from generalizing certain
specific examples, for moduli of degree $2$ \cite{alexeev2019stable-pair}
and elliptic K3 surfaces \cite{alexeev2022compactifications-moduli}.
In both of these papers, Kulikov models with divisor are constructed explicitly for all possible
$\lambda$, providing the necessary input for the general theory to work.

Every degree~2 K3 surface $(X,L)$ with ADE singularities admits an involution,
and the fixed locus $R\in |3L|$ can be taken as a canonical choice of polarizing divisor.
Every elliptic K3 surface admits the polarizing divisor $R:=s+\sum_{i=1}^{24} f_i\in |s+24f|$ formed from
the section, plus the sum of the singular fibers counted with appropriate multiplicity.

Using the theory of integral-affine structures on the two-sphere $S^2$
\cite{engel2018looijenga,engel2021smoothings}, one can, in both of these cases,
explicitly construct a family of Kulikov surfaces $\cX_0\to \Delta_\lambda$ which smooths
 to a $\lambda$-family $(\cX,\cR)\to S_\lambda$
with $\cR$ relatively big and nef, and not containing strata of any fiber.
This is achieved by building a continuously
varying family of ``polarized integral-affine spheres" 
 $(\cB,\cR_{\rm trop}) \to C_\delta^+$ over a cone associated to
 an appropriate hyperbolic lattice (Def.~\ref{def:c-plus}).
 The cone $C_\delta^+$ contains all possible $\lambda$. Once triangulated,
 the integral-affine sphere $\mathcal{B}_\lambda=\Gamma(X_0)$ can be 
 identified with the dual complex of $X_0$ for some Kulikov model
 $X\to (C,0)$ with monodromy $\lambda$. The tropical divisor $\cR_{{\rm trop},\lambda}\subset \cB_\lambda$
 describes the dual complex $\Gamma(R_0)$ of $R_0\subset X_0$.

The upshot of these constructions is that recognizability is verified explicitly,
and all degenerations with a fixed Picard-Lefschetz transformation admit a
stable model of a fixed combinatorial type.
Furthermore, the family $(\mathcal{B},\cR_{\rm trop})\to C_\delta^+$
determines which monodromy invariants $\lambda$
give rise to degenerations into a specified stratum of slc stable pairs---it is those
$\lambda$ on which the family $(\cB,\cR_{\rm trop})$ is combinatorially constant.
In the above examples, these loci of combinatorial constancy in $C_\delta^+$ are
the cones of a semifan $\fF_R$.
The semifan of Theorem \ref{thm:recognizable-semitoroidal} is exactly this one.
In fact, we prove here that for any recognizable divisor, the cones of $\fF_R$ are the loci on
which a well-defined ``stratum" function $\mathbb{S}\colon C_\delta^+\to \{\textrm{slc strata of }\oF_{2d}^R\}$
is constant (Thm.~\ref{stratum-function}).

For degree $2$ K3 surfaces with ramification divisor, $\fF_R$ is a semifan but not a fan. It is
a coarsening of the Coxeter fan. For elliptic K3 surfaces with divisor $s+\sum_{i=1}^{24}f_i$
it is a fan which refines the maximal cone of the Coxeter fan into $9$ subcones.
Theorems \ref{thm:recognizable-semitoroidal} and \ref{thm:main-rc-divisor} imply the
existence of a semifan $\mathfrak{F}^{\rm rc}$ for all degrees $2d$. But
unlike for $A_g$ and the 2nd Voronoi fan, or the above two examples,
we have no explicit description, because the structure of a hypothetical ``tropical 
K3-rational curves pair" $(\cB, \cR^{\rm rc}_{\rm trop})$ is unknown.
Such a description
is an interesting open question. Integral-affine structures make no appearance
in this paper because we do not explicitly construct $\fF_R$ for any given $R$---we
prove a general existence result.

After this work appeared, the authors, with Han \cite{alexeev2021nonsymplectic},
proved recognizability for fixed curves of non-symplectic automorphisms.
Explicit semifans $\fF_R$ for the fixed divisor $R$ were given
in \cite{alexeev2022mirror-symmetric}, for moduli spaces of K3 surfaces
with nonsymplectic involution.

\subsection*{Summary of contents}

 In Section~\ref{sec:k3-moduli}, we recall different notions of
 moduli of K3 surfaces, such as smooth analytic,
$M$-quasipolarized, and polarized with ADE singularities.
In Section~\ref{sec:one-param-degens} we study one-parameter
degenerations: Kulikov models, as well as nef, divisor, and stable
models, by adding nef line bundles, effective nef divisors, and by
taking the canonical models of the latter, respectively.
In Section~\ref{sec:period-map} we define the periods of Kulikov
surfaces. These sections compile
known results about K3 moduli, giving a unified
treatment of Type II and III degenerations.

Section~\ref{sec:arithmetic-quotients} recalls the combinatorially
defined, period-theoretic compactifications of arithmetic quotients:
Baily-Borel, toroidal, and semitoroidal.
A major result is
Theorem \ref{thm:semi-is-normal} which states that, for Type IV
arithmetic quotients, a semitoroidal
compactification is precisely a normal compactification which is
sandwiched between the Baily-Borel and some toroidal compactification.
Section~\ref{sec:slc-pairs} discusses the compactifications
of $F_{2d}$ via stable pairs, associated to a
canonical choice of divisor $R$. Here, we introduce the critical
notion of recognizable divisors (Def.~\ref{recog-def}).

In Section~\ref{sec:lambda-families}, we construct the $\lambda$-families
which appear in the proof of Theorem \ref{thm:recognizable-semitoroidal}.
A new result (Theorem \ref{global-fs}) globalizes the main
result of Friedman-Scattone \cite{friedman1986type-III}: Any two $\lambda$-families
with the same values of $\lambda^2$ and
imprimitivity of $\lambda$ are connected by a series of birational
modifications falling into three special forms. 
In Section~\ref{sec:recognizable-divisors}, we prove the
 main properties of recognizable divisors with respect to $\lambda$-families.
Theorem \ref{thm:equiv-rec} summarizes equivalent formulations of
recognizability.

Sections~\ref{sec:main-thm-recognizable} and \ref{sec:rc-divisor}
contain the proofs of Theorems
\ref{thm:recognizable-semitoroidal} and \ref{thm:main-rc-divisor}, respectively.

\section{Moduli of K3 surfaces}
\label{sec:k3-moduli}

\subsection{Analytic moduli}\label{analyticK3}
We begin by setting notation and reviewing fundamental
results about K3 surfaces. For general references,
see \cite{huybrechts2016lectures-on-k3} or 
\cite{asterisque1985geometrie-des-surfaces}.

\begin{definition}
A {\it K3 surface} $X$ is a compact, complex surface 
with $h^1(X)=0$ and $K_X=\mathcal{O}_X$. \end{definition}

\begin{definition}
Let $L_{K3}:=I\!I_{3,19}=H^{\oplus 3}\oplus E_8^{\oplus 2}$
be a fixed copy of the unique even unimodular lattice of
signature $(3,19)$. 
\end{definition}

Endowed with the cup product, $H^2(X,\Z)$ is isometric to $L_{K3}$
for any K3 surface $X$. Let $\mathcal{K}_X\subset H^{1,1}(X,\R)$
denote the K\"ahler cone of $X$. It is a fundamental chamber 
for the group $W_X=\langle r_\beta\rangle\subset O(H^2(X,\Z))$
generated by reflections in the roots $\beta\in NS(X)$, $\beta^2=-2$ acting
on the positive cone of $H^{1,1}(X,\R)$.

\begin{theorem}[\cite{piateski-shapiro1971torelli,looijenga1980torelli}]
\label{torelli-i}
  Two K3 surfaces $X$, $X'$ are isomorphic if and only if
  they are {\rm Hodge-isometric}: there is an isometry
  $i : H^2(X',\Z)\to H^2(X,\Z)$ for which
  $i(H^{2,0}(X'))=H^{2,0}(X)$. Furthermore, $i=f^*$ for a unique
  isomorphism $f:X\to X'$ if and only if
  $i(\mathcal{K}_{X'})=\mathcal{K}_X$.
\end{theorem}

Note that $\pm 1$ and $g\in W_X$ act by Hodge isometries on
$H^2(X,\Z)$. For any Hodge isometry $i$ between $X'$ and $X$, there is
a unique sign and unique element $g\in W_X$ such that
$\pm g\circ i(\mathcal{K}_{X'})=\mathcal{K}_X$. Thus, the
group of Hodge isometries of $X$ fits into a split exact sequence of groups
$$0\to \{\pm 1\}\times W_X\to {\rm HodgeIsom}(X)\to {\rm Aut}(X)\to
0.$$

\begin{definition}
Let $\pi \colon \mathcal{X}\to S$ be a family of
  smooth analytic K3 surfaces over an analytic space $S$. A {\it
    marking} is an isometry of local systems
  $\sigma\colon R^2\pi_*\underline{\Z}\to \underline{L}_{K3}$.
\end{definition}

\begin{definition} 
The {\it period domain} of analytic K3 surfaces is
  $$\mathbb{D}:=\mathbb{P}\{x\in L_{K3}\otimes \C\,\big{|}\, x\cdot x=0,
  \,x\cdot \overline{x}>0\}.$$ It is an analytic open subset of a
  $20$-dimensional quadric in $\mathbb{P}^{21}$. Let
  $(\mathcal{X}\to S, \,\sigma)$ be a marked family of K3
  surfaces. The {\it period map} $P\colon S\to \mathbb{D}$ is defined by
$s\mapsto \sigma(H^{2,0}(X_s)).$ 
\end{definition}

By \cite[Exp. XIII]{asterisque1985geometrie-des-surfaces},
there is a non-Hausdorff complex manifold $\cM$ of dimension~$20$,
forming a fine moduli space of marked K3 surfaces, together with a period
map $P:\cM\to \bD$. For $x\in \mathbb{D}$ a period, let $W_x$
be the group generated by reflections in roots of $x^\perp\cap L_{K3}$.
Then $P^{-1}(x)$ is a torsor over $\{\pm 1\}\times W_x$ with action given by
$(X,\sigma)\mapsto (X,g\circ \sigma)$.

\subsection{Quasipolarized moduli}\label{sec:qpol-moduli}
We now give analogous definitions to Section \ref{analyticK3} in the
polarized case. The standard reference is
\cite{dolgachev1996mirror-symmetry}. However, Thm. 3.1 in \emph{ibid}
is incorrect. We modify the definition in a way that
this theorem remains true.

Let $J\colon M\hookrightarrow L_{K3}$ be a primitive 
hyperbolic sublattice of signature $(1,r-1)$ with $r\leq 20$. 
A vector $h\in M\otimes \R$ is {\it very irrational} if $h\notin
M'\otimes \bR$ for any proper sublattice $M'\subsetneq M$.
We will henceforth fix one such, of positive norm $h^2>0$.

\begin{definition}\label{def:mqpol}
  An {\it $M$-quasipolarized K3 surface} $(X,j)$ is a K3 surface
  $X$, and a primitive lattice embedding
  $j\colon M\hookrightarrow NS(X)$ for which $j(h)\in \ocK_X$ is big
  and nef. Two such $(X,j)$, $(X',j')$ are {\it isomorphic} if there is
  an isomorphism $f\,:X\to X'$ of K3 surfaces for which
  $j=f^*\circ j'$.
\end{definition}

\begin{definition}
A {\it marking} of $(X,j)$ is an isometry
  $\sigma:H^2(X,\Z)\to L_{K3}$ for which $J = \sigma\circ j$.
\end{definition}

The $M$-quasipolarized period domain is
$$\mathbb{D}_M:=\mathbb{P}\{x\in M^\perp \otimes \C \,\big{|}\,x\cdot
x=0,\,x\cdot \overline{x}>0\}.$$ Define the {\it Weyl group} of a
point $x\in \mathbb{D}_M$ to be
  $W_x(M^\perp):=\langle r_\beta\colon \beta \in x^\perp\cap
  M^\perp\rangle$. Note that now $W_x(M^\perp)$ is finite because
$x^\perp\cap M^\perp$ is negative-definite.

\begin{theorem}\label{qpol-thm}
There is a non-Hausdorff complex manifold
$\mathcal{M}_M\subset \mathcal{M}$
of dimension $20-r$, 
admitting a universal family of marked $M$-quasipolarized
K3 surfaces. The fiber $P\inv(x)$ of the
  period map $P\colon \mathcal{M}_M \to \mathbb{D}_M$ is a
  torsor over $W_x(M^\perp)$.
\end{theorem}

\begin{proof} The proof follows that of
\cite[Thm. 3.1]{dolgachev1996mirror-symmetry}, which now works
because of the modified Definition \ref{def:mqpol} for
an $M$-quasipolarization.
\end{proof}

Let $\mathcal{F}_M^{\rm q}$ denote the moduli stack of (unmarked)
$M$-quasipolarized K3 surfaces.

\begin{corollary}
There is an isomorphism of stacks
$\mathcal{F}_M^{\rm q} = [\mathcal{M}_M:\Gamma]$ where
$$\Gamma:=\{\gamma\in O(L_{K3})\colon \gamma\big{|}_M = {\rm id}_M \}$$
is the group of changes-of-marking. The quotiented period map
$\mathcal{M}_M/\Gamma\to \mathbb{D}_M/\Gamma$ is a bijection.
\end{corollary}

\subsection{ADE moduli}
We now modify the above moduli problems to produce
a Hausdorff moduli space. Theorem \ref{smooth-coarse} below is
well-known. 

\begin{definition}
An {\it $M$-polarized K3 surface} $(\oX,j)$ is a
surface $\oX$ with at worst rational
double point (ADE) singularities
whose minimal resolution $X\to \oX$ is a smooth
K3 surface, together with an isometric embedding
$j:M\hookrightarrow {\rm Pic}(\oX)$ for which
$j(h)$ is ample.
\end{definition}

\begin{theorem}\label{smooth-coarse}
The coarse moduli space of
$M$-polarized K3 surfaces is $F_M=\mathbb{D}_M/\Gamma$. The moduli
stack $\cF_M$ is the separated quotient
of the stack $[\mathcal{M}_M:\Gamma]$.
\end{theorem}

\begin{remark}\label{stack-difference}
The stack $\mathcal{F}_M$ and the
quotient stack $[\mathbb{D}_M:\Gamma]$ are not equal.
In the latter stack, the inertia group at $x\in \mathbb{D}_M$
is the stabilizer $\Gamma_x$. In the former stack, the Torelli Theorem \ref{torelli-i}
implies the inertia group is rather $\Gamma_x/W_x(M^\perp)={\rm Aut}(\oX,j)$.

Consider an open neighborhood
$U_x\ni x$ in $\mathbb{D}_M$
preserved by $\Gamma_x$. First, quotient
$U_x$ by $W_x(M^\perp)$. Since $W_x(M^\perp)\subset\Gamma_x$
is normal, the quotient group
acts on the coarse space $U_x/W_x(M^\perp)$, which is a smooth
complex manifold. The stack
quotient $[U_x/W_x(M^\perp):\Gamma_x/W_x(M^\perp)]$ defines 
orbifold charts for the smooth DM stack $\mathcal{F}_M$. \end{remark}

\section{One-parameter degenerations}
\label{sec:one-param-degens}

\subsection{Kulikov models}
\label{sec:kulikov-models}

We now examine degenerations of K3 surfaces over a curve.
Let $(C,0)$ denote the analytic germ of a smooth curve at a
point $0\in C$ and let $C^*=C\setminus 0$. Let $X^*\to C^*$ be
a family of smooth analytic K3 surfaces. 

\begin{definition}\label{def:kulikov}
A {\it Kulikov model} $X\to (C,0)$ is an
extension of $X^*\to C^*$ for which $X$ is smooth,
$K_X\sim_C 0$, and $X_0$ has reduced normal crossings
with all components K\"ahler.
We say $X$ is {\it Type I, II,} or {\it III}, respectively, depending
on whether $X_0$ is smooth, has double curves but no triple
points, or has triple points, respectively. \end{definition}

A key result is the theorem of
Kulikov \cite{kulikov1977degenerations-of-k3-surfaces}
and Persson-Pinkham \cite{persson1981degeneration-of-surfaces}:

\begin{theorem}
Let $Y^*\to C^*$ be a family of
analytic K3 surfaces admitting an extension
$Y\to (C,0)$ for which every component of $Y_0$ is K\"ahler.
There is a base change $(C',0)\to (C,0)$ and a sequence of
bimeromorphic modifications $Y'\dashrightarrow X$
of the pullback, such that $X$ is a Kulikov model.
\end{theorem}

%\begin{remark}
%The assumption that $Y_0$ have K\"ahler
%components is necessary \cite{nishiguchi1988degeneration}.
%\end{remark}

Assume for notational
convenience that the strata of $X_0$ are globally normal crossings.
Let $V_i\subset X_0$ denote the irreducible components, $D_{ij}=V_i\cap V_j$ and $T_{ijk}=V_i\cap V_j\cap V_k$ the double curves and triple points, respectively.
By convention, we write $D_{ij}\subset V_i$ and $D_{ji}\subset V_j$.

\begin{proposition}\label{prop:kulikov}
Let $X\to (C,0)$ be a Kulikov
model. Let $D_i=\sum_j D_{ij}$ be the part of the double locus
of $X_0$ lying on $V_i$. Then:
\begin{enumerate}
\item $D_i\in |-K_{V_i}|$ is an anticanonical cycle of
rational curves in Type III, and an elliptic curve or the disjoint union
of two elliptic curves in Type II.
\item $D_{ij}^2+D_{ji}^2=-2+2g$ where $g$ is the arithmetic genus
of $D_{ij}$ in $X_0$.
\item The dual complex $\Gamma(X_0)$ is a triangulation
of $S^2$ in Type III, and a segment decomposed
into subsegments in Type II.
\end{enumerate}
\end{proposition}

\begin{definition} A reduced normal crossings surface $X_0$ satisfying (1), (2), (3)
is a {\it Kulikov surface}. \end{definition}

There is a converse to Proposition \ref{prop:kulikov}
due to Friedman \cite{friedman1983global-smoothings}:

\begin{theorem}\label{dss:def}
Let $X_0$ be a Kulikov surface.
Then, $X_0$ deforms to a smooth K3 surface if and only
if it satisfies an additional property called {\rm $d$-semistability}:
$$\mathcal{E}xt^1(\Omega_{X_0},\,\mathcal{O}_{X_0})\cong \mathcal{O}_{(X_0)_{\rm sing}}.$$
\end{theorem}

The components $V_i$ in Type III are rational surfaces with a nodal
anticanonical cycle as the double locus. The two ends of a Type II degeneration
are rational surfaces with a smooth elliptic anticanonical double curve, and the
intermediate components are elliptic ruled surfaces with
double locus an anticanonical disjoint union of two elliptic sections.

\begin{definition}
An {\it anticanonical pair} or simply
{\it pair} $(V,D)$ is a smooth surface $V$ with a reduced, at worst
nodal, anticanonical divisor $D\in |-K_V|$.
A {\it toric pair} $(\oV,\oD)$ is a smooth toric
surface $\oV$ with $\oD\in |-K_{\overline{V}}|$ the toric boundary.
\end{definition}

The topologically trivial deformations of $X_0$ consist of
deforming the moduli of the pairs $(V_i,D_i)$, and regluing the
double curves by an element of $\C^*$ (in Type III) or
by a translation of the elliptic double curve $E$ (in Type II).
Only some of these regluings are smoothable by Theorem \ref{dss:def}.

\begin{definition}\label{def:charge}
The {\it charge} of an anticanonical pair $(V,D)$ is $\chi_{\rm top}(V\setminus D)$. If
$D=\sum D_j$, $$Q(V,D):=\threepartdef{12+ \sum (-3-D_j^2)}{D\textrm{ is nodal with at least two components,}}{11-D^2}{D\textrm{ is nodal and irreducible,}}{12\chi(\cO_V)-D^2}{D\textrm{ is smooth.}}$$
\end{definition}

\begin{proposition}[Conservation of Charge, {\cite[Prop. 3.7]{friedman1983smoothing-cusp}}]\label{conserve-charge}
Let $X=\bigcup\, (V_i,D_i)$ be a Kulikov surface. Then $\sum Q(V_i,D_i)=24.$
\end{proposition}

\begin{definition}\label{corner-and-internal}
A {\it corner blow-up} of $(V,D)$ is the blow-up at
a node of $D$. An {\it internal blow-up} is the blow-up at a
smooth point of $D$.
\end{definition}

 Both the corner and internal blow-ups
$\widetilde{V}\to V$ are naturally anticanonical pairs
$(\widetilde{V},\widetilde{D})$.
For a corner blow-up, $\widetilde{D}$
is the reduced inverse image of $D$. For an internal blow-up,
$\widetilde{D}$ is the strict transform of $D$. 
The formula for charge easily implies $Q(\widetilde{V},\widetilde{D})=Q(V,D)$
for a corner blowup, while $Q(\widetilde{V},\widetilde{D})=Q(V,D)+1$
for an internal blow-up.

Any toric pair satisfies $Q(\oV,\oD)=0$. 
When $V$ is rational and $D$ has nodes, as is the case
for any component in Type III, \cite[Prop.~1.3]{gross2015mirror-symmetry-for-log}
proves the existence of a diagram $$(V,D)\xleftarrow{f}
 (\widetilde{V},\widetilde{D})\xrightarrow{g} (\oV,\oD)$$ where $f$
 is a sequence of corner blow-ups, $g$ is a sequence of internal
 blow-downs, and $(\oV,\oD)$ is a toric pair. We call this
data a {\it toric model} of $(V,D)$.
By the existence of a toric model, 
$Q(V,D)\geq 0$ for all $(V,D)$ in Type III,
with $Q(V,D)=0$ if and only if $(V,D)$ is toric. So the
conservation-of-charge formula (\ref{conserve-charge})
says that $X_0$ is ``24 steps from being toric."

When $(V,D)$ is an elliptic ruled component of a Type II Kulikov surface,
we have $Q(V,D)=0$ if and only if
$V\cong \mathbb{P}_E(\mathcal{O}\oplus \mathcal{L})$
with $D$ the disjoint union of the zero and infinity sections.
Otherwise $Q(V,D)$ measures the number of steps from
being a smooth $\mathbb{P}^1$-bundle over an elliptic curve $E$.

Finally, we discuss base change, following
\cite{friedman1983base-change}.  Consider an order $k$ base change
$X'\to X$ of a Kulikov model along a branched cover $(C',0)\to (C,0)$.
Let $t$ be an analytic coordinate on $(C,0)$. The smoothing of
$X_0$ is locally $xy=t$ or $xyz=t$
near a double curve or triple point, respectively.
So the base change is locally $xy=t^k$ or $xyz=t^k$.
There is a locally toric, SNC resolution $X[k]\to X'$
near the singular locus of $X_0$ corresponding to the standard
order $k$ subdivision of the simplices of the dual
complex $\Gamma(X_0)$.
Each triangle decomposes into $k^2$ triangles, and each segment
into $k$ subsegments. All components of $X_0[k]$ not appearing
in $X_0$ satisfy $Q=0$.

\subsection{Nef, divisor, and stable models}
\label{sec:models}
We now describe some additional structures on a Kulikov model in the
presence of a quasipolarization.

\begin{definition}
  Let $L^*$ be a line bundle on $X^*$, relatively nef and big over $C^*$.
  A relatively nef extension $L$ to a Kulikov model $X$ over $C$
  is called a {\it nef model}.
\end{definition}

\begin{definition}
  Let $R^*\subset X^*$ be the vanishing locus of a
  section of $L^*$ as above, containing no vertical
  components. A {\it divisor model} is an extension $R\subset X$ to a relatively
  nef divisor $R\in |L|$ for which $R_0$ contains no strata of $X_0$.
\end{definition}

\begin{definition}
  The {\it stable model} of $(X^*,R^*)$ is
  $$(\overline{X},\overline{R}):=\textrm{Proj}_C\textstyle\bigoplus_{n\geq
    0} H^0(X, \mathcal{O}(nR))$$ for some divisor model. It is unique and
  independent of the choice of divisor model $(X,R)$ 
  by the theory of canonical models, since for $0<\epsilon\ll 1$
  the pair $(\oX,\oX_0+ \epsilon \oR)$ has log canonical singularities and the
  divisor $K_\oX + \epsilon \oR$ is relatively ample.

  By adjunction, the central fiber $(\oX_0,\epsilon \oR_0)$ has semi
  log canonical (slc) singularities and the divisor $K_{\oX_0} +
  \epsilon\oR_0$ is ample.
\end{definition}

The existence of a nef model is due to Shepherd-Barron
\cite{shepherd-barron1981extending-polarizations},
and the existence of a divisor model is proved in
\cite[Thm. 2.11]{laza2016ksba-compactification},
\cite[Thm. 3.12]{alexeev2019stable-pair}.

Now suppose one starts with a family $(\oX^*, \oR^*)\to C^*$ of K3
surfaces with ADE singularities. After a finite base change it admits
a simultaneous resolution of singularities $f\colon X^*\to\oX^*$. Let
$R^*=f^*(\oR^*)$. After a further finite base change, by the above we get a
divisor model, whose stable model $(\oX,\oR)$ is the {\it stable extension} 
of $(\oX^*,\oR^*)$ over $C$. It is unique and stable under base changes
by a standard argument, see e.g. \cite[Thm. 2.47]{kollar2023families-of-varieties}.

\subsection{Topology of Kulikov models}
The primary reference for this section is \cite{friedman1986type-III}.
Let $X\to (C,0)$ be a Kulikov model.
For convenience, denote integral
singular cohomology by $H^i(-)$.
Let $T:H^2(X_t)\to H^2(X_t)$ be the Picard-Lefschetz
transformation along an oriented simple loop in
$C^*$ enclosing $0$. Since $X_0$ is reduced normal crossings,
$T$ is unipotent. Let $N:=\log T$
be its logarithm. 

\begin{theorem}[{\cite{friedman1986type-III,friedman1984a-new-proof}}]\label{thm:pic-lef}
Let $X\to (C,0)$ be a Kulikov model. We have that
 \begin{enumerate} 
\item[] if $X$ is Type I, then $N=0$,
\item[] if $X$ is Type II, then $N^2=0$ but $N\neq 0$,
\item[] if $X$ is Type III, then $N^3=0$ but $N^2\neq 0$.
\end{enumerate} 
Furthermore, $N$ is integral, and of the form
$Nx = (x\cdot \lambda)\delta-(x\cdot \delta)\lambda$ for $\delta\in H^2(X_t)$
a primitive isotropic vector, and $\lambda\in \delta^\perp/\delta$ satisfying
$\lambda^2=\#\{\textrm{triple points of }X_0\}.$ When $\lambda^2=0$,
its imprimitivity is the number of double curves of $X_0$.
\end{theorem}

Thus, the Types I, II, III of Kulikov model are distinguished by the
behavior of the {\it monodromy invariant} $\lambda$: either
$\lambda=0$, $\lambda^2=0$ but $\lambda\neq 0$, or $\lambda^2\neq 0$ respectively.

\begin{remark}
If $X^*\to C^*$ admits a quasipolarization
$M\hookrightarrow {\rm Pic}(X^*)$ then $T\in O(H^2(X_t))$
lies in the subgroup $\Gamma$ fixing $M$.
In particular, $\delta\in M^\perp$ and we can consider the lattice
of monodromy invariants $\lambda \in \delta^\perp/\delta$ as valued
in a subquotient of $M^\perp$.
\end{remark}

\begin{definition}
Let $I\subset H^2(X_t)$ denote the primitive
isotropic lattice $\Z\delta$ in Type III or the saturation of
$\Z\delta\oplus \Z\lambda$ in Type II. 
\end{definition}

As a simple normal crossings degeneration,
there is a deformation-retraction $c:X\times [0,1]\to  X_0$ called the
{\it Clemens collapse} \cite{clemens1969picard}.
So we have $H^*(X_0)=H^*(X)$. In particular, the map
$c_t^*\colon H^*(X_0)\to H^*(X_t)$ coincides with restriction from $X$ to $X_t$.

The integral cohomology of a Type III Kulikov surface $X_0$ is computed 
in \cite[Sec.~1]{friedman1986type-III} by the Mayer-Vietoris spectral sequence,
associated to the exact sequence of sheaves
$$\textstyle 0\to \underline{\Z}_{X_0}\to \bigoplus_i \underline{\Z}_{V_i}\to
\bigoplus_{i\leq j}\underline{\Z}_{D_{ij}}\to \bigoplus_{i\leq j\leq k}\underline{\Z}_{T_{ijk}}\to 0.$$ It follows that there is an exact sequence
$$0\to \Z\to H^2(X_0)\to \widetilde{\Lambda}\to 0,\textrm{ where}$$

\begin{definition}\label{def:wLambda}
The {\it numerically Cartier classes} on a
Kulikov model $X_0$
$$\textstyle\widetilde{\Lambda}=\widetilde{\Lambda}(X_0):=
\ker\left(\bigoplus_i H^2(V_i)\to \bigoplus_{i\leq j} H^2(D_{ij})\right)$$
are collections of classes
$(\alpha_i)$ for which $n_{ij}:=\alpha_i\cdot D_{ij} = \alpha_j\cdot D_{ji}$ for all double curves.
\end{definition}

The lefthand term $\Z$ arises in the spectral sequence
from the second simplicial cohomology $H^2(\Gamma(X_0))$ of the
dual complex. Choosing an orientation on $\Gamma(X_0)$ gives a
generator $1\in \Z$ which satisfies $c_t^*(1)=\delta$. 

Mayer-Vietoris for a Type II Kulikov surface $X_0$
implies that there is an analogous exact sequence
$0\to \Z^2\to H^2(X_0)\to \widetilde{\Lambda}(X_0)\to 0$ with the
$\Z^2$ arising from $H^1(D_{i,i+1})$ for some double curve $D_{i,i+1}$.
Here the image $c_t^*(\Z^2)$ is identified with the rank two lattice $I$.

\begin{definition}\label{def:i0} Let $I_0$ denote the sublattice
$\Z\cong H^2(\Gamma(X_0))\subset H^2(X_0)$
in Type III or $\Z^2\cong H^1(D_{i,i+1})\subset H^2(X_0)$ in Type II arising
from Mayer-Vietoris.
 \end{definition}
 
So for both Type II and III,
$c_t^*(I_0)=I$ and $H^2(X_0)/I_0=\widetilde{\Lambda}$.

\begin{definition}
Define the intersection form $\cdot$ on
$\widetilde{\Lambda}$ by $(\alpha_i)\cdot (\beta_i) = \sum_i \alpha_i\cdot \beta_i.$
\end{definition}

\begin{definition}\label{xi-def}
Define $\hat{\xi}_i:=c_1(\mathcal{O}_X(V_i))\big{|}_{X_0}\in H^2(X_0)$
and let $\xi_i\in\widetilde{\Lambda}$ be the image of $\hat{\xi}_i$.
Then $\xi_i=\sum_j (D_{ji}-D_{ij})$ and $\sum_i \xi_i=0$.
Define $\textstyle \Xi := \Z\textrm{-span}\{\xi_i\}\subset 
\widetilde{\Lambda}$
and declare $\Lambda:=\widetilde{\Lambda}/\Xi.$
\end{definition}

 It is easy to check directly from property (2) of Proposition
 \ref{prop:kulikov} that $\Xi\subset \widetilde{\Lambda}$ is contained
 in the null sublattice of the intersection form.

\begin{proposition} \label{prop:lambda-lattice}
The map $c_t^*\colon H^2(X_0)\to H^2(X_t)$ induces a surjection
$\widetilde{\Lambda}\twoheadrightarrow \{\delta,\lambda\}^\perp/I$ sending
$\Xi$ to zero, which thus descends to $\Lambda$.
 Furthermore, $\Xi=\widetilde{\Lambda}\!\,^{\rm null}$
is the null sublattice. Hence $\Xi$ is saturated, $\Lambda$ is torsion-free
and the induced map $\Lambda\to \{\delta,\lambda\}^\perp/I$
is an isometry of lattices.
 \end{proposition} 

\begin{proof}
 \cite[4.13]{friedman1986type-III} gives an exact sequence
$$0\to \hat{\Xi}\to H^2(X_0)\xrightarrow{c_t^*} \ker(N)=\{\delta,\lambda\}^\perp\to 0$$
where $\hat{\Xi}:=\Z\textrm{-span}\{\hat{\xi}_i\}$.
Noting that $c_t^*(I_0)=I$,
we can quotient the second and third factors in the above exact sequence
to get an exact sequence
$$0\to \Xi\to \widetilde{\Lambda}\to \{\delta,\lambda\}^\perp/I\to 0.$$
Since the third term is torsion-free,
the kernel $\Xi$ must be the saturated. It is exactly the null lattice
because the target $\{\delta,\lambda\}^\perp/I$ is nondegenerate
and $c_t^*$ preserves the intersection form.
\end{proof}

\section{The period map}
\label{sec:period-map}

\subsection{The period of a Kulikov surface}\label{sec:periods-kulikov}
Let $X_0$ be
a Kulikov surface, not necessarily $d$-semistable.
The period map is a homomorphism $\widetilde{\psi}$
from $\widetilde{\Lambda}(X_0)$ (see Def. \ref{def:wLambda})
to $\C^*$ in Type III
or the elliptic double curve $E$ in Type II, which measures the
obstruction to a class being represented by a Cartier divisor. 
First, we consider the Type III case.

A resolution of the sheaf of non-vanishing holomorphic functions
is given by 
$$\textstyle 1\to \mathcal{O}^*_{X_0}\to
\bigoplus_i \mathcal{O}^*_{V_i}\to
\bigoplus_{i\leq j}\mathcal{O}^*_{D_{ij}}\to
\bigoplus_{i\leq j\leq k}\mathcal{O}^*_{T_{ijk}}\to 1.$$
Computing ${\rm Pic}(X_0) = H^1(X_0,\mathcal{O}_{X_0}^*)$
via the Mayer-Vietoris spectral sequence
\cite[Sec.~3]{friedman1986type-III} shows that ${\rm Pic}(X_0) $
is the kernel of a homomorphism
$$\ker\left(\textstyle \bigoplus_i {\rm Pic}(V_i)\to
\bigoplus_{i\leq j}{\rm Pic}(D_{ij})\right)\to
H^2(\Gamma(X_0),\C^*)\cong\C^*$$ where the latter
space is identified with $\C^*$ by choosing an orientation on the dual complex.
Note that since $V_i$ and $D_{ij}$ are rational, we have
${\rm Pic}(V_i)=H^2(V_i)$ and ${\rm Pic}(D_{ij})=H^2(D_{ij})$
so the first term is nothing more than the lattice
$\widetilde{\Lambda}$ of \eqref{def:wLambda}.

\begin{definition} The {\it period point} of a Type III Kulikov surface $X_0$
is $\widetilde{\psi}_{X_0}\in {\rm Hom}(\widetilde{\Lambda},\C^*)$.
\end{definition}

\begin{construction}\label{iii-unwind} Unwinding the maps in the
spectral sequence, one can explicitly construct
the homomorphism $\widetilde{\psi}_{X_0}$.
Let $\alpha=(\alpha_i)\in\widetilde{\Lambda}$ be a numerically Cartier divisor.
Then $\alpha_i$ determines a unique line bundle $L_i\in{\rm Pic}(V_i)$ for all $i$.
We have
$$L_i\big{|}_{D_{ij}} \cong L_j\big{|}_{D_{ji}}\cong \mathcal{O}_{\mathbb{P}^1}(n_{ij})$$
so we can extend a line bundle $L_i\to V_i$ by $L_j\to V_j$ to a
line bundle on $V_i\cup V_j$.
We may continue successively until only one component $V_1$ remains.
The result is a line bundle $L\to \overline{X_0\setminus V_1}$ and
we may consider the line bundle
$$L\big{|}_{D_1}\otimes L_1^{-1}\big{|}_{D_1}=:L_\alpha \in {\rm Pic}^0(D_1).$$
We have ${\rm Pic}^0(D_1)=\C^*$
because the cycle $D_1$ is oriented by the choice of
orientation on the dual complex $\Gamma(X_0)$. So $\alpha$ determines
a period $\widetilde{\psi}_{X_0}(\alpha)=L_\alpha\in \C^*$.
It is independent of the choice of component $V_1$ and clearly obstructs
$\alpha$ being represented by a Cartier divisor.\end{construction}

\begin{construction}\label{ii-unwind} In analogy to Construction \ref{iii-unwind},
we now construct a period map
$\widetilde{\psi}_{X_0}\colon \widetilde{\Lambda}\to E$ in Type II.
Orient the segment $\Gamma(X_0)$ so that
$X_0=V_0\cup\cdots \cup V_k$ with indices increasing with respect to
the orientation. Let $\alpha = (\alpha_i)\in \widetilde{\Lambda}$.
Then $\alpha_0\in H^2(V_0)$ and $\alpha_k\in H^2(V_k)$ define
line bundles $L_0$ and $L_k$ because the end surfaces are
rational elliptic. On the other hand, the lifts of an element
$\alpha_i\in H^2(V_i)$, $i\neq 0,k$ to an element $L_i\in {\rm Pic}(V_i)$
form a torsor over $E={\rm Pic}^0D_{i,i+1}$. So there is a unique lift
$L_1$ of $\alpha_1$ for which $L_0\big{|}_{D_{01}}\cong L_1\big{|}_{D_{10}}$.
Take this lift to extend $L_0$ to $V_0\cup V_1$ by $L_1$.
Continuing inductively gives a unique line bundle
$L\to \overline{X_0\setminus V_k}$. Then define
$$\widetilde{\psi}_{X_0}(\alpha):=L\big{|}_{D_{k-1,k}}\otimes L_k^{-1}\big{|}_{D_{k,k-1}}\in {\rm Pic}^0(D_{k-1,k})=E.$$
\end{construction}

The period map can also be defined from the exponential long exact sequence
\begin{align*} \cdots & \to H^1(X_0)\to H^1(X_0, \mathcal{O})\to
{\rm Pic}(X_0) \to H^2(X_0)\xrightarrow{\Psi} H^2(X_0,\mathcal{O})
\to H^2(X_0,\mathcal{O}^*)\to\cdots.\end{align*}

Note that $H^2(X_0,\mathcal{O}) = H^0(X_0,\omega_{X_0})^*\cong \C$
is one-dimensional. Quotienting by the image of $I_0\subset H^2(X_0)$,
we reproduce the period homomorphism
$$\widetilde{\psi}_{X_0} \colon \widetilde{\Lambda}\to \C/\Psi (I_0).$$ In Type III, we have
$\C/\Psi(I_0)\cong \C^*$ while in Type II we have $\C/\Psi(I_0)\cong E$
for an elliptic curve $E$. In both cases,
${\rm Pic}(X_0)$ is the kernel because $H^1(X_0,\mathcal{O})=0$.

\begin{proposition} The surface $X_0$ is smoothable if and
only if the period point
$\widetilde{\psi}_{X_0}\in {\rm Hom}(\widetilde{\Lambda},\C^*)$ or 
${\rm Hom}(\widetilde{\Lambda},E)$
descends to a period point
$\psi_{X_0}\in {\rm Hom}(\Lambda,\C^*)\textrm{ or }
{\rm Hom}(\Lambda,E).$
\end{proposition}

\begin{proof} By Theorem \ref{dss:def}, $X_0$ is smoothable if and
only if it is $d$-semistable. But $X_0$ is $d$-semistable if and only
if $\widetilde{\psi}_{X_0}(\xi_i)=1$ for all $i$, i.e. $\widetilde{\psi}_{X_0}$
descends to $\Lambda =\wt\Lambda/\Xi$. \end{proof}

\subsection{Markings of Kulikov surfaces}\label{sec:kulikov-marking}
 In this section, we define the analogues of
markings for Kulikov surfaces $X_0$ to properly
formulate results on the period map. 
Let $\Lambda_0=\Lambda_0(t,k)$
denote a model for $\{\delta, \lambda\}^\perp/I$
in $L_{K3}$. It depends only on the even integer
$\lambda^2=t$ giving the number of triple points of $X_0$ 
and the imprimitivity $k$ of $\lambda\in \delta^\perp/\delta$.
We suppress $(t,k)$ in the notation. 

\begin{definition} Let $X_0$ a $d$-semistable Kulikov surface.
A {\it marking} $(\sigma,\underline{b})$ consists of:
\begin{enumerate}
\item An isometry $\sigma\colon\Lambda(X_0)\to \Lambda_0$
(see Def.~\ref{def:wLambda}) and  
\item An ordered basis $\underline{b}$ of $I_0\subset H^2(X_0)$ (see Def.~\ref{def:i0}).
\end{enumerate}\end{definition}

The notion of a marking naturally extends to equisingular
families $\mathcal{X}\to S$ of Kulikov surfaces using local
systems. We can now define the period map:

\begin{definition}\label{def:kul-period}
Let $(\mathcal{X}\to S,\sigma)$
be a family of marked $d$-semistable Kulikov surfaces.
The {\it period map} is defined by
\begin{displaymath}
S \to {\rm Hom}(\Lambda_0, \C^*)\textrm{ or }
{\rm Hom}(\Lambda_0,\widetilde{\mathcal{E}}), \qquad
s \mapsto B\circ \Psi_s\circ \sigma^{-1}.
\end{displaymath}
Here $\Psi_s$ comes from the exponential exact sequence
as in Section \ref{sec:periods-kulikov},
$\widetilde{\mathcal{E}} \to \C\setminus \R$ is the
universal marked elliptic curve
$\C/\Z\oplus \Z\tau\to \{\tau\in \C\setminus \R\}$,
and $B$ is the quotient map
\begin{displaymath}
  B\colon H^2(X_0,\mathcal{O}_{X_0})\to \C/\Z= \C^*
  \quad\textrm{or}\quad
  B\colon H^2(X_0,\mathcal{O}_{X_0})\to \C/\Z\oplus \Z\tau
\end{displaymath}
induced by the ordered basis $\underline{b}$ of $I_0$
(the first element $b_1$ of the basis is sent to $1\in \C$).
\end{definition}

\begin{remark}\label{rem:tau-up}
In the Type II case, we
could also require that $\underline{b}$
is an oriented basis, in the sense that $\tau\in \mathbb{H}$.
Then the period map can be defined with target
$\widetilde{\mathcal{E}}^+:=\widetilde{\mathcal{E}} \big{|}_\mathbb{H}$
instead.
\end{remark}

\subsection{Partial markings of K3 surfaces}
Let $X\to (C,0)$ be a Kulikov model.
We determine what information a marking of
$X_0$ induces on the general fiber $X_t=Y$.
In this subsection, we denote an analytic K3 surface by
$Y$ to distinguish it from the Kulikov model $X$.

\begin{definition}
A
{\it partial marking} of $Y$ is a distinguished
primitive isotropic class $\delta\in H^2(Y)$, a distinguished vector
$\lambda\in \delta^\perp/\delta$ of non-negative norm, and an isometry
$\sigma\colon \{\delta,\lambda\}^\perp/I\to \Lambda_0$.
We say the partial marking is {\it Type II} or {\it III}
depending on whether $\lambda^2=0$ or $\lambda^2>0$, respectively.
\end{definition}

\begin{proposition}\label{partial-iii}
Let $X\to (C,0)$ be a Kulikov model. A partial marking
of $X_t$ whose distinguished classes $\delta$, $\lambda$
are the monodromy invariants
determines uniquely a marking of $X_0$.
\end{proposition}

\begin{proof}
Proposition \ref{prop:lambda-lattice} gives
an isometry $c_t^*: \Lambda(X_0)\to \{\delta,\lambda\}^\perp/I$.
So a partial marking of $X_t$ induces an isometry
$\Lambda(X_0)\to \Lambda_0$ by composing with $c_t^*$.
The class $\delta$ (and $\lambda$ in Type II) determines
a basis of $I_0\subset H^2(X_0)$ via $c_t^*$.
\end{proof}

\begin{definition}\label{exact-sequence}
The parabolic stabilizer of an isotropic lattice $I\subset L_{K3}$
fits into an exact sequence
$0\to U_{I}\to\textrm{Stab}_{O(L_{K3})}(I)\to \Gamma_{I}\to 0$
where $U_{I}$ is the {\it unipotent radical}: the normal subgroup acting
trivially on the graded pieces $I$ and $I^\perp/I$. In Type III, $U_I$ is isomorphic to
the additive group ${\rm Hom}(I^\perp/I,\,I)$. In Type II, $U_I$
is a central $\Z$-extension of ${\rm Hom}(I^\perp/I,\,I)$.
The quotient has the structure $\Gamma_{I}\cong O(I^\perp/I)\times GL(I)$.
These exact sequences play an important role in the toroidal
compactifications (Sec. \ref{sec:tor-comp}).
\end{definition}

\begin{definition}
A partially marked K3 surface $(Y,\sigma)$ is
{\it admissible} if $[\Omega]  \colon I\otimes \R\to \C$
sending $i\mapsto [\Omega]\cdot i$ is injective for any 
non-zero two-form $\Omega$ on $Y$. Similarly, define 
$$\mathbb{D}^{I}:=\{x\in \mathbb{D}\,\big{|}\, I\otimes \R\xrightarrow{\cdot x} \C\textrm{ is injective} \}.$$
\end{definition}

Note that $\mathbb{D}^I\subset \mathbb{D}$ is an open subset.
The period maps described in Sec.~\ref{sec:periods-kulikov}
can be understood as Carlson's extension class \cite{carlson1985one}
for the limit mixed Hodge structure of $X_0$ and $\mathbb{D}^I$ 
is the domain of Hodge structures on $L_{K3}$
for which $I$ happens to also define a mixed Hodge structure.

\begin{proposition}\label{partial-fine}
There is a fine
moduli space of admissible, partially marked, analytic K3
surfaces, admitting a  period map to $\mathbb{D}^{I}/U_{I}$.
\end{proposition}

\begin{proof} The partial markings of a K3 surface
$Y$ are identified with $U_I$-orbits of the set of
markings of $Y$. The fine moduli space
$\mathcal{M}$ of marked analytic K3 surfaces
admits a period map $P\colon \mathcal{M}\to \mathbb{D}$,
and if the partial marking associated to a
marked K3 surface $(Y,\sigma)$ is admissible, 
then its image under the period map lies in $\mathbb{D}^I$.
 The action of $U_I$ by post-composition on $\sigma$ is 
  free on $P^{-1}(\mathbb{D}^I)$ (as it is free on $\mathbb{D}^I$). 
  The quotient is a non-Hausdorff complex manifold. By the Torelli theorem,
  the universal family descends to a universal
  family of partially marked K3 surfaces. 
   \end{proof}

\begin{proposition}\label{partial-embedding} In Type III, there
is an open embedding
$\mathbb{D}^I/U_I\hookrightarrow I^\perp/I\otimes \C^*$
into a $20$-dimensional algebraic torus. In Type II,
there is an open embedding $\mathbb{D}^{I}/U_{I}\hookrightarrow A_{I}$
where $A_{I}\to I^\perp/I\otimes \widetilde{\mathcal{E}}$
is a punctured holomorphic disk bundle.  \end{proposition}

\begin{proof}[Sketch] Though the period domain $\mathbb{D}$ of analytic
K3s is not Hermitian symmetric, these embeddings are defined in
exactly the same way as the ``torus embeddings"
of the unipotent quotients of Hermitian symmetric domains 
\cite{ash1975smooth-compactifications}. In Type III,
one realizes $\bD^I$ as a tube domain inside $\C^{20}$.
The translation group $U_I={\rm Hom}(I^\perp/I,I)$ acts by translations by $\Z^{20}$
on $\C^{20}$ and so the quotient $\bD^I/U_I$ embeds into $(\C^*)^{20}$.

In Type II, $\bD^I$ is contained in an upper half-plane bundle,
fibered over the total space of a $\C^{18}$-bundle over $\C\setminus \R$.
The central $\Z$ acts on the upper-half plane bundle
by fiberwise translation. Quotienting gives a punctured holomorphic
disk bundle over the $\C^{18}$-bundle.
Then ${\rm Hom}(I^\perp/I,I)$ further acts on the $\C^{18}$-fiber
over $\tau\in \C\setminus \R$ by translation by
$(\Z\oplus \Z\tau)^{18}$. So the quotient $\bD^I/U_I$
embeds into a punctured holomorphic disk bundle
$A_I\to I^\perp/I\otimes \widetilde{\cE}$. \end{proof}

The unipotent quotient of $\mathbb{D}_M$ embeds
into $\mathbb{D}^I/U_I$ for all $M$. Let $\mathbb{D}(I):=\mathbb{D}^I/U_I$.

\begin{definition}
Define an enlargement $\mathbb{D}(I)\hookrightarrow \mathbb{D}(I)^{\lambda}$
as follows: In Type III, it is the closure of $\mathbb{D}(I)$ in the toric
variety $T_\lambda$ extending the torus $I^\perp/I\otimes \C^*$
whose fan consists of the unique ray $\R_{\geq 0}\lambda $.
In Type II, it is the holomorphic
disk bundle $\overline{A}_I$ extending the punctured disk bundle $A_I$.
\end{definition}

In Type III, the boundary divisor in $D(I)^\lambda$
is isomorphic to $\delta^\perp/\{\delta,\lambda\}^{\rm sat}\otimes \C^*$.
Since $\delta^\perp/\delta$ is unimodular, 
this torus can be identified with ${\rm Hom}(\Lambda,\C^*)$.
Similarly, the added boundary divisor in Type II is naturally
isomorphic to the base $I^\perp/I\otimes \widetilde{\mathcal{E}}$
of the disk bundle, which is identified with
${\rm Hom}(\Lambda, \,\widetilde{\mathcal{E}})$ again
because $I^\perp/I$ is unimodular.

\begin{definition}\label{mixed-type}
Let $\mathcal{X}\to S$ be a family of $d$-semistable Kulikov surfaces
of Types I + II or I + III over a smooth base $S$. Suppose furthermore
that the discriminant locus $\Delta\subset S$ is a smooth divisor.
A {\it mixed marking} $\sigma$ is a partial marking of the family
$\mathcal{X}^*\to S\setminus \Delta$ of smooth fibers together
with a compatible (Prop.~\ref{partial-iii}) marking of
the equisingular family $\mathcal{X}_0\to \Delta$ of Kulikov surfaces.
\end{definition}

\begin{theorem}\label{period-extension}
Let $(\mathcal{X}\to S,\sigma)$
be a mixed marked family of admissible surfaces as in Definition \ref{mixed-type}.
The period map $\psi \colon S\setminus \Delta \to \mathbb{D}(I)$ extends
to a morphism $$\overline{\psi}\colon S\to \mathbb{D}(I)^\lambda$$
sending the discriminant $\Delta$ to the boundary divisor
${\rm Hom}(\Lambda,\C^*)\textrm{ or }{\rm Hom}(\Lambda,\widetilde{\mathcal{E}})$.
Furthermore, $\overline{\psi}\big{|}_\Delta$ is the period map for
the family of marked Kulikov surfaces $\mathcal{X}_0\to \Delta$,
as in Definition \ref{def:kul-period}.
\end{theorem}

\begin{proof} This theorem is essentially the same as
\cite[Thm.~5.3]{friedman1986type-III}. The primary tool is the
nilpotent orbit theorem \cite[Thm.~4.9]{schmid1973variation}.
 \end{proof}

% \begin{definition} An {\it $M$-quasipolarization} of a $d$-semistable
% Kulikov surface $X_0$
% is an embedding $M\hookrightarrow \ker(\psi_{X_0})\subset \Lambda(X_0)$. \end{definition}

 For an $M$-quasipolarized $d$-semistable Kulikov surface $X_0$, we
 fix an embedding $M\hookrightarrow \ker(\psi_{X_0})\subset \Lambda(X_0)$.
In the $M$-quasipolarized case, the period point
${\rm Hom}(\Lambda, \,\C^*\textrm{ or }E)$ descends to
a period point in ${\rm Hom}(\Lambda/M,\,\C^*\textrm{ or }E)$
which we will also denote $\psi_{X_0}$ by abuse of notation.
More precisely, a primitive sublattice
$M\subset\Lambda\oplus I$ is the same as a not necessarily
primitive sublattice $M\subset\Lambda$ plus a homomorphism 
$\Psi\colon\Tors(\Lambda/M)\to \bC^*$ or $E$.
The period point 
belongs to the coset of $\Hom(\Lambda/M^{\rm sat},\bC^*\textrm{ or }E)$
of points with $\psi_{X_0}|_{\Tors(\Lambda/M)} = \Psi$.
The discussion of the period map in the above sections holds,
replacing everywhere $H^2(X_t)$ with $j(M)^\perp$,
$L_{K3}$ with $M^\perp$,
$\mathbb{D}^I$ with $\mathbb{D}^I_M:=\mathbb{D}^I\cap \mathbb{D}_M$,
$U_I$ with $U_I\cap \Gamma$, $I^\perp$ with
$I^{\perp}_{M^\perp}$, $\Lambda$ with $\Lambda/j_0(M)$,
$\Lambda_0$ with $\Lambda_0/M$. Recall that $\Gamma$ is the subgroup of
$O(L_{K3})$ acting by the identity on~$M$.

\begin{proposition}
$\mathbb{D}_M^I=\mathbb{D}_M$ for any $M$.
\end{proposition}

\begin{proof}
An $x\in \mathbb{D}_M\setminus \mathbb{D}_M^I$
would satisfy $x\cdot i=0$ for some nonzero isotropic vector $i\in I\otimes \R$.
But then ${\rm Re}(x)$, ${\rm Im}(x)$ span a positive
definite $2$-plane in $i^\perp_{M^\perp}\otimes \R$
which is impossible since $M^\perp$ has signature $(2,20-r)$.
\end{proof}

The moduli space of partially marked, $M$-quasipolarized K3
surfaces admits a period map to the torsion translate
of a subtorus $\mathbb{D}_M(I):=\mathbb{D}_M/U_I\cap
\Gamma\subset \mathbb{D}(I)$
which is generically one-to-one. A mixed marked $M$-quasipolarized family 
admits a period map to
the toroidal extension $$\mathbb{D}_M(I)^\lambda:=\overline{\mathbb{D}_M(I)}\subset \mathbb{D}(I)^\lambda.$$

\begin{notation}
We henceforth write $I^{\perp}_{M^\perp}$ simply
as $I^\perp$ whenever it is clear from context that we are
working with $M$-quasipolarized surfaces.
\end{notation}

\section{Compactifications of arithmetic quotients}
\label{sec:arithmetic-quotients}

\subsection{Baily-Borel compactification}\label{sec:bb}
By Theorem \ref{smooth-coarse},
the coarse space of $M$-polarized
(ADE) K3 surfaces $F_M$ is the quotient of the period domain
$$\mathbb{D}_M:=\mathbb{P}\{x\in M^\perp\otimes \C\,\big{|}\,
x\cdot x=0, x\cdot \overline{x}>0\}$$ by the arithmetic group $\Gamma$. 
In this capacity, the space $F_M=\mathbb{D}_M/\Gamma$ has a
Baily-Borel compactification \cite{baily1966compactification-of-arithmetic},
which we now describe.

\begin{remark} Note that $\mathbb{D}_M=\mathbb{D}_M^1 \sqcup \mathbb{D}_M^2$
has two connected components and so $F_M$ may have either $1$ or $2$ connected
components, depending on whether or not $\Gamma$ contains an element interchanging
the two components. To simplify (but abuse)
notation, we refer to $\mathbb{D}_M^1$ and its stabilizer
$\Gamma^1\subset \Gamma$ as $\mathbb{D}_M$ and $\Gamma$, respectively. \end{remark}

\begin{definition} The {\it compact dual} is $\mathbb{D}_M^c:=
\mathbb{P}\{x\in M^\perp\otimes \C\,\big{|}\,x\cdot x=0\}$. It is the 
compact hermitian symmetric domain containing $\mathbb{D}_M$
 as an open subdomain. \end{definition} 

\begin{definition}\label{boundary-comp} A {\it boundary component} 
of $\mathbb{D}_M$ is a maximal connected complex
submanifold of the boundary
 $\partial \mathbb{D}_M \subset \mathbb{D}_M^c$. The {\it rational}
  boundary components $B_I$ are in bijection
  with primitive isotropic lattices $I\subset M^\perp$ via
  $$B_I=\{x\in \partial \mathbb{D}_M\,\big{|}
  \,\textrm{span}\{\textrm{Re}\,x,\textrm{Im}\,x\}=I\otimes \R\}.$$ 
  We have $B_I\cong\mathbb{H}$ when ${\rm rk}\, I=2$.
  We have $B_I\cong\{\textrm{pt}\}$ when ${\rm rk}\, I=1$. We call these 
  {\it Type II and III} boundary components, respectively. The 
  {\it rational closure} of $\mathbb{D}_M$ is $\mathbb{D}_M^+:=
  \mathbb{D}_M\cup_I B_I\subset \mathbb{D}_M^c$
   topologized at the boundary 
   points using horoballs as a neighborhood base. \end{definition}

\begin{theorem}[\cite{baily1966compactification-of-arithmetic}] The
quotient $\oF_M^\bb:= \mathbb{D}_M^+/\Gamma$ is compact and
has the structure of a projective variety, and the projective coordinate
ring is the ring of modular forms for $\Gamma$.
\end{theorem}

The image of a boundary component $B_I$ in $\oF_M^\bb$
is isomorphic to $B_I/\textrm{Stab}_\Gamma(I)$ and so is either
a point when ${\rm rk}\,I=1$ or a modular curve when ${\rm rk}\,I=2$.

\begin{definition} The {\it $0$- and $1$-cusps} of $\oF_M^\bb$ are
the zero- and one-dimensional boundary components, respectively.
They are, respectively, in bijection with $\Gamma$-orbits
of rank $1$ and $2$ primitive
isotropic lattices $I\subset M^\perp$. \end{definition}

\begin{proposition} Let $X\to (C,0)$ be an $M$-quasipolarized
 Kulikov model. The extension of the period map $C^*\to F_M$ to 
 the Baily-Borel compactification sends $0$ into the cusp associated
  to the monodromy lattice $I$. In Type II, the $j$-invariant $j(D_{i,i+1})$ of a double curve 
  agrees with the $j$-invariant $j:B_I/{\rm Stab}_\Gamma(I)\to
   \mathbb{H}/SL_2(\Z)=\mathbb{A}^1_j$ of the corresponding 
   image point. \end{proposition}

\begin{proof} This well-known fact follows directly from the
asymptotics of the period map and the nilpotent orbit theorem, as in Theorem
  \ref{period-extension}. \end{proof}

\subsection{Toroidal compactification}\label{sec:tor-comp} The 
original source on this subject is \cite{ash1975smooth-compactifications}. 
The reference \cite{namikawa1980toroidal-compactification} in
 the case of Siegel space $\mathbb{D}=\mathcal{H}_g$ is particularly
  clear. The following well-known theorem is key to constructing 
  toroidal compactifications:

\begin{theorem}\label{stab-cusp} Let $B_I$ be a rational boundary
 component of $\mathbb{D}_M$. There exists a horoball neighborhood
  $\overline{N}_I\supset B_I$ preserved by ${\rm Stab}_\Gamma(I)$ 
  and an embedding $$N_I/{\rm Stab}_\Gamma(I)\hookrightarrow 
  \mathbb{D}_M/\Gamma\textrm{ where }N_I=\overline{N}_I\setminus B_I.$$
   \end{theorem}

So a punctured neighborhood 
of a Baily-Borel cusp can be constructed locally as a quotient by 
the parabolic stabilizer ${\rm Stab}_\Gamma(I)$. Let
$0\to U_I\to {\rm Stab}_\Gamma(I)\to \Gamma_I\to 0$ be the exact
sequence associated to the unipotent radical $U_I$ (\ref{exact-sequence}).
Then $N_I/U_I\hookrightarrow \mathbb{D}_M(I)=\mathbb{D}_M/U_I$
has an open embedding into the unipotent quotient.
The Levi group $\Gamma_I$ has a residual action on both.

\begin{definition}\label{def:c-plus} Let $I=\Z\delta$ be a rank $1$ isotropic lattice. 
Let $C_\delta\subset \delta^\perp/\delta\otimes \R$ denote a connected component of 
the positive norm vectors and let $C_\delta^+$ be its {\it rational closure}: 
the union of $C_\delta$ with all rational rays on its boundary. \end{definition}

\begin{definition}\label{def:fan} A {\it $\Gamma$-admissible collection of fans} 
$\mathfrak{F}$ (or for short, {\it fan}) is, for each $I=\Z\delta$, a fan 
$\mathfrak{F}_\delta$ with support $C_\delta^+$, such that the 
collection $\{\mathfrak{F}_\delta\}$ is $\Gamma$-invariant, with
 finitely many orbits of cones. \end{definition}

By ``fan" $\mathfrak{F}_\delta$ we mean a  decomposition into 
rational polyhedral cones, closed under taking faces 
and intersections, and locally finite in the positive cone~$C_\delta$. 
Infinitely many cones meet at rational rays on the boundary of
 $C_\delta^+$. Recall, when ${\rm rk}\,I= 1$, there is a ``torus embedding"
 $\bD_M(I)\hookrightarrow \delta^\perp/\delta\otimes \C^*$ (\ref{partial-embedding}).

\begin{construction}\label{tor-construction}
The {\it toroidal compactification} $\oF_M^\mathfrak{F}$
 associated to a fan $\mathfrak{F}$ is built as follows: Take
 the closure $N_I/U_I\hookrightarrow \overline{N_I/U_I}\subset
 X(\mathfrak{F}_\delta)$ in the toric variety containing
 $\delta^\perp/\delta\otimes \C^*$ associated to the fan $\mathfrak{F}_\delta$.
   Then quotient by $\Gamma_I$ to get an analytic space 
   $V_I:=(\overline{N_I/U_I})/\Gamma_I.$ This is possible by 
   $\Gamma$-invariance of $\fF$. Note that $V_I$ contains an
    open subset $$(N_I/U_I)/\Gamma_I= N_I/{\rm Stab}_\Gamma(I)
    \hookrightarrow \mathbb{D}_M/\Gamma.$$ Define the {\it Type III extension} 
    to be the gluing of $F_M=\mathbb{D}_M/\Gamma$ to $V_I$ along this
     open set, ranging over all $\Gamma$-orbits of rank $1$ isotropic $I$. 

If $I=\Z\delta\oplus \Z\lambda$ is isotropic of rank $2$,
take the closure $\overline{N_I/U_I}\subset \mathbb{D}_M(I)^\lambda$
  in the projective line bundle over $I^\perp/I\otimes \widetilde{\mathcal{E}}^+$
(\ref{partial-embedding})
and define $V_I$ as above.
The {\it Type II extension} is the gluing
    of $F_M$ with $V_I$ along their common open subset 
    $(N_I/U_I)/\Gamma_I = N_I/{\rm Stab}_\Gamma(I)$. 

The toroidal compactification $\oF_M^\mathfrak{F}:=F_M\cup_I V_I $
 is the gluing of the Type II and III extensions.

  \end{construction}
  
Let $\delta\in I=\Z\delta\oplus \Z\lambda$.
The analytic structure where the corresponding Type III 
 and II loci meet is described by the Mumford construction
\cite{mumford1972an-analytic-construction}
applied to a periodic, rational polyhedral tiling
 $\mathfrak{F}_{\delta,I}$ of $I^\perp/I\otimes \R $. The polyhedral 
 tiles are defined as follows: Quotient the cones of $\mathfrak{F}_\delta$ 
 passing through $\R^+\lambda\subset C_\delta^+$, viewing
 $I^\perp/I$ as the subquotient $\lambda^\perp/\lambda$ of $\delta^\perp/\delta$.
  Geometrically, $I^\perp/I\otimes \R$ is identified with a small horosphere 
  through $\lambda$ (minus $\lambda$) in the hyperbolic space 
  $\mathbb{P} C_\delta$. The projectivized cones of $\mathfrak{F}_\delta$
   decompose this horosphere in a 
   ${\rm Stab}_{\Gamma_\delta}(\lambda)$-invariant manner.

%\begin{remark} Note that $\mathbb{D}_M=\mathbb{D}_M^1\cup \mathbb{D}_M^2$
% has two connected components, and each component distinguishes a 
% connected component of $C_\delta=C_\delta^1\cup C_\delta^2$ for all 
% $\delta$. ($C_{-\delta}^1=C_\delta^2$ and $C_{-\delta}^2=C_\delta^1$ by
% Remark \ref{double-type-iii}). 
% When $F_M=F_M^1\cup F_M^2$ has two connected components,
%  any $\gamma\in \Gamma$ acts by sending 
%  $\gamma:C_\delta^i\to C_{\gamma(\delta)}^i$ for both $i=1,2$. So the 
%  $\Gamma$-orbits of components of $C_\delta$ split naturally into two
%   disjoint collections of orbits, and the fans in these collections of orbits 
%   are unrelated to each other.  This corresponds to the fact that the 
%   toroidal compactifications of the two connected components need 
%   not be related to each other. 
%
%On the other hand, when $F_M$ is irreducible, the $\Gamma$-orbit
%of a connected component of $C_{\delta'}$ is represented by
%$C_\delta^1$ for some $\delta\in \Gamma\cdot \delta'$.
%Hence, the data of a $\Gamma$-admissible collection
%of fans is determined by a collection of fans with support
%$(C_\delta^1)^+$ which is preserved by the subgroup of
%$\Gamma$ not permuting the components of $\mathbb{D}_M$.
%\end{remark}

\subsection{Semitoroidal compactification}\label{sec:semitor-comp}

The papers
\cite{looijenga1985semi-toric, looijenga2003compactifications-defined2}
are the only references for this section.
Semitoroidal compactifications are determined combinatorially,
unify the toroidal and Baily-Borel compactifications, and
form the smallest class of compactifications
closed under taking normal images of toroidal compactifications
(proven in Theorem \ref{thm:semi-is-normal} below).
 
The combinatorial input is similar
to toroidal compactifications, with two differences:

\begin{definition}[{\cite[Def.~6.1]{looijenga2003compactifications-defined2}}]\label{def:semifan}
A {\it semifan} $\mathfrak{F}$
requires the same data as a fan (Def.~\ref{def:fan}), but we allow the
cones in $\mathfrak{F}_\delta$ to be only {\it locally polyhedral} in
$C_\delta$.

We additionally require ``compatibility"
 at each $1$-cusp: Let $\delta\in I$ be a primitive integral
vector in a rank $2$ isotropic lattice and let $\mathfrak{F}_{\delta,I}$
denote the corresponding polyhedral tiling of $I^\perp/I\otimes \R$.
The tiles of $\mathfrak{F}_{\delta,I}$
are of the form $B\times (H_{I,\delta}\otimes \R)$ for bounded polytopes $B$ and
$H_{I,\delta} \subset I^\perp/I$ a primitive sublattice. We require that
$H_{I,\delta}=H_I$ is independent of choice of $\delta$. \end{definition}

\begin{example} Any fan is a semifan. The tiles of
$\mathfrak{F}_{\delta,I}$ are bounded polytopes, so $H_I=\{0\}$ for all $I$.
At the other extreme, the semifan $\mathfrak{F}$ for which
$\mathfrak{F}_\delta=\{C_\delta^+\}$ is locally finite
and the compatibility condition holds: $H_{I,\delta} = I^\perp/I$ for all $\delta\in I$.
The resulting compactification is $\oF_M^\mathfrak{F} = \oF_M^\bb$.
\end{example}

We now compile the key results we need about semitoroidal compactifications:

\begin{theorem}
There is a normal compactification
$\oF_M^\mathfrak{F}$ whose boundary strata are in bijection with
$\Gamma$-orbits of cones of $\mathfrak{F}$ \cite[Thm.~6.7]{looijenga2003compactifications-defined2}. 
The stratum ${\rm Str}_\sigma$ corresponding to a cone
$\sigma\subset \mathfrak{F}_\delta$ is
finite quotient of $\delta^\perp/\{\delta,\sigma\}\otimes \C^*$
in Type III, and a finite quotient of
$I^\perp/\{I,H_I\}\otimes \mathcal{E}$ in Type II \cite[p.~552]{looijenga2003compactifications-defined2}.
For any semifan
$\mathfrak{G}$ which refines $\mathfrak{F}$, there is a morphism
$$\oF_M^\mathfrak{G}\to \oF_M^\mathfrak{F}$$ mapping strata to
strata \cite[Lem.~6.6]{looijenga2003compactifications-defined2}.
Given an inclusion of cones $\sigma_\mathfrak{G} \subset \sigma_\mathfrak{F}$
 the map of corresponding strata is induced by the natural
  quotient map on tori.
\end{theorem}

Unlike for fans, a Type III cone $\sigma$ of a
semifan may have an infinite stabilizer ${\rm Stab}_{\Gamma_\delta}{\sigma}$.
Still, the corresponding stratum is a finite quotient of a torus.

A simple way to visualize a semifan or fan $\mathfrak{F}$ is as follows:
For each $\Gamma$-orbit of isotropic vector $\delta$, associate a cusped,
real-hyperbolic orbifold $M_\delta:=\mathbb{H}^{19-{\rm rk}\,M}/\Gamma_\delta$ where
$\mathbb{H}^{19-{\rm rk}\,M}=\mathbb{P}C_\delta$ is
real-hyperbolic space. The cusps of $M_\delta$ correspond to
$\Gamma_\delta$-orbits of
isotropic rays in $C_\delta^+$.
A semifan $\mathfrak{F}_\delta$ gives rise to a
finite decomposition of $M_\delta$ (for all $\delta$)
into metrically convex, rational, locally polyhedral cells, compatible
with the hyperbolic cusps. For a fan, these cells are polyhedra, while for a
semifan, they may have nontrivial topology.

We now prove a key theorem characterizing semitoroidal
compactifications:

\begin{theorem}\label{thm:semi-is-normal} Let $F$ be a
Type IV arithmetic quotient and let $\oF$ be a normal 
compactification of $F$. The following are equivalent:
 \begin{enumerate}
 \item $\oF$ sits between some toroidal and the Baily-Borel
 compactification: $\oF^\mathfrak{G}\xrightarrow{m} \oF\to \oF^\bb.$
 \item There exists a semifan $\mathfrak{F}$ for which $\oF = \oF^\mathfrak{F}$.
 \end{enumerate}
\end{theorem}

\begin{proof}
  The implication (2) $\implies$ (1) follows from
refining $\mathfrak{F}$ to some fan $\mathfrak{G}$.

Now we prove (1) $\implies$ (2). Define an equivalence relation
$\sigma_1\sim \sigma_2$ on maximal cones of $\mathfrak{G}$
generated by: $\sigma_1\sim \sigma_2$ if $\sigma_1$ and $\sigma_2$
share a codimension one face $\tau$ such that the corresponding
$1$-dimensional boundary stratum $\textrm{Str}_\tau$ is contracted
by $m$. Our strategy is to show that the curves contracted by any birational morphism
$m\colon \oF^\mathfrak{G}\to \oF$ over $\oF^\bb$ are algebraically
equivalent to a union of $1$-dimensional torus orbits. So
the ``toroidally definable" equivalence relation $\sim$
captures everything one needs to know about the
contracting morphism $m$.

Define a decomposition of $C_\delta^+$ into a collection
of maximal dimensional sets
$$[\sigma_0]:=
\textstyle\bigcup_{\sigma\sim \sigma_0} \sigma.$$ We claim that
the $[\sigma_0]$ form the maximal cones of some semifan $\mathfrak{F}$.
The $\Gamma$-invariance is automatic, so it
 suffices to show that $[\sigma_0]$ satisfy the semifan axioms, 
 including the compatibility condition (Def. \ref{def:semifan})
over the $1$-cusps.   

Begin with a Type III cone $\tau\in \mathfrak{G}_\delta$. The stratum $\overline{{\rm Str}}_\tau\subset \oF^{\mathfrak{G}}$
 is the ${\rm Stab}_{\Gamma_\delta}(\tau)$-quotient of the toric variety
 $X(\mathfrak{G}_\delta/\tau)$ associated to the quotient fan.
Consider the Stein factorization 
$X(\mathfrak{G}_\delta/\tau)\to Z_\tau\to m(\overline{\Str}_\tau)$.
It is proved in \cite[Lem.~2.3.4]{brown2018geometric_characterization} that the
 target of a morphism from a proper toric variety to a
 normal variety ($Z_\tau$ here)
 is automatically toric and with the morphism also toric. 
 
 Since the maps $Z_\tau \to m(\overline{{\rm Str}}_\tau)$ and
 $X(\mathfrak{G}_\delta/\tau)\to \overline{{\rm Str}}_\tau$ are finite,
the curves contracted by $\overline{{\rm Str}}_\tau\to m(\overline{{\rm Str}}_\tau)$
are exactly those that lift to curves contracted by $X(\mathfrak{G}_\delta/\tau)\to Z_\tau$. 
 So the equivalence relation $\sim$ on the maximal cones
 containing $\tau$ is induced by the morphism of fans corresponding to the toric morphism
 $X(\mathfrak{G}_\delta/\tau)\to Z_\tau$. Thus the
 cones $[\sigma_0]$ locally form a fan in a tubular neighborhood of
 $\tau\subset C_\delta^+$. In particular, they
 are locally polyhedral and convex at their boundary.
 So the $[\sigma_0]$ define a semifan within $C_\delta$.

Next, we examine the Type II locus. Since $\oF$ has a morphism to
$\oF^\bb$, we conclude that $m$ induces a fiberwise morphism over
the modular curve $1$-cusp. Let $j$ be a point in the
$1$-cusp of $\oF^\bb$. The fiber over $j$ in any toroidal
compactification is a finite quotient of
$I^\perp/I\otimes E_j$ so there is a morphism 
$I^\perp/I\otimes E_j\to \oF_j$
induced by $m$. In analogy with the
Type III case, take the Stein factorization of this morphism
$I^\perp/I\otimes E_j\to Z_j\to \oF_j$. Since the normal image
of an abelian variety is an abelian variety,
this map is the quotient by a sub-abelian variety $H_I\otimes E_j$.
  
  The contracted curves are generated, up to algebraic equivalence,
  by $h\otimes E_j$ for $h\in H_I$. Taking the limit $$j\to \textrm{the }
0\textrm{-cusp of }\oF^\bb\textrm{ associated to }\delta,$$ the elliptic curve $h\otimes E_j$
breaks in the Type III locus 
to a cycle of rational curves, according to the Mumford degeneration
discussed after Construction \ref{tor-construction}.

Applying the torus action to the cycle of rational curves
$\lim_jh\otimes E_j$ we can break it further into a cycle of contracted
$1$-dimensional
boundary strata connecting $0$-dimensional boundary strata of Type III.
So the equivalence relation $\sim$ on maximal cones
induces a polyhedral decomposition of $I^\perp/I\otimes \R$ whose
tiles are fixed under translation by $h$, and in turn by $H_I$. Conversely,
consider an $h\in I^\perp/I$ fixing all tiles in the polyhedral
decomposition $H_{I,\delta}$ of $I^\perp/I$ induced by the cones $[\sigma_0]$. 
This $h$ corresponds to a contracted cycle of rational curves, which deforms
to a contracted elliptic curve $h\otimes E_j$. Hence $h\in H_I$.
Thus $H_I = H_{I,\delta}$ is independent of $\delta$.
So there exists a semifan $\mathfrak{F}$ whose maximal
cones are $[\sigma_0]$.

Since $\mathfrak{F}$ is a coarsening
of $\mathfrak{G}$, there is a morphism
$\oF^\mathfrak{G}\xrightarrow{n} \oF^\mathfrak{F}$ and the above
arguments prove that the curves contracted by $n$ are exactly those
contracted by $m$. We conclude by Zariski's main theorem
 that $\oF=\oF^\mathfrak{F}$, because $\oF$ is normal.
\end{proof}

\section{Moduli of stable slc pairs}
\label{sec:slc-pairs}

\subsection{Canonical choices of polarizing divisor}

Let $\cF_M^{\rm q}$ be the moduli stack
of $M$-quasipolarized K3 surfaces.
 Fix a class $L\in M$, not necessarily primitive, which defines
 a relatively big and nef line bundle $\mathcal{L}\to \cX\to \cF_M^{\rm q}$
 on the universal family, canonical up to twisting
  by line bundles pulled back from $\cF_M^{\rm q}$. 
 Since $\mathcal{L}$ is big and nef on every fiber,
 $h^i(\mathcal{X}_s, \mathcal{L}_s)=0$ for $i>0$
 for all $s\in\mathcal{F}_M^{\rm q}$. By Cohomology and Base Change,
  the pushforward of $\mathcal{L}$ from the universal family defines a
 vector bundle of rank $2+\frac{1}{2}L^2$
 on $\mathcal{F}_M^{\rm q}$, canonical up to twisting by line
 bundles, cf. \cite[Cor.~2.69]{kollar2023families-of-varieties}.
  Let $\mathbb{P}_\cL$ denote its projectivization, a $\mathbb{P}^g$-bundle
  over the stack, where $g=d+1$.

\begin{definition}\label{def:can-choice}
A {\it canonical choice of polarizing divisor} is
a rational section $R$ of the projective bundle $\mathbb{P}_\cL$. 
Alternatively, it is an ample divisor $R$ on the generic K3 surface.
\end{definition}

Let $U$ be the regular locus of this rational section.
The key definition of the paper is:

\begin{definition}\label{recog-def}
A canonical choice of
polarizing divisor $R$ is {\it recognizable for $F_M$} if every
$M$-quasipolarized Kulikov surface $X_0$ of Type I, II, or III
contains a divisor $R_0\subset X_0$ which, for any $M$-quasipolarized
smoothing $X\to (C,0)$ with $C^*\subset U$, has the property that $R_0$ is
the flat limit of $R_t\subset X_t$, $t\neq 0$,
up to the action of ${\rm Aut}^0(X_0)$.
\end{definition}

Here ${\rm Aut}^0(X_0)$ is the connected component of the identity
of the automorphism group, which is always trivial in Type III,
and is isomorphic to $(\C^*)^{k-1}$ where $k-1$ is the number
of intermediate elliptic ruled components, in Type II.

We use the term ``smoothing" to mean specifically a Kulikov model $X\to (C,0)$.
Roughly, Definition \ref{recog-def} amounts to saying that the
canonical choice $R$ can also be made on any Kulikov surface,
including smooth K3s, at least up to ${\rm Aut}^0(X_0)$.

\begin{proposition}\label{recog-section-type-i}
Let $(\mathcal{X}^*, \mathcal{R}^*)\to U$ be the
universal family of pairs. If $R$ is recognizable, it extends
to a flat family of pairs $(\mathcal{X}, \mathcal{R})\to \mathcal{F}_M^{\rm q}$.
That is, $R$ defines a regular section of
$\mathbb{P}_\mathcal{L}\to \cF^{\rm q}_M$. \end{proposition}

\begin{proof} Let $0\in \mathcal{F}_M^{\rm q}$ be in the complement of
$U$. Choose any curve $C\subset \mathcal{F}_M^{\rm q}$ containing
 $0$ for which $C^*=C\setminus 0\subset U$. Then $X\to C$ is a Type I Kulikov
  model, for which $X_0$ is a smooth $M$-quasipolarized K3 surface. 
  By assumption, there is a divisor $R_0\in |L|$ on $X_0$ which is the
  flat limit of the curves $R_t$ for $t\neq 0$. We may extend the
  section $\mathcal{R}^*\,:\,U\to \mathbb{P}_\mathcal{L}\big{|}_U$
  set-theoretically by declaring $\mathcal{R}(0)=R_0$. This 
  extension is algebraic when restricted to any curve in 
  $U\cup \{0\}$. Since $\cF_M^{\rm q}$ is normal, we conclude 
  that the rational section $\mathcal{R}^*$ extends over $0$.
\end{proof} 

This proposition only concerns Type I Kulikov models. 
The properties of recognizability in Types II and III is discussed in
Section \ref{sec:recognizable-divisors}. 

\subsection{Compact moduli of stable pairs}

We refer the reader to \cite{kollar2023families-of-varieties} for a
definitive account. 
An {\it slc} (or KSBA) {\it stable pair} $(X,B=\sum b_i B_i)$ consists of a projective variety
and a $\bQ$-divisor which has semi log canonical (slc) singularities
such that the divisor $K_X + B$ is ample. A particular case is a
log Calabi-Yau pair $(X, \Delta + \epsilon R)$ such that $\Delta$ is
reduced and log canonical, $0<\epsilon\ll1$, $K_X+\Delta\sim_\bQ 0$
and $R$ is ample, not containing any log centers of $\Delta$. In
our notations, $R$ is a polarizing divisor. By
\cite{kollar2019moduli-of-polarized}, in any dimension the irreducible
components of the moduli of log Calabi-Yau pairs with a polarizing divisor
are projective. Sections 6.4 and 8.3 of
\cite{kollar2023families-of-varieties} are closely related to our
setup. 

The situation for K3 surfaces (note $\Delta=0$) is easier because
if $(X_0,\epsilon R_0)$ is the stable limit of a one-parameter family
of K3 pairs $(X_t,\epsilon R_t)$ then the divisor $R_0$ is, perhaps
surprisingly, Cartier and not merely $\bQ$-Cartier. Indeed, the pair
$(X_0,\epsilon R_0)$ is the central fiber of the stable model of a
divisor model we defined and discussed in Section~\ref{sec:models}.
We state the main theorem for the moduli functor we need in this
paper. The details are in given in
\cite[Sec. 3]{alexeev2019stable-pair}.
 
\begin{definition}\label{moduli-conditions}
  For a fixed degree $e\in\bN$ and fixed rational number
  $0<\epsilon\le 1$, a \emph{stable $K$-trivial pair} of type $(e,\epsilon)$ is
  a pair $(X,\epsilon R)$ such that
  \begin{enumerate}
  \item $X$ is a Gorenstein surface with $\omega_X\simeq\cO_X$,
  \item The divisor $R$ is an effective, ample Cartier divisor of degree $R^2=e$.
  \item The pair $(X,\epsilon R)$ has semi log
    canonical singularities.
  \end{enumerate}
\end{definition}

\begin{definition}\label{def:family-stable-k3}
  A family of stable $K$-trivial pairs of type $(e,\epsilon)$ is a flat
  morphism $f\colon (\cX,\epsilon\cR)\to S$ such that
  $\omega_{\cX/S}\simeq \cO_\cX$ locally on $S$, the divisor $\cR$ is a
  relative Cartier divisor, such that every fiber is a stable $K$-trivial pair
  of type $(e,\epsilon)$.
\end{definition}

By \cite[Lem.~3.6]{alexeev2019stable-pair}, 
  for a fixed degree $e$ there exists an $\epsilon_0(e)>0$ such that
  for any $0<\epsilon\le \epsilon_0$ the moduli stacks
  $\cM^\slc(e,\epsilon_0)$ and $\cM^\slc(e,\epsilon)$ coincide.

\begin{definition}
  A family of stable $K$-trivial pairs of degree $e$ is a family of type
  $(e,\epsilon_0)$, with $\epsilon_0(e)$ chosen as above.
  We will denote the corresponding
  moduli functor by $\cM^\slc_e$.  For a scheme $S$,
  \begin{math}
    \cM^\slc_e(S) = 
    \{ \text{families of type $(e,\epsilon_0(e))$ over } S\},
  \end{math} with the equivalence relation being 
  $S$-isomorphisms of the family $\cX \to S$ preserving $\cR$.
\end{definition}

\begin{proposition}[{\cite[Prop.~3.8]{alexeev2019stable-pair}}]
 $\cM^\slc_e$ is a Deligne-Mumford stack of stable $K$-trivial pairs.
\end{proposition}

We denote the coarse moduli space by $M^\slc_e$.

\begin{definition}
  Let $N\in\bN$. The moduli stack $\cP_{N,2d}$ parameterizes proper
  flat families of pairs $(X, R)$ such that $(X,L)$ is a polarized K3
  surface with ADE singularities and a primitive ample line bundle
  $L$, $L^2=2d$, and $R\in |NL|$ is an arbitrary divisor. One has
  $R^2=2dN^2$. In particular, one defines $\cP_{2d}:= \cP_{1,2d}$.
\end{definition}

If we take $\epsilon_0(e)$ as above
then the pair $(X,\epsilon_0 R)$ is stable. Obviously, the stack
$\cP_{N,e}$ is fibered over the stack $\cF_{2d}$ with fibers
isomorphic to $\bP^{dN^2+1}$. The automorphism groups of stable pairs
are finite, and it is easy to see that $\cP_{N,2d}$  is
coarsely represented by a scheme $P_{N,2d}$.

\begin{definition}
  One defines $\oP_{N,2d}$ (resp. $\ocP_{N,2d}$) to be the closure of the coarse moduli
  space $P_{N,2d}$ (resp. stack $\cP_{N,2d}$) in $M^\slc_e$ (resp. $\cM_e^\slc$) for $e=2dN^2$. 
\end{definition}

For K3 surfaces polarized by a lattice $M\subset \Pic X$, choose a
primitive vector $L\in M$ with $L^2>0$. Then the substack $\cF_M
\subset \cF_{2d}$ parameterizing $M$-polarized K3 surfaces
inside of $\Z L$-polarized K3 surfaces has dimension
$20-\rank M$. A canonical choice $R$ of polarizing divisor over a Zariski
open subset $U\subset \cF_M$ (or equivalently $\cF_M^{\rm q}$ as in
Def.~\ref{def:can-choice})
defines an embedding of $U\subset \cP_{N,2d}$ if $R\in |NL|$. 

\begin{remark} We should choose $U$ to avoid the non-separated
locus of $\cF_M^{\rm q}$ to ensure that $U\subset \cP_{N,2d}$. This embedding exists
on the stack level. For instance, on the stack $\cF_2^{\rm q}$ of degree $2$ K3 surfaces,
there is a nontrivial generic inertia group $\Z_2$. Then $R$ must be preserved by the involution,
and defines an embedding of stacks $U\subset \cP_{N,2d}$. \end{remark}

\begin{definition}\label{stable-pair-comp} 
Let $\ocF_M^R$ denote the closure of $U$ in $\ocP_{N,2d}$
and let $\oF_M^R$ be its coarse space. \end{definition}

\begin{proposition}\label{contains-int} If $R$ is recognizable,
$\ocF_M^R$ contains $\cF_M$ as an open substack.\end{proposition}

\begin{proof} By Proposition \ref{recog-section-type-i}, the choice of divisor $R$
extends to all of $\cF_M^{\rm q}$ when $R$ is recognizable. Taking the relative
stable model of the universal family of pairs $(\cX,\cR)\to \cF_M^{\rm q}$
gives a classifying morphism $\cF_M^{\rm q}\to \ocF_M^R$ which necessarily
factors through the separated quotient $\cF_M$.
\end{proof}

\begin{theorem}[{\cite[Thm.~3.11]{alexeev2019stable-pair}}]
  \label{thm:stable-moduli}
  $\oP_{N,2d}$ and thus also $\oF_M^R$ are 
  projective. 
\end{theorem}

\section{$\lambda$-families}
\label{sec:lambda-families}

The goal of this section is to construct ``$\lambda$-families" of Kulikov models,
both unpolarized and $M$-quasipolarized, of a fixed combinatorial type,
and to describe the birational modifications which relate them.
These are families of Kulikov models,
which complete any one-parameter degeneration with monodromy
invariant $\lambda$ and play a critical role in the main theorem
of \cite{friedman1986type-III}: two Kulikov models with the same
$\lambda$ are related by Atiyah flops and
topologically trivial deformations.

Some improvements on {\it loc.cit.}~are made: We construct
families for which the boundary period mapping is an isomorphism onto
the period torus ${\rm Hom}(\Lambda,\C^*\textrm{ or }\cE)$, as opposed
to simply an isogeny. Also, we globalize the main theorem of  \cite{friedman1986type-III}:
two $\lambda$-families are related by certain global birational modifications
(Thm.~\ref{global-fs}, Thm.~\ref{qpol-fs}). These global modifications are key to
proving that different formulations of recognizability are equivalent (Sec.~\ref{sec:recognizable-divisors}).

Unlike Kulikov models, which depend on continuous parameters,
the $\lambda$-families depend only on combinatorial parameters, and thus
are countable in number.
Similar families of Kulikov surfaces previously appear in
work of Olsson \cite{olsson2004semistable}. See
Remark~\ref{rem:olsson} for a comparison with our version. 

\subsection{Deformation spaces of Kulikov models} We
recall the description of the universal deformation of a $d$-semistable
Kulikov surface $X_0$ given in \cite[Thm.~5.10]{friedman1983global-smoothings},
when $(X_0)_{\rm sing}$ is connected.
The deformation space $S\cup T$ has two smooth components. The
component $S$ is smooth and $20$-dimensional,
with a smooth, divisorial discriminant locus $\Delta$.
The general fiber over $s\in S$ is a
smooth K3 surface. The other component $T$
has large dimension $\rk \widetilde{\Lambda}(X_0)$,
and consists of the topologically trivial
deformations of $X_0$. These result from deforming the
gluings of double curves or the moduli of anticanonical pairs
$(V_i,D_i)$ and are generally not $d$-semistable.
$\Delta=S\cap T$ consists of the $d$-semistable, topologically
trivial deformations of $X_0$.
The universal family $\mathcal{X}\to S$ is topologically
a product $\mathcal{X}\approx_{\rm diff} \Delta\times X$
with a fixed Kulikov model $X\to (C,0)$, and has smooth total space.
In particular, $\mathcal{X}\to S$ admits a mixed marking (\ref{mixed-type})
over a contractible $S$.

As for deformations of smooth K3 surfaces, 
the local period map on $S$
is understood:

\begin{theorem}\label{thm:kulikov-torelli} Let $X_0$ be a $d$-semistable
Kulikov surface. Suppose $t>0$, or $t=0$, $k=1$
(Sec.~\ref{sec:kulikov-marking}). The period map
$S\to \mathbb{D}(I)^\lambda$ is an order
$k$ cyclic cover, branched along the boundary divisor.
\end{theorem}

\begin{proof} In Type III,
this is \cite[Thm.~5.3]{friedman1986type-III}.
The Type II case is similar
\cite{friedman1984a-new-proof}. 
 \end{proof}
 
By Theorem \ref{thm:kulikov-torelli} we can ensure
that $\mathcal{X}\to S$ is universal at all $s\in S$: 
A topologically trivial family $\mathcal{X}_0\to \Delta$ of
$d$-semistable Kulikov surfaces is a fiberwise
universal deformation if and only if the period map
to the boundary divisor
${\rm Hom}(\Lambda,\,\C^*\textrm{ or }\widetilde{\mathcal{E}})
\subset \mathbb{D}(I)^\lambda$ is a local isomorphism. 

\begin{remark}\label{not-universal} In the remaining case $t=0$,
$k>1$ the singular locus of $X_0$ is disconnected, making it possible to
independently smooth each double curve. 
The $d$-semistable deformations of $X_0$ have
dimension $19+k$ and fiber over $\C^k$, with each
coordinate hyperplane parameterizing deformations which do not
smooth a given double curve of $X_0$.

In this case, we define $S$ as the inverse image of the line
$\C(1,\dots,1)\subset \C^k$. It gives a slice transverse to the
natural action of ${\rm Aut}^0(X_0)\cong (\C^*)^{k-1}$. The discriminant
locus $\Delta\subset S$ is the inverse image of $0\in \C^k$
and is still the universal $d$-semistable topologically trivial deformation,
while the general fiber is a K3 surface that simultaneously
smooths all $k$ double curves. \end{remark}

\begin{proposition}\label{smooth-sing}
Let $\mathcal{X}_0\to \Delta_\lambda$
be a topologically trivial family of marked Kulikov surfaces for which
the period map $\Delta_\lambda\to {\rm Hom}
(\Lambda,\,\C^*\textrm{ or }\widetilde{\mathcal{E}})\subset \mathbb{D}(I)^\lambda$
is an isomorphism. There is a smoothing \vspace{5pt}

\begin{centering}

\begin{tikzcd}
\mathcal{X}_0 \arrow{r} \arrow{d} & \mathcal{X} \arrow{d} \\
\Delta_\lambda \arrow{r}  & S_\lambda
\end{tikzcd}

\end{centering}
\vspace{5pt}
\noindent for which the mixed period map 
$S_\lambda \to \mathbb{D}(I)^\lambda$ defines an 
order $k$ cyclic branched cover to an open neighborhood
of the boundary divisor. The analytic germ of the family along
 $\Delta_\lambda \subset S_\lambda$ is unique.
\end{proposition}

\begin{proof}[Sketch.] The construction parallels that 
of \cite[Exp. XIII]{asterisque1985geometrie-des-surfaces}.
When $t>0$, or $t=0$ and $k=1$, we glue together the $20$-dimensional
bases of everywhere-universal deformations of the
fibers $X_0\subset \cX_0$. With the mixed markings, these bases either glue
uniquely (when $k=1$) or uniquely up to the order $k$ cyclic action
permuting the sheets of the period mapping (Thm.~\ref{thm:kulikov-torelli}).
Taking care to ensure that the glued base is Hausdorff, 
the resulting family $\mathcal{X}\to S_\lambda$ smooths
$\mathcal{X}_0\to \Delta_\lambda$ and the germ is unique
by local universality. 
The $t=0$, $k>1$ case is proven in the same way, by instead gluing
the slices $S$ of the ${\rm Aut}^0(X_0)$ action, see
Remark \ref{not-universal}. \end{proof}

\subsection{The gluing and period complexes} We now
explicitly construct families of Kulikov surfaces
$\mathcal{X}_0\to \Delta_\lambda$
satisfying the hypotheses of Proposition \ref{smooth-sing},
developing ideas in \cite[Sec.~4]{friedman1986type-III}.
We assume here that $X_0=\bigcup\, (V_i,D_i)$ is Type III.
For notational convenience, we drop the index $i$ when
analyzing an individual component.

Each component $(V,D)$ admits a toric model
$(V,D)\xleftarrow{f} (\widetilde{V},\widetilde{D})
\xrightarrow{g} (\overline{V},\overline{D})$
where $f$ is a sequence of corner blow-ups and
$g$ is a sequence of internal blow-ups (\ref{corner-and-internal}).
Note that $f$ has no moduli whereas $g$ can be varied
by moving the non-nodal points blown up on the $\overline{D}_j$. 
Note that unless $(V,D)$ is itself toric, the
toric model is non-unique.

\begin{definition} An {\it ordered toric model} of $(V,D)$ is an orientation
of the cycle $D$ and a toric model $f , g$ as above, together with a
factorization $g=\tau_Q\circ\cdots\circ \tau_1$ into internal blow-ups.
Here $Q=Q(V,D)$ is the charge (\ref{def:charge}).
An {\it ordered toric model} of $X_0$ is an orientation of 
$\Gamma(X_0)$, a toric model of each component $(V_i,D_i)$,
and a total ordering of the $24$ internal blow-ups. 
\end{definition}

An ordered toric model of $X_0$ orients
each cycle $D_i\subset V_i$ and thus gives a way to
label the nodes on the component $D_{ij}\subset D_i$
as $0$ and $\infty$. But on the double curve
$D_{ji}\subset D_j$ the corresponding nodes have
the opposite label, so the non-nodal points of
$D_{ij}$ and $D_{ji}$ are inverse torsors for $\C^*$.

\begin{construction}\label{gigantic} Fix an ordered toric model of $X_0$ and fix
 copies of the toric surfaces $(\overline{V}_i,\overline{D}_i)$.
For a given toric surface $(\overline{V},\overline{D})$, construct a family
$$(\widetilde{\mathcal{V}}, \widetilde{\mathcal{D}})\xrightarrow{\tau_Q} \cdots \xrightarrow{\tau_1}
(\overline{V},\overline{D})\times (\C^*)^Q$$ of anticanonical pairs over $(\C^*)^Q$ by freely
varying the points blown up by $\tau_k$. There exists a simultaneous contraction
$(\mathcal{V}, \mathcal{D})\to (\C^*)^Q$ which 
contracts the corner blowdowns of $f$ fiberwise. So we have families 
$(\mathcal{V}_i,\mathcal{D}_i)\to (\C^*)^{Q_i}$ for all $i$.

Now, for each $i$, choose some fiber $(V_i,D_i)$ of this family and glue
$D_{ij}\subset D_i$ to $D_{ji}\subset D_j$ by a map identifying the
appropriate nodes of $D_i$ and $D_j$. The set of such gluings is a
torsor over $\C^*$. Varying all such gluings, we get a
 family of Kulikov surfaces $$\mathcal{X}_0^{\rm gig}\to
\textstyle \prod_i (\C^*)^{Q_i}\times (\C^*)^E=(\C^*)^{24+E}$$
whose fibers are not necessarily $d$-semistable. Here $E$ is the number
of double curves of $X_0$. We call this the
{\it gigantic gluing family} of Kulikov surfaces associated to the
ordered toric model of $X_0$. It is globally topologically trivial by construction.
\end{construction} 

Choose an origin of the open torus orbit in the
fixed toric surface $(\overline{V}_i,\overline{D}_i)$. This choice defines
a distinguished origin point of any toric boundary component and 
thus defines an isomorphism of the $\C^*$-torsor associated to
any internal blow-up or any edge-gluing with $\C^*$.
So such a choice identifies the base of the gigantic gluing family
with ${\rm Hom}(\mathcal{G}_0,\,\C^*)$ where
$$\mathcal{G}_0:=\textstyle \bigoplus_{k=1}^{24} \Z E_{ijk}
\oplus \bigoplus_{i< j} \Z D_{ij}$$ is a free $\Z$-module
encoding the blow-up points of the $\tau_k$ and the gluing maps.
Here the index ${ijk}$ indicates that $E_{ijk}$ meets the component
$D_{ij}$ and is the $k$th internal blow-up in the ordered toric model.
Note that $D_{ij}$ range only over the curves corresponding
to actual double curves appearing in $X_0$ and not to boundary
components blown down in $(\widetilde{\mathcal{V}}_i, \widetilde{\mathcal{D}}_i)$.

Consider now automorphisms. Define $\textstyle \mathcal{G}_1:=\bigoplus_i M_i$
with each $M_i\cong \Z^2$ the character lattice of the
toric surface $(\overline{V}_i,\overline{D}_i)$. The set of choices
of origin points in the open torus orbit of $(\overline{V}_i,\overline{D}_i)$
is naturally a torsor over ${\rm Hom}(\mathcal{G}_1,\C^*)$.
Fixing the family $\mathcal{X}_0^{\rm gig}$ but varying the chosen
origin point defines an equivariant action of
${\rm Hom}(\mathcal{G}_1,\C^*)$ on 
$\mathcal{X}_0^{\rm gig}\to {\rm Hom}(\mathcal{G}_0,\C^*)$
by isomorphisms. On the base, this action
is determined by a map of $\Z$-modules $\mathcal{G}_0\to \mathcal{G}_1$.

\begin{definition} The {\it gluing complex} $\mathcal{G}$ is the two-step
complex $\mathcal{G}_0\xrightarrow{\partial_\mathcal{G}} \mathcal{G}_1$. \end{definition}

We describe the map $\partial_\mathcal{G}$ explicitly. 
An orientation of the cycle $D_i$ (and thus of $\overline{D}_i$)
gives a canonical identification $w_i:M_i\to N_i$ sending $v\mapsto \det(v,-)$.

\begin{proposition}\label{glue-boundary}
  We have
  \begin{displaymath}
    \partial_\mathcal{G}(E_{ijk}) = w_i^{-1}(v_{ij}),
    \qquad
    \partial_\mathcal{G}(D_{ij}) = w_i^{-1}(v_{ij})+w_j^{-1}(v_{ji}).
  \end{displaymath} 
Here $v_{ij}\in N_i$ in the cocharacter lattice of $(\overline{V}_i,\overline{D}_i)$ 
is the primitive integral vector in the fan of $(\overline{V}_i,\overline{D}_i)$ 
corresponding to the component $\overline{D}_{ij}$.
\end{proposition}  

\begin{proof}
The action of a change-of-origin $c_i\in {\rm Hom}(M_i,\C^*)$
on the induced origin point of $\overline{D}_{ij}$ is given by
  $c_i(w_i^{-1}(v_{ij}))$. This factor scales either the
  gluing parameter between $D_{ij}$  and $D_{ji}$ or the position of
  any blow-up $E_{ijk}$ on the edge $\overline{D}_{ij}$.
     \end{proof}

Given Proposition \ref{glue-boundary}, it is convenient to identify
$\mathcal{G}_1 \cong \bigoplus_i N_i$ using the isomorphisms 
$w_i^{-1}$ on each summand, so that $\partial_\mathcal{G}(E_{ijk})=v_{ij}$ 
and $\partial_\mathcal{G}(D_{ij})= v_{ij}+v_{ji}$.

\begin{definition}
The {\it period complex} $\mathcal{P}$ of a Kulikov
model is the two-step complex $\mathcal{P}_0\xrightarrow{\partial_\mathcal{P}} \mathcal{P}_1$
where $\mathcal{P}_0=\bigoplus_i H^2(V_i)$,
$\mathcal{P}_1=\bigoplus_{i\leq j} H^2(D_{ij})$, and $\partial_P$ is the
signed restriction map with respect to an orientation
of the edges of $\Gamma(X_0)$. \end{definition}

\begin{theorem}\label{glue-period}
Let $X_0$ be a Type III
Kulikov surface with an ordered toric model. The gluing and
period complexes are quasi-isomorphic as complexes of $\Z$-modules.
In particular $H^0(\mathcal{G})=H^0(\mathcal{P})=\widetilde{\Lambda}$
and $K:=H^1(\mathcal{G})=H^1(\mathcal{P})$.
 \end{theorem}

\begin{proof} We first record some exact sequences of Picard groups arising from the basic results
on smooth projective toric surfaces: 

\begin{lemma}\label{quasiiso} Write $(V_i,D_i=\sum_j D_{ij})$ as $(V,D=\sum D_j)$.
For each component, one has the following exact sequences:
  \begin{eqnarray*}
    &0 \to \Pic \oV \to \oplus \bZ\oD_j \to N \to 0,
      &\textstyle \oL \mapsto \sum (\oL\cdot \oD_j) \oD_j \\
    &0 \to \Pic \wV \to \oplus \bZ\oD_j\oplus\bZ E_{jk} \to N \to 0,
      & \textstyle \wL \mapsto \sum (\oL\cdot \oD_j) \oD_j +
    \sum (\wL\cdot E_{jk}) E_{jk}\\
    &0 \to \Pic V\to \oplus \bZ\oD_j\oplus\bZ E_{jk} \to N \to 0,
      & \textstyle L \mapsto \sum (\oL\cdot \oD_j) \oD_j +
    \sum (\wL\cdot E_{jk}) E_{jk}
  \end{eqnarray*}
  where
  \begin{enumerate}
  \item $L,\wL,\oL$ are the line bundles on $V,\wV,\oV$,
    and $\oL = g_*\wL$.
  \item In the last line we take $\wL = f^*L$.
  \item In the last line the sum goes only over $\oD_j$
    such that $D_j = f_* \wD_j \ne 0$.
  \item $\oD_j\mapsto v_j$ and $E_{jk} \mapsto v_j$.  \end{enumerate}
\end{lemma}

Notationally reincorporating the dependence on $i$, and 
summing these exact sequences, we get a short exact sequence of two-term complexes:

  \begin{displaymath}
  \begin{tikzcd}
    & 0 \arrow[d] 
    & 0 \arrow[d] \\
    0 \arrow[r]
    & \oplus \Pic V_i \arrow[d] \arrow[r]
    & \oplus_{i\leq j} \bZ D_{ij} \arrow[d] \arrow[r]
    & 0 \\
    0 \arrow[r]
    & \oplus_{i,j} \bZ\oD_{ij} \oplus_{k=1}^{24}\bZ E_{ijk} \arrow[d] \arrow[r]
    & \oplus_{i\leq j} \bZ D_{ij} \oplus N_i \arrow[d] \arrow[r]
    & 0 \\
    0 \arrow[r]
    & \oplus N_i \arrow[d] \arrow[r]
    & \oplus N_i \arrow[d] \arrow[r]
    & 0 \\
    & 0 
    & 0 
  \end{tikzcd}
  \end{displaymath}
  
  Note that: 
  \begin{enumerate}
  \item The first column is a direct sum of sequences from the
    previous lemma.
  \item In the first line $\partial_\mathcal{P}\colon L_i\mapsto \pm \sum (L_i\cdot D_{ij}) D_{ij}$
  is the signed restriction map. 
  \item In the second column, the first map sends $D_{ij}\mapsto D_{ij} + v_{ij}$.
  \item In the second line,
    $\oD_{ij} \mapsto D_{ij} + v_{ij}$ and
    $\oD_{ji} \mapsto -D_{ij} + v_{ji}$ if the corresponding edge is oriented from
    $i$ to $j$. Also, for all $i$ and $j$,
    $E_{ijk} \mapsto v_{ij}$.
  \end{enumerate}

  The commutativity of the diagram follows from 
  \begin{math}
    \wD_i = f_i^* (\oD_i) - \sum_{j,k} E_{ijk}.
  \end{math}
  Since the last complex is acyclic, the complex $\mathcal{P}$ is
  quasi-isomorphic to the second complex $\widetilde{\mathcal{G}}$. There also is a
  quasi-isomorphism $\mathcal{G}\to \widetilde{\mathcal{G}}$. On
$\mathcal{G}_0$ it maps $\overline{D}_{ij}\mapsto \overline{D}_{ij}+\overline{D}_{ji}$ and
$E_{ijk}\mapsto E_{ijk}$ and on $\mathcal{G}_1$ it is $(0,\,{\rm id})$.
\end{proof}

\begin{proposition}\label{glue-period2}
The gigantic gluing family
$\mathcal{X}_0^{\rm gig}\to {\rm Hom}(\mathcal{G}_0,\C^*)$
descends along the canonical surjection
${\rm Hom}(\mathcal{G}_0,\C^*)\twoheadrightarrow{\rm Hom}(H^0(\mathcal{G}),\C^*)$.
Furthermore, the isomorphism
${\rm Hom}(H^0(\mathcal{G}),\C^*) = {\rm Hom}(\widetilde{\Lambda},\,\C^*)$
induced by Theorem \ref{glue-period} is the period map
of the descended family.
\end{proposition}

\begin{proof}
As noted, ${\rm Hom}(\mathcal{G}_1,\C^*)$ acts 
by automorphisms on $\mathcal{X}_0^{\rm gig}$. The action is free 
and the quotient can be constructed, for instance, by restricting 
$\mathcal{X}_0^{\rm gig}$ to a subtorus of ${\rm Hom}(\mathcal{G}_0,\C^*)$
 which intersects each orbit of ${\rm Hom}(\mathcal{G}_1,\C^*)$ exactly once.

For the second statement, it suffices to show that the action 
of regluing on periods is described by the isomorphism 
$H^0(\mathcal{G})\to \widetilde{\Lambda}$. By Construction 
\ref{iii-unwind}, regluing $D_{ij}$ or moving the blow-up point of $E_{ijk}$
by $c\in \C^*$ can be computed by gluing in $V_i$ as the last component.
The action is
\begin{align*}
  \psi_{X_0}(\gamma)
  \mapsto
    c^{\gamma\cdot D_{ij}}\psi_{X_0}(\gamma)
  \quad\textrm{and}\quad
    \psi_{X_0}(\gamma)
  \mapsto c^{f_i^*\gamma\cdot E_{ijk}}\psi_{X_0}(\gamma)
    \textrm{ for }\gamma\in\widetilde{\Lambda}.
\end{align*}
This exactly corresponds to the first chain map
$\mathcal{P}_0\to \widetilde{\mathcal{G}}_0$ in Theorem
\ref{glue-period}. Thus the map $\mathcal{P}_0\to\widetilde{\mathcal{G}}_0$ in
Lemma \ref{quasiiso} is the natural one,
$H^0(\mathcal{P})=H^0(\mathcal{G})$ canonically, 
and the proposition follows.
 \end{proof}

\begin{definition}
Let $X_0$ be a Type III Kulikov surface with ordered
 toric model. Define the {\it big gluing family} to be the descended
 family $\mathcal{X}_0^{\rm big}\to {\rm Hom}(\widetilde{\Lambda},\C^*)$
 from Proposition \ref{glue-period2}, for which the period
 map is an isomorphism.
 \end{definition} 

Now observe that ${\rm Hom}(\Lambda,\C^*)$ is the subtorus
of ${\rm Hom}(\widetilde{\Lambda},\C^*)$ corresponding to the
$d$-semistable Kulikov surfaces. So we define:

\begin{definition}
The {\it gluing family} associated to the ordered
toric model of $X_0$ is the restriction of $\mathcal{X}_0^{\rm big}$
to the subtorus $\mathcal{X}_0\to {\rm Hom}(\Lambda,\C^*)=\Delta_\lambda.$
The period map is an isomorphism.
\end{definition}

\begin{definition}\label{lambda-fam}
The {\it $\lambda$-family} of a Type III Kulikov
surface $X_0$ with ordered toric model is the unique germ $\mathcal{X}\to S_\lambda$
of the universal smoothing (Prop. \ref{smooth-sing}) of the gluing family
$\mathcal{X}_0\to \Delta_\lambda$. \end{definition} 

We delay the construction of $\lambda$-families in Type II, as some complications
arise from the non-existence of toric models.

\subsection{The global Friedman-Scattone theorem}\label{sec:global-fs}
We now discuss birational modifications. Let $X\to (C,0)$ be a Kulikov model of Type I, II, or III.

\begin{definition}
An {\it (M0), (M1), or (M2) modification} of $X$ is the
flop along a curve $E\cong \mathbb{P}^1$ in
the central fiber $X_0$.
The cases are distinguished by
when $E\cap (X_0)_{\rm sing}=\emptyset$,
when $E\cap (X_0)_{\rm sing}=\{\rm pt\}$, or 
when $E\subset (X_0)_{\rm sing}$, respectively.
\end{definition}

We describe the effect of each modification
on the central fiber $X_0$:

\begin{enumerate}

\item[(M0)] flops a smooth
$(-2)$-curve in $X_0$ which does not
deform to the general fiber.
It leaves the isomorphism
type of $X_0$ invariant. 

\item[(M1)] flops an internal exceptional
$(-1)$-curve $E$ on a component $V_i\subset X_0$.
The effect on the central
fiber is to contract $E\subset V_i$ and blow
up the intersection point $E\cap D_{ij}$
on $V_j$.

\item[(M2)] flops a double curve
$D_{ij}$ which is exceptional
on both components on which it lies. The effect
on $X_0$ is to contract $D_{ij}$ on both
$V_i$, $V_j$ and to make corner blow-ups
on the two remaining components
$V_\ell$, $V_r$ which $E$ intersects.
\end{enumerate}

\begin{notation}
In the book \cite{friedman1983the-birational-geometry},
M0, M1, M2 modifications
are called Type 0, 1, 2 modifications, but we find this to 
conflict with the already existing usage of the word
``Type." 
\end{notation}

\cite[Thm.~0.6]{friedman1986type-III} states that
any two Kulikov models with the same $(k,t)$ are related by M0, M1, and M2 modifications and topologically trivial deformations. We require an analogue of this
statement on the level of the entire $\lambda$-family.

\begin{definition}\label{gm-mod} Let $\mathcal{X}\to S_\lambda$ be the $\lambda$-family
associated to some ordered toric model of $X_0$. Let $B\subset S_\lambda$
be a smooth divisor and let $\mathcal{E}\to B$ be a smooth $\mathbb{P}^1$-fibration
for which the normal bundle to $\mathcal{E}$ restricts to
$\mathcal{O}(-1)\oplus \mathcal{O}(-1)$ on every fiber. We call the
relative flop along $\mathcal{E}$:
\begin{enumerate}
\item[(GM0)] if $B$ is the closure in $S_\lambda$ of a Noether-Lefschetz divisor of
K3 surfaces with a $(-2)$-curve, and $\mathcal{E}$ is the family of $(-2)$-curves.
\item[(GM1)] if $B=\Delta_\lambda$ is the discriminant, and $\mathcal{E}$ is a
family of internal exceptional curves meeting a relative double curve $\mathcal{D}_{ij}$.
\item[(GM2)]  $B=\Delta_\lambda$ is the discriminant, and $\mathcal{E}$
is a family of relative double curves $\mathcal{D}_{ij}$ which is, on each fiber,
exceptional on both components.
\end{enumerate}
 \end{definition}

In all three cases, the divisor $B\subset S_\lambda$ is smooth. Indeed
in the GM1, GM2 cases, $B$ is the (smooth) discriminant
divisor $\Delta$, and in the GM0 case it is the closure of
Noether-Lefschetz divisor, a hypersurface subtorus, in the
divisorial toroidal extension $S_\lambda$.

The relative flop $\cX\dashrightarrow\cX'$ along $\cE$ exists. Indeed, let
$\wt\cX\to\cX$ be the blowup along $\cE$. The exceptional divisor
$\wt\cE$ is a $(\bP^1\times\bP^1)$-fibration over $B$ and $\wt\cE
\cdot \ell=-1$ for the lines of either ruling. By
\cite{nakano1971on-the-inverse} there exists the contraction
$\wt\cX\to\cX'$ along the second ruling so that $\wt\cX$ is the
blowup along $\cE'\subset\cX'$, a $\bP^1$-fibration over $B$.

\begin{example}
Fix a Kulikov surface $X_0=\bigcup V_i$
in the $\lambda$-family $\mathcal{X}\to S_\lambda$
and consider a boundary divisor $\overline{D}_{ij}\subset \overline{V}_i$
of some ordered
toric model which receives two internal blow-ups $E_1$ and $E_2$.
On the sublocus of $\Delta_\lambda$ where the two
blow-up points coincide, the first $(-1)$-curve $E_1$
breaks into the union of the $(-1)$-curve $E_2$
and a $(-2)$-curve with class
$E_1-E_2$.
But the second exceptional curve $E_2$ never breaks, and thus
satisfies the conditions of Definition \ref{gm-mod}(GM1).
\end{example}

\begin{theorem}\label{global-fs} Any two $\lambda$-families
$\mathcal{X}\to S_\lambda$ and $\mathcal{X}'\to S_\lambda$ 
with the same $(k,t)$ are related by a series of GM0, GM1,
and GM2 modifications. \end{theorem}

\begin{proof} Choose an arc $(C,0)$
intersecting $\Delta_\lambda\subset S_\lambda$ transversely,
mapping $C^*$ generically
into a locus of K3 surfaces with Picard group $\Z L$,
$L^2=2d$.
Consider the two Kulikov models
$X\to (C,0)$ and $X'\to (C,0)$. Then the punctured families
$X^*,(X')^*\to (C,0)$ are isomorphic as families of (partially
marked) K3 surfaces, and admit polarizations.
So by \cite[Cor.~3.1]{shepherd-barron1981extending-polarizations}
there exists a sequence of M1 and
M2 modifications $X\dashrightarrow X'$ connecting them.
Requiring $\ker\psi_{X_0}=\Z L$, no modification of $X$
supports a $(-2)$-curve, eliminating the need for M0 modifications.

We now seek to globalize these modifications to
a sequence of GM0, GM1, GM2 modifications. There
is no obstruction to globalizing an M2 modification to GM2,
since the relative double curve $\mathcal{D}_{ij}$ never
breaks. For GM0 and GM1 modifications, it suffices to
work component-wise.

\begin{lemma}\label{toric-mod} Let $(\mathcal{V},\mathcal{D})\to (\C^*)^Q$
 and $(\mathcal{V}',\mathcal{D}')\to (\C^*)^Q$ be
 families of anticanonical pairs (see \ref{gigantic}) associated to two ordered toric models
 of a given pair $(V,D)$. Then, there is an isomorphism
 in $S_Q^{\pm }$ of the bases 
 and a sequence of GM0 modifications connecting
 $\mathcal{V}$ and $\mathcal{V}'$.  \end{lemma}
 
 Here $S_Q^\pm$ is the signed symmetric group, acting on $(\C^*)^Q$ by permuting
 and inverting coordinates.

\begin{proof} First, note that M0 and GM0 modifications
also make sense for anticanonical pairs, by flopping $(-2)$-curves
in the complement $V\setminus D$. Since we will only be
making birational modifications in the complement of the
anticanonical cycle, and $V\setminus D = \widetilde{V}\setminus \widetilde{D}$,
we may as well assume that $(V,D)=(\widetilde{V},\widetilde{D})$
i.e. there are no corner blow-ups in the toric model.

Fix $(V,D)$ very general, in the sense that it has no $(-2)$-curves
 disjoint from $D$. An ordered toric model is given by an 
 ordered collection $(E_1,\dots,E_Q)$ of $Q$ disjoint internal 
 exceptional curves. It follows from a theorem of Blanc
 \cite[Thm.~1]{blanc2013symplectic} describing the birational automorphism
 group of $((\C^*)^2, \tfrac{dx}{x}\wedge \tfrac{dy}{y})$---see
 \cite[Prop.~3.27]{hacking2020homological}
 for the interpretation we employ---that any two such tuples
 are related by a series of two moves:
 
\begin{definition} An {\it elementary mutation} replaces the first
exceptional curve $(E_1,E_2,\dots, E_Q) \mapsto (E_1',E_2,\dots,E_Q)$
where $E_1+E_1'$ is the pullback of a fiber of a toric ruling on $(\oV,\oD)$. \end{definition}

\begin{definition} An {\it order switch} sends 
  $(E_1,\dots,E_i,E_{i+1},\dots,E_Q)\mapsto
  (E_1,\dots, E_{i+1},E_i,\dots,E_Q)$.
\end{definition}

An elementary mutation of the ordered toric model gives
rise to an isomorphism $(\mathcal{V},\mathcal{D})\to (\mathcal{V}',\mathcal{D}')$
of the corresponding families: The construction of the family by
successive blow-ups $$(\mathcal{V},\mathcal{D})\xrightarrow{\tau_Q}\cdots
\xrightarrow{\tau_2}(\mathcal{V}_1,\mathcal{D}_1)\xrightarrow{\tau_1}
(\mathcal{V}_0,\mathcal{D}_0)=(\overline{V},\overline{D})\times (\C^*)^Q$$
is unaltered when the blow-up $\tau_1$ is replaced with the blow-up $\tau_1'$.
The bases $B\cong (\C^*)^Q\cong B'$ of the two families of varying blow-ups
$g= \tau_Q\circ\cdots\circ\tau_2 \circ \tau_1$ and
$g'=  \tau_Q\circ\cdots\circ\tau_2 \circ \tau_1'$ can thus be canonically identified.
But with respect to this canonical identification, the point blown up by $\tau_1'$
lives in the inverse $\C^*$-torsor to the point blown up by $\tau_1$. We require the 
coordinates on $B'\cong (\C^*)^Q$ to be compatible with the orientation,
so we must invert the first coordinate.

An order switch gives an isomorphism of the families
whenever $E_i$ and $E_{i+1}$ meet distinct components---$\tau_i$ and
$\tau_{i+1}$
commute. But when $E_i$ and $E_{i+1}$ meet
the same component, the families are only canonically isomorphic
over the locus where $E_i$ and $E_{i+1}$ meet distinct points.
Then $(\mathcal{V}_{i+1},\mathcal{D}_{i+1})$ and the family
$(\mathcal{V}_{i+1}',\mathcal{D}_{i+1}')$ constructed with the
reverse ordering are related by a flop along the relative
$(-2)$-curve $\mathcal{E}_{i+1}-\mathcal{E}_i\subset \cV_{i+1}\big{|}_L$
fibering over the locus $L$ where the blow-up points coincide.
The remaining blow-ups $\tau_j$ for $j>i+1$
do not interfere with the flop because $\mathcal{E}_{i+1}-\mathcal{E}_i$
is disjoint from the boundary.
The order switch permutes two $\C^*$ coordinates of $B$.

Thus, we can connect any two families $(\mathcal{V},\mathcal{D})$ and
$(\mathcal{V}',\mathcal{D}')$ by a series of isomorphisms
and GM0 modifications. The sequence of elementary mutations and
order switches connecting $(E_1,\dots,E_Q)$ to $(E_1',\dots,E_Q')$
induces an isomorphism $B\to B'$ valued in $S_Q^\pm$.
 \end{proof} 

By Lemma \ref{toric-mod},
 two families $(\mathcal{V}_i,\mathcal{D}_i)$ of anticanonical pairs associated
  to ordered toric models of a given component $(V_i,D_i)\subset X_0$ 
are connected by GM0 modifications. To
 globalize an M1 modification along $E\subset V_i$ we apply
   Lemma \ref{toric-mod} to find a sequence of isomorphisms
    and GM0 modifications until $E=E_Q$ is the {\it last} exceptional
     curve in the ordered toric model. Then $\mathcal{E}_Q$ never breaks 
     in the gluing family $\mathcal{X}_0$ as it is the last blow-up 
     performed on any given fiber. So $\mathcal{E}_Q$ can be flopped 
     in $\mathcal{X}$. The GM0 modifications on the discriminant
family $\mathcal{X}_0$ extend to the smoothing $\mathcal{X}$
because the relative $(-2)$-curve $\mathcal{E}_{i+1}-\mathcal{E}_i$
deforms over the Noether-Lefschetz divisor $B\subset S_\lambda$. 

This proves that there exists a series of GM0, GM1, GM2
modifications of $\mathcal{X}\to S_\lambda$ to a new
$\lambda$-family $\mathcal{X}''\to S_\lambda$ for which the sequence
of modifications restricts to the given sequence of birational modifications
$X\dashrightarrow X'$. Again applying Lemma \ref{toric-mod}, 
 perform a sequence of GM0 modifications until the ordered
  toric models defining $\mathcal{X}''$ and $\mathcal{X}'$ are the same.
  The theorem follows.
    \end{proof}

\subsection{Type II $\lambda$-families} We construct topologically
trivial families of Type~II Kulikov surfaces for which the period map
is an isomorphism. It is simplest to construct a family for a single
combinatorial type with $(k,t)=(k,0)$, then just apply GM0 and GM1
modifications to it. 

\begin{proposition}\label{ii-gluing}
For each $k$, there exists a
family of Type II Kulikov models $\mathcal{X}_0\to {\rm Hom}(\Lambda,\widetilde{\mathcal{E}})$
for which the period map is the identity.
\end{proposition}

\begin{proof} It suffices restrict to the $k=1$
as we may otherwise insert $k-1$ intermediate components which are
$\mathbb{P}^1$-bundles over elliptic curves.

Let $\oD\subset \mathbb{P}^2$ be an arbitrary smooth cubic. Take $18$
points $p_1,\dots,p_9, q_1,\dots,q_9\in \oD$ satisfying the single
condition $\cO_{\oD}(6)\cong \cO_{\oD}(\sum p_i +\sum q_i)$. Let $D_1$ denote
the strict transform of $\oD$ in $V_1:=Bl_{p_1,\dots,p_9}\mathbb{P}^2$
and let $D_2$ denote the strict transform of $\oD$ in
$V_2:=Bl_{q_1,\dots,q_9}\mathbb{P}^2$. 
Then $X_0:=(V_1,D_1)\cup (V_2,D_2)$
is a $d$-semistable Type II Kulikov surface, even
when $D_1$ and $D_2$ are glued via an arbitrary translation.
This construction produces a $1+(18-1)+1=19$-dimensional space of
Kulikov surfaces. Respectively, the parameters
are the $j$-invariant of $\oD$, the $18$ points $p_i,q_i$ subject
to the single condition, and the translation to glue by.

There is a projective linear automorphism acting by translation on
$\oD$ and sending one $9$-tuple to another
$(p_1,\dots,p_9)\mapsto (p_1',\dots,p_9')$ if and only if
$p_i'-p_i$ are all equal to a fixed element of ${\rm Pic}^0(\oD)[3]\cong \Z_3^2$.
Thus the family of Kulikov surfaces $\widehat{\cX}_0\to \widehat{S}$
gotten by varying the data of $\oD$, $p_i$, $q_i$, and the gluing
descends to a quotient $\cX_0\to S=\widehat{S}/(\Z_3^2\times \Z_3^2)$.
A straightforward computation of the period map on $\widehat{S}$
using Construction \ref{ii-unwind} shows that in fact, the fibers of the period
mapping $\widehat{S}\to {\rm Hom}(\Lambda,\widetilde{\cE})$ are exactly
the orbits of the $\Z_3^2\times \Z_3^2$-action. So $\cX_0\to S$ is the
desired family.
\end{proof}

\begin{corollary}\label{ii-lambda} In Type II, there is a
 family $\mathcal{X}\to S_\lambda$ of mixed marked surfaces 
 for which the period map is an order $k$ branched cover of a 
 tubular neighborhood of the boundary divisor of $\mathbb{D}(I)^\lambda$. 
 \end{corollary} 

\begin{proof}
 This follows from Proposition \ref{smooth-sing} 
and Proposition \ref{ii-gluing}. 
\end{proof}

\begin{definition} \label{type-ii-lambda}
A Type II {\it $\lambda$-family} is a family 
of surfaces which arises from a series of GM0, GM1 modifications
 of the family $\mathcal{X}\to S_\lambda$ in Corollary \ref{ii-lambda}. 
 \end{definition}

Using techniques of Theorem \ref{global-fs}, replacing
toric models with minimal models, we can construct a Type II
$\lambda$-family for any fixed combinatorial type of surface
$X_0$ via a series of GM0, GM1 modifications of the one in
Corollary \ref{ii-lambda}. 

\subsection{Quasipolarized $\lambda$-families} 

\begin{definition}
\label{lamfam-pol} An {\it $M$-quasipolarized
 $\lambda$-family} is the restriction of a $\lambda$-family 
 $\mathcal{X}\to S_\lambda$ to the Noether-Lefschetz locus 
 $\mathbb{D}_M(I)^\lambda\cap S_\lambda\subset \mathbb{D}(I)^\lambda$, 
 such that the embedding $M\to {\rm Pic}(\mathcal{X}_t)$ induced by the marking
 defines an $M$-quasipolarization on
 a generic fiber $\mathcal{X}_t$.
 \end{definition}

\begin{notation} 
When the context is clear, we reuse symbols 
$S_\lambda$, $\Delta_\lambda$ and $\mathcal{X}$, $\mathcal{X}_0$
 for the intersections $\mathbb{D}_M(I)^\lambda\cap S_\lambda$, 
 $\mathbb{D}_M(I)^\lambda\cap \Delta_\lambda$ and the restrictions
  of the unpolarized $\lambda$-families $\mathcal{X}$, $\mathcal{X}_0$
   to them.
    \end{notation}

The elements $L\in M$ extend to line bundles $\mathcal{L}\to \mathcal{X}$
 which are unique up to twisting by the relative components 
 $\mathcal{O}_{\mathcal{X}}(\mathcal{V}_i)$ and line bundles
 pulled back from the base $S_\lambda$.

\begin{theorem}\label{qpol-fs} 
Any two $M$-quasipolarized $\lambda$-families
 are related by a series of GM0, GM1, GM2 modifications.
 \end{theorem} 

\begin{proof} By Theorem \ref{global-fs}, the two unpolarized 
$\lambda$-families from which they are restricted (see Def.~\ref{lamfam-pol})
are related by GM0, GM1, GM2 modifications.
  These modifications specialize to birational modifications
   of the restricted family in all cases, except for a GM0 modification 
   associated to a $(-2)$-curve $\beta\in M$. But in this case, the two 
   restricted families are isomorphic before and after the modification 
   so we simply replace the GM0 modification with this isomorphism. 
   \end{proof}

We now define analogues of nef and divisor models.

\begin{definition} A {\it nef $\lambda$-family} is an $M$-quasipolarized
 $\lambda$-family $\mathcal{X}\to S_\lambda$ together with an
  extension of $L\in M$ to a relatively
  big and nef line bundle 
  $\mathcal{L}\to \mathcal{X}$.  \end{definition}

\begin{definition}\label{div-lambda} A {\it divisor $\lambda$-family} 
$(\mathcal{X},\mathcal{R})\to S_\lambda$ is an $M$-quasipolarized
 nef $\lambda$-family and a relatively big and nef divisor 
 $\mathcal{R}\in |\mathcal{L}|$ which contains no stratum of any fiber. 
 \end{definition}

\begin{proposition}
\label{nef-lambda-exist} Given a nef model 
$L\to X$ of a Type III $M$-quasipolarized Kulikov model $X\to (C,0)$,
 there is an ordered toric model of $X_0$ for which $L$ defines
  a nef $\lambda$-family $\mathcal{L}\to \mathcal{X}\to S_\lambda$. 
  \end{proposition}

\begin{proof} Write $L\big{|}_{X_0} = (L_i)$ with each $L_i\in {\rm Pic}(V_i)$. 
Note that $L_i$ is nef for all $i$ and at least one $L_i$ is big. It follows from
 \cite[Prop.~1.5]{engel2021smoothings} that there exists a toric model of 
 $V_i$ for which $$f_i^*L_i=\textstyle \sum a_{ij}\widetilde{D}_{ij} + 
 \sum b_{ijk} E_{ijk}.$$ with $a_{ij},b_{ijk}\geq 0$. We order this toric model so
 that $b_{ijk_1}> b_{ijk_2}$ implies that $E_{ijk_2}$ is blown up after 
 $E_{ijk_1}$. Then $L_i$ defines a relatively nef line bundle on the
  family $(\mathcal{V}_i,\mathcal{D}_i)$ because the only irreducible 
  curves which $L_i$ could possibly intersect negatively are $(-2)$-curves
   of the form $\beta=(f_i)_*(E_{ijk_2}-E_{ijk_1})$ but $$L_i\cdot \beta =
   f_i^*L_i\cdot (E_{ijk_2}-E_{ijk_1})= b_{ijk_1}-b_{ijk_2}> 0.$$

\begin{definition} An element $(\alpha_i)\in \widetilde{\Lambda}$ is 
{\it numerically nef} if $\alpha_i$ is the class of a nef line bundle 
on each component $V_i$. \end{definition}

We have that $L_i$ defines a numerically nef class on every 
fiber of the unpolarized gluing family over $ {\rm Hom}
(\Lambda,\,\C^*\textrm{ or }\widetilde{\mathcal{E}})$. 
On the sublocus of the discriminant $\Delta_\lambda$ where $(L_i)$ actually 
defines a Cartier divisor, in particular over the locus where 
$\psi_{X_0}(M)=1$, we get a relatively big and nef line bundle 
$\mathcal{L}_0\to \mathcal{X}_0$. On the smoothing 
$\mathcal{X}\to S_\lambda$, the line bundle $\mathcal{L}_0$ 
extends to a relatively big and nef line bundle $\mathcal{L}$, 
because big and nefness is an open condition. 
\end{proof}

\begin{proposition}
\label{nef-lambda-exist-ii} Given a nef model 
$L\to X$ of a Type II $M$-quasipolarized Kulikov model $X\to (C,0)$,
there is a nef $\lambda$-family $\mathcal{L}\to \mathcal{X}\to S_\lambda$
extending it.
  \end{proposition}
  
  \begin{proof}[Sketch.] The proof is roughly the same
  as Proposition \ref{nef-lambda-exist}, the key point being to
  order the exceptional curves one must successively
  blow down to get a minimal model of each component $V_i\subset X_0$.
  This ordering comes from the intersection numbers of $L_i$ with each
  exceptional curve.
   \end{proof}

\begin{definition}\label{stable-lambda} A {\it stable $\lambda$-family} 
  $(\overline{\mathcal{X}},\epsilon\overline{\mathcal{R}})\to S_\lambda$
  is defined as $\Proj_{S_\lambda} \oplus_{n\ge0} \pi_* \cO(n\cR)$ for
  a divisor $\lambda$-family. 
\end{definition}
Cohomology and Base Change theorem
 \cite[III.12.11]{hartshorne1977algebraic-geometry} implies that the fibers of a
stable $\lambda$-family are stable pairs $(\oX,\epsilon \oR)$.

\begin{remark}\label{rem:olsson}
  Olsson defined a moduli space closely related to $\lambda$-families in
  \cite{olsson2004semistable}. The functor is defined by families of
  Kulikov surfaces together with a line bundle $\cL$ extending
  a polarization,
  such that $\cL^n$ for some $n>0$ gives
  a morphism fiberwise contracting only finitely many curves.  (Olsson
  uses the language of stacks and log schemes, so this description is
  approximate, see \cite{olsson2004semistable} for complete details.)
  Our $\lambda$-families are different in a number of ways: 
  our primary focus is a divisor $\cR$, and the corresponding nef line bundle $\cL^n = \cO(n\cR)$ 
  usually contracts irreducible components of the fibers.
\end{remark}

\section{Recognizable divisors}
\label{sec:recognizable-divisors}

When a canonical choice of polarizing divisor (\ref{def:can-choice})
is recognizable (\ref{recog-def}),
Proposition \ref{recog-section-type-i} allows us to extended
 $\mathcal{R}^*$ to the whole quasipolarized moduli space
 $\mathcal{F}_M^{\rm q}$. We now generalize this to
 $\lambda$-families $\mathcal{X}\to S_\lambda$.
Recall that ${\rm Aut}^0(X_0)$ is non-trivial only when $t=0$, $k>1$, i.e.~
$X_0$ is of Type II with intermediate elliptic ruled components.
This case for $\lambda$ has a number of subtleties not present in the general case,
and we delay its treatment to Proposition \ref{type-ii-rec}. 

\begin{proposition}\label{stronger-rec} Let  $\mathcal{X}\to S_\lambda$
be an $M$-quasipolarized $\lambda$-family. If $R$ is recognizable,
then the Zariski closure of $\mathcal{R}^*$ is a flat family of curves
in $\mathcal{X}$. Conversely, if the canonical choice of divisor $R$ extends
to a flat family of divisors $\mathcal{R}^*$ on $\mathcal{F}_M^{\rm q}$,
then the existence of a further flat extension of $\mathcal{R}^*$
over any $\lambda$-family $\mathcal{X}$ implies
that $R$ is recognizable. \end{proposition}

\begin{proof} Note that $\mathcal{R}^*$ extends to a flat family
of curves in $\mathcal{X}$ if and only if the Zariski closure
$\mathcal{R}:=\overline{\mathcal{R}^*}\subset \mathcal{X}$
defines a relative curve, even over the discriminant $\Delta_\lambda$.
Equivalently, $\mathcal{R}$ contains no component
of any singular fiber $X_0$.
By recognizability, there is a ``candidate curve" $R_0\subset X_0$ which
enjoys the following property: if we take any curve $(C,0)$
transverse to $\Delta_\lambda$ at $0$, then the Zariski
closure of $\mathcal{R}^*\big{|}_{C^*}\subset \mathcal{X}\big{|}_C$
 intersects $X_0$ at $R_0$. We say that $R_0$ is the {\it flat limit
 of $\cR^*$ along the arc $C$}.
This follows from recognizability 
 because $\mathcal{X}\big{|}_C$ is Kulikov.

More generally, suppose that $(C,0)$ is an arc passing 
through $0$ which has intersection multiplicity $k$ with 
$\Delta_\lambda$. This arc defines a degenerating family
 $X\to (C^\nu,0)$ with monodromy invariant $k\lambda$. 
 Letting $t$ be a local parameter at $0\in C^\nu$, the local 
 analytic equation of the smoothing is of the form $xy=t^k$ 
 and $xyz=t^k$ near the double curves and triple points of $X_0$.

Such a family admits a standard resolution (Sec. \ref{sec:kulikov-models}) to a new Kulikov
 model $X[k]\to (C^\nu,0)$ whose dual complex $\Gamma(X_0[k])$ 
 is gotten by subdividing the triangles and segments of $\Gamma(X_0)$ into
  $k^2$ triangles and $k$ segments.
  Then $X[k]$ defines a map $(C^\nu,0)\to S_{k\lambda}$
   which is transverse to $\Delta_{k\lambda}.$ Here the Kulikov surfaces
    over the discriminant have the same combinatorial type as $X_0[k]$.
    The boundary divisors $\Delta_{k\lambda}=\Delta_\lambda$ are
     naturally isomorphic and the arcs $(C^\nu,0)$ in both $S_\lambda$ and $S_{k\lambda}$
     limit to the same point under this isomorphism.

Then $X_0[k]$ contains a distinguished curve $R_0[k]$ which is 
the flat limit of the canonically chosen divisors over any arc transverse 
to $\Delta_{k\lambda}$. So the image $\oR_0[k]$ under the morphism
 $X_0[k]\to X_0$ is equal to the flat limit of the restriction of $\mathcal{R}^*$
 to any arc with tangency $k$ to $\Delta_\lambda$. So the
   flat limit of $\mathcal{R}^*$ over any arc $(C,0)$ not fully contained
    in $\Delta_\lambda$ lies in the countable union of curves
     $$\textstyle \bigcup_{k\geq 1} \oR_0[k]\subset X_0.$$

Supposing for the sake of contradiction $\mathcal{R}\cap X_0$ 
contained a component $V_i$, there would be some point
 $p\in \mathcal{R}\cap V_i$ avoiding the above countable union. 
 Choose some irreducible curve contained in $\mathcal{R}$ passing
  through $p$ whose projection is not contained in $\Delta_\lambda$.
   Taking the image in $S_\lambda$ gives an arc $C$ passing
    through $0$, possibly singular, which intersects $\Delta_\lambda$
     with some finite multiplicity $k$ for which the restriction $\mathcal{R}^*\big{|}_{C^*}$
     contains $p$ in its Zariski closure. Contradiction.

To prove the converse is easy: Every $M$-quasipolarized
 smoothing of $X_0$ corresponds to a transverse arc $(C,0)$ 
 in the base of the $\lambda$-family $\mathcal{X}\to S_\lambda$ 
 and so the flat extension $\mathcal{R}\cap X_0$ defines a 
 curve $R_0$ satisfying the recognizability property. \end{proof}

Intuitively, recognizability implies that the limits of 
canonically chosen curves over arcs $(C,0)$ approaching 
the discriminant with tangency $k$ are rigid, for all $k$. 
On the other hand, if the closure of $\mathcal{R}^*$ 
contained a surface in $X_0$, there would have to be
 some finite tangency order $k$ for which these limit curves moved. 

\begin{remark} Proposition \ref{stronger-rec} 
implies that any of the images $\overline{R}_0[k]$ must in 
fact equal $R_0$. In particular, the divisor $R_0\subset X_0$ is 
compatible with base change plus standard resolution.
\end{remark} 

\begin{definition} We say that $R$ is  (resp. {\it weakly}) 
{\it $\lambda$-recognizable} if $\mathcal{R}^*$ extends 
to a flat family of curves in $\mathcal{X}\to S_\lambda$ for 
any (resp. some) ordered toric model of any (resp. some) Kulikov model 
with monodromy invariant $\lambda$.  \end{definition}

\begin{remark} The existence of an extension of $\mathcal{R}^*$ 
to $\mathcal{F}^{\rm q}_M$ can be considered as 
$\lambda$-recognizability in the $\lambda=0$ case. Then 
Proposition \ref{stronger-rec} states that $R$ is recognizable
 if and only if it is $\lambda$-recognizable for all possible 
 $\lambda$, including $\lambda=0$. \end{remark}

We now show equivalence with weak recognizability:

\begin{proposition}\label{strong-is-weak} $R$ is
  $\lambda$-recognizable if and only if it is weakly $\lambda$-recognizable.
\end{proposition}

\begin{proof} $\lambda$-recognizability clearly implies weak
  $\lambda$-recognizability. To show the converse, apply Theorem
  \ref{qpol-fs}: There exists a sequence of GM0, GM1, GM2
  modifications connecting any two $\lambda$-families. The condition
  that the closure of $\mathcal{R}^*$ in a $\lambda$-family
  $\mathcal{X}\to S_\lambda$ contain no fiber component is a property
  invariant under all three types of modifications, because the center
  of any such modification contains no fiber component. Hence weak
  $\lambda$-recognizability implies $\lambda$-recognizability.
\end{proof}

Proposition \ref{strong-is-weak} shows that recognizability
can be certified by finding {\it some} $\lambda$-family 
$\mathcal{X}$ for which $\mathcal{R}^*$ extends, for all $\lambda$.
The following is a key statement:

\begin{proposition}\label{no-strata-rec} Suppose $R$ is
 recognizable, and let $X\to (C,0)$ be a Kulikov model for 
 which $R_0$ contains no strata of $X_0$. Then all fibers 
 of the flat extension $(\mathcal{X},\mathcal{R})\to S_\lambda$ (Prop.~\ref{stronger-rec})
  enjoy the same property: $\mathcal{R}\cap X_p$ contains no strata of $X_p$.
\end{proposition}

\begin{proof}
  We show the Type III case; Type II works the same but easier.  Assume
  the opposite: for some $p\in \Delta_\lambda$ the divisor
  $\cR\cap X_p$ contains a triple point.  Following the argument in
  \cite[Claim 3.13]{alexeev2019stable-pair}, there is an order $k$
  base change and (possibly non-standard) simultaneous toric
  resolution producing a $k\lambda$-family $\cX'\to S_{k\lambda}$ for
  which the closure of $\cR'\big{|}_{C^*}$ in $\cX'\big{|}_{(C,p)}$ contains
  no strata. Here $(C,p)$ is an arc intersecting 
  $\Delta_{k\lambda}$ transversely at $p$ and $\mathcal{R}'\subset \mathcal{X}'$
  extends (Prop.~\ref{stronger-rec}) the canonical choice of polarizing divisor.
   
   The discriminant family
  $\cX'_0\to \Delta_{k\lambda}$ is
  topologically trivial. The divisor $\cR'$ intersects some irreducible component
  $V_p'\subset X_p'$ lying over the triple points of $X_p$ but it is
  disjoint from the corresponding component $V_0'\subset X_0'$.
  The divisor $\cR'$ is a section of a line bundle $\cL'=\cO_{\cX'}(\cR')$.
  It restricts to line bundles 
  $L_p'$ resp. $L_0'$ on $V_p'$ resp. $V_0'$, with $R_p'\in |L_p'|$ and $R_0'\in |L_0'|$.

But  since $L_0'$
  restricts to the trivial bundle on $V_0'$, the 
	 topological triviality implies that 
  $L_p'$ is the trivial bundle on $V_p'$. So $R_p'$ contains
  $V_p'$ if it intersects it. Contradiction.  An alternative
  contradiction avoiding reference to the line bundles is that
  $\mathcal{R}'_0\subset \mathcal{X}_0'$ is a flat family of curves
   intersecting $V_p'$ but not intersecting the corresponding component
   for a generic nearby fiber. This would only possible if $R_p'$
   contained a triple point, which is does not. 
  \end{proof}

Next, we study when we have the freedom to multiply $\lambda$ by an integer:

\begin{proposition}\label{no-strata-rec2} Suppose $R$ is $m\lambda$-recognizable
and the fibers of the flat extension
$\mathcal{R}\subset \mathcal{X}\to S_{m\lambda}$ contain no strata 
of any fiber. Then $R$ is $n\lambda$-recognizable for all $n$. Conversely, 
if $R$ is $n\lambda$-recognizable for all $n\in \N$, then there is an 
$m\in \N$ for which the flat extension $\mathcal{R}\subset \mathcal{X}$
 contains no strata of fibers. \end{proposition}

\begin{proof} First we prove the forward direction, i.e. we 
have a flat family of curves $\mathcal{R}\subset \mathcal{X}\to S_{m\lambda}$
not containing strata of any fiber. Let $n=mk$, and consider the standard 
resolution $\mathcal{X}[k]$ of the global base change. On any fiber,
the map $u:X_0[k]\to X_0$ satisfies the property that the inverse image
$R_0[k]:=u^{-1}(R_0)$
is still a divisor. This divisor certifies recognizability for $X_0[k]$. This would
be false if $R_0$ contained a singular stratum of $X_0$,
as then $u^{-1}(R_0)$ would contain a component.

Hence $R$ is weakly $n\lambda$-recognizable for all $m\mid n$.
By Proposition \ref{strong-is-weak}, we conclude that $R$ is
$n\lambda$-recognizable whenever $m\mid n$.
So consider the case $m\nmid n$.
Supposing $R$ were not $n\lambda$-recognizable,
the limiting divisor $R_0$ would vary depending on the
chosen arc $(C,0)\to S_{n\lambda}$. But taking a standard
resolution and base change of order $r$, we would conclude that $R$ is not
$rn\lambda$-recognizable for an $r\in \N$ as the base-changed arcs 
would also produce different limiting divisors.
Taking $r=m$ gives a contradiction.

The reverse direction follows from the existence of divisor models: 
There exists some Kulikov model $X\to (C,0)$ with monodromy
$m\lambda$ for which the limit $R_0$ contains no strata of $X_0$.
Taking a $\lambda$-family, Proposition \ref{no-strata-rec}
shows we get a flat extension $\mathcal{R}\subset \mathcal{X}$
containing no strata of fibers.
\end{proof}

\begin{proposition}\label{divisor-lambda-rec} $R$ is $n\lambda$-recognizable for all
$n\in \mathbb{N}$ if and only if there exists a divisor $m\lambda$-family
$(\mathcal{X},\mathcal{R})\to S_{m\lambda}$ for some $m\in \mathbb{N}$.
\end{proposition} 

\begin{proof} The existence of a divisor $m\lambda$-family $(\mathcal{X},\mathcal{R})$
implies that $R$ is (weakly) $m\lambda$-recognizable with $\mathcal{R}$
containing no strata of fibers, so Proposition \ref{no-strata-rec2} implies
that $R$ is $n\lambda$-recognizable for all $n\in \mathbb{N}$.
Conversely, choose a divisor model $(X,R)\to C$ with monodromy
invariant $m\lambda$. Then Proposition \ref{nef-lambda-exist} (or Proposition
\ref{nef-lambda-exist-ii} for Type II)
implies that we may choose an
ordered toric model of $X_0$ for which the line
bundle $\mathcal{O}_{X_0}(R_0)$
extends to a relatively big and nef line bundle $\mathcal{L}\to \mathcal{X}$
on the corresponding $\lambda$-family. By recognizability and
Proposition \ref{no-strata-rec}, the 
closure $\mathcal{R}=\overline{\mathcal{R}^*}$
is a section of $\mathcal{L}$
which doesn't contain strata. We conclude
that $(\mathcal{X},\mathcal{R})$ it is a divisor $m\lambda$-family. \end{proof}

We also show equivalence with a weaker condition:

\begin{proposition}\label{recog-vary} Let $\mathcal{X}\to (C,0)\times B$
 be a family of Kulikov models over a curve $B$ for which the
  discriminant family $\mathcal{X}_0=X_0\times B$ is constant, and 
  the restriction of $\mathcal{X}$ to $(C,0)\times \{b_0\}$
    gives a divisor model. Then $R$ is recognizable if and only if 
    $R_{0,b}:=\lim_{t\to 0} R_{t,b}$ is independent of $b$, i.e. 
    $R_{0,b}=R_{0,{b_0}}\subset X_0$ for any such $\mathcal{X}\to (C,0)\times B$.
  \end{proposition}

\begin{proof} Certainly if $R$ is recognizable, then $R_{0,b}$ will
equal the divisor $R_0\subset X_0$ certifying recognizability for any $b$.
Conversely, suppose $R$ is not recognizable. Following
the proof of Proposition \ref{stronger-rec}, there must be a one-parameter
family of Kulikov models $\mathcal{X}\to (C,0)\times B$ for which $R_{0,b}$ varies.
It remains to show that we may assume
these Kulikov models are divisor models. To do so, we perform a series of
GM0, GM1, GM2 modifications (possibly after a global base change and
standard resolution) until the restriction of the modified family $\mathcal{X}'$ to a fixed arc
$(C',0)\times \{b_0\}$ is a divisor model. These modifications do not affect
the triviality of the discriminant family $\mathcal{X}_0'=X_0'\times B$ and
the limit curves $R_{0,b}'$ still vary
on $X_0'$ because they cover some component. \end{proof}

\begin{proposition}\label{type-ii-rec} Suppose that $t=0$ and $k>1$. That is, $X_0$ is a Type II
Kulikov surface with intermediate elliptic ruled components.
Then, there exist $\lambda$-families $\cX\to S_\lambda$
for which Propositions \ref{stronger-rec}, \ref{strong-is-weak},
\ref{no-strata-rec}, \ref{no-strata-rec2}, \ref{divisor-lambda-rec}, \ref{recog-vary} hold.
\end{proposition}

\begin{proof} Recall that the smoothing component of such a Kulikov surface
$X_0$ has dimension $19+k$ and is fibered over $\C^k$, with the $k$th coordinate axis
corresponding to the deformations which smooth the $k$th double curve. Imposing an
$M$-quasipolarization reduces the dimension to $19+k-{\rm rk}\,M$. Given any smooth arc
$(C,0)\hookrightarrow (\C^{19+k-{\rm rk}\,M},0)=:(S,0)$ whose tangent direction $T_0C$ is transverse to
all the coordinate axes under the projection to $(\C^k,0)$, the restriction of the universal family
to $(C,0)$ is a Kulikov model, simultaneously smoothing all of the double curves. 

The closure $\mathcal{R}=\overline{\mathcal{R}^*}$ over the full $(19+k-{\rm rk}\,M)$-dimensional
smoothing component of such an $X_0$ could contain an entire intermediate
elliptic ruled component. In fact, this does occur: Applying 
$g\in {\rm Aut}^0(X_0)\cong (\C^*)^{k-1}$ to the arc $(C,0)$ in the deformation space
will translate the flat limit $R_0\subset X_0$ by $g$. But a recognizable divisor $R_0$
need not be ${\rm Aut}^0(X_0)$-invariant, see the $\widetilde{A}_{17}$ case in
\cite[Construction~9.27]{alexeev2019stable-pair}.

Fixing one arc $(C,0)\hookrightarrow  (S,0)$ gives a flat limit $R_0\subset X_0$
and assuming $R$ is recognizable, the flat limit $R_0'$ along any other arc
$(C',0)\hookrightarrow  (S,0)$ differs from $R_0$ by an element $g\in {\rm Aut}^0(X_0)$, i.e. $g(R_0')=R_0$.
But then, the flat limit along $g^*(C',0)$ equals $R_0$. So for any arc transverse to the coordinate
axes of $(\C^k,0)$, there is a representative of its ${\rm Aut}^0(X_0)$-orbit
for which the flat limit is {\it equal} to $R_0$.
Thus, there exists a slice of the ${\rm Aut}^0(X_0)$-action on $(S,0)$ for which the flat limit
along the slice is always $R_0$.

This procedure can be performed analytically-locally
along the fibers over the equisingular locus $\Delta\subset S$. We call such a slice {\it well-chosen}. 
Summarizing, a well-chosen slice gives a local $\lambda$-family over an open set $U\subset S_\lambda$
around $0\in \Delta_\lambda$ for which $\cR^*$ extends
to a flat family of divisors $\cR$.

Now consider a collection $\{U_i\}$ of well-chosen slices for which $U_i\cap \Delta$ 
cover the equisingular deformation space $\Delta_\lambda$. On the double overlaps $U_i\cap U_j$ these well-chosen
slices are isomorphic, by a unique isomorphism preserving the mixed marking,
because the isomorphisms on the smooth smooth fibers are unique (Prop.~\ref{partial-fine}).
Thus, when $R$ is recognizable, we can glue to form a $\lambda$-family $(\cX,\cR)\to S_\lambda$
on which $\cR$ extends to a flat family of divisors.
The arguments of the above propositions apply verbatim to such a well-chosen slice. \end{proof}

%\begin{remark} \label{well-chosen-remark}
%The local parameter space of well-chosen slices is a torsor over the holomorphic
%functions from $\Delta_\lambda$ to $(\C^*)^{k-1-\ell}$, corresponding to the automorphisms of $X_0$
%which move $R_0$. Thus, on the full equisingular locus $\Delta_\lambda={\rm Hom}(\Lambda,\cE)$,
%there is global $(\C^*)^{k-1-\ell}$-torsor $T\to \Delta_\lambda$ whose sheaf of local holomorphic
%sections are the well-chosen slices. It is unclear if this torsor $T$ admits a global trivialization,
%after some finite \'etale base change. It would suffice to find a finite ramified base change
%$V\to \Delta_\lambda$ for which the pullback of $T$ is trivial
%on the abelian variety ${\rm Hom}(\Lambda,E)$ for any fixed elliptic curve $E$.
%
%The fibers of $T$ are identified with the orbits of $R_0$ under the $(\C^*)^{k-1}$-action,
%and thus a trivialization of $T$ is nothing more than a globally consistent manner of choosing
%$R_i:=R_0\big{|}_{V_i}$ on each intermediate elliptic ruled component $V_i\subset X_0$ on which $R_i$
%is not $\C^*$-invariant. Note that $V_i\to E$ admits a (possibly non-minimal) ruling over the double curve.
%The condition that $R_i$ is not $\C^*$-invariant is exactly the condition that $R_i$ dominates $E$.
%But it is unclear how to proceed, and perhaps $T$ {\it can} be non-trivial. We leave this question open.
%\end{remark}

We summarize the results proven above:

\begin{theorem}\label{thm:equiv-rec} Let $R$ be a canonical
choice of polarizing divisor, defining a divisor $\cR^*$ on the universal
K3 surface over a Zariski open subset $U\subset \cF_M^{\rm q}$.
Then the following are equivalent: 
\begin{enumerate}
\item Any one-parameter deformation of a divisor model $(X,R)\to (C,0)$
keeping $X_0$ constant in moduli gives rise
to a constant limiting curve $R_0$, up to ${\rm Aut}^0(X_0)$.
\item $R$ is recognizable. 
\item For all primitive isotropic $\delta$ and all $\lambda\in C_\delta^+\cap \delta^\perp/\delta$,
there is some $\lambda$-family for which $\mathcal{R}^*$ extends a flat divisor $\mathcal{R}\subset \mathcal{X}$.
\item $\mathcal{R}^*$ extends to a flat divisor $\mathcal{R}\subset \mathcal{X}$
in every $\lambda$-family.
\item For every projective class $[\lambda]$, there exists
some $k\in \N$ for which $\mathcal{R}^*$ extends to a divisor
$\lambda$-family $(\mathcal{X},\mathcal{R})\to S_{k\lambda}$.\end{enumerate}

If $t=0$, $k>1$, the above equivalences hold when 
the $\lambda$-family is a well-chosen slice.
\end{theorem} 

\begin{proof} Note that we are allowing the case $\lambda=0$,
which in conditions (3), (4), (5) amounts to saying that $\mathcal{R}^*$
extends to a section of the projective bundle
$\mathbb{P}_\mathcal{L}\to \mathcal{F}^{\rm q}_M$. Then
$(2)\iff (4)$ by Proposition \ref{stronger-rec}, $(3)\iff (4)$ by
Proposition \ref{strong-is-weak}, and $(4)\iff (5)$ by
Proposition \ref{divisor-lambda-rec}. Finally, $(1)\iff (2)$ by
Proposition \ref{recog-vary}.
\end{proof}

The conditions in Theorem \ref{thm:equiv-rec} are roughly
in increasing order of strength. As such, we use condition
(5) in the proof of Theorem \ref{thm:recognizable-semitoroidal}, but
use condition (1) in the proof of Theorem \ref{thm:main-rc-divisor}.

\begin{definition}\label{def:slc-type}
  Let $(\oX,\epsilon\oR)= \bigcup_i (\oV_i,\oD_i,\epsilon\oR_i)$
  be a stable degeneration
  of K3 pairs. The {\it slc combinatorial type} is the data of: 
\begin{enumerate} 
\item The deformation types of
the quasipolarized minimal resolutions $(V_i,D_i, L_i)$ of each component, where
$L_i=\mathcal{O}_{V_i}(R_i)$, and
\item the combinatorics $\Gamma(\oX)$ of the
singular strata.
\end{enumerate}
\end{definition}

\begin{corollary}\label{cor:slc-type-defined-by-lambda}
Suppose $R$ is recognizable and let $(\oX^*,\epsilon \oR^*)\to C^*$, $\epsilon\ll 1$ be a 
family of stable K3 pairs over a punctured curve $C^*=C\setminus 0$.
The slc combinatorial type of the unique stable limit $(\oX_0,\epsilon\oR_0)$
  depends only on the projective class $[\lambda]$ of the monodromy invariant.
\end{corollary}
\begin{proof}
  Consider the divisor $\lambda$-family as in
  Theorem~\ref{thm:equiv-rec}(5). The family of canonical models
  $(\overline{\cX},\epsilon \overline{\cR})$, where
  \begin{math}
    \overline{\cX} = \Proj \oplus_{n\ge0} \pi_* \cO_\cX(n\cR),
  \end{math}
  and $\overline{\cR} = \im \cR$, 
  is the corresponding family of stable slc pairs.  Every
  one-parameter degeneration with monodromy invariant $\lambda$
  has a unique limit in this family. The combinatorial type of the discriminant 
  family $(\cX_0,\cR_0)$ is fixed, with the line bundles 
  $L_0=\cO_{X_0}(R_0)$ on every fiber identified
   by the Gauss-Manin connection because $\mathcal{X}_0$ is
   topologically trivial.
   Since the contraction $X_0\to \oX_0$ is defined only by the line bundle $L_0$,
   the combinatorial type of the stable models is also fixed.
%   In the $t=0,k>1$ case, we apply the same argument to a divisor $\lambda$-family
%   $(\cX,\cR)\to U$ associated to a well-chosen slice. By Remark \ref{well-chosen-remark},
%   $U$ is given by a local holomorphic section of $T$ over $\Delta_\lambda\cap U$.
%   The resulting family $(\overline{\cX},\epsilon\overline{\cR})$ of stable slc surfaces is independent of the
%   well-chosen slice. Choosing well-chosen slices arbitrarily
%   over an open cover $\{U_i\}$ of $S_\lambda$, the resulting local stable
%   $\lambda$-families glue together to a global family
%   $(\overline{\cX},\epsilon\overline{\cR})\to S_\lambda$ 
%   and the constancy of the slc combinatorial type is checked locally,
%   as $S_\lambda$ is connected.
\end{proof}

\section{Main theorem for recognizable divisors}
\label{sec:main-thm-recognizable}

\subsection{Proof of Theorem~\ref{thm:recognizable-semitoroidal}}
We have proven in Corollary~\ref{cor:slc-type-defined-by-lambda} that 
whenever $R$ is recognizable, the slc combinatorial type of an $M$-polarized degeneration depends only on the projective class $[\lambda]$ of the monodromy invariant. This is the key input which recognizability gives us: from
here we have an essentially birational-geometric argument to show that the KSBA compactifications associated to recognizable divisors are
(up to normalization) semitoroidal.

\begin{theorem}\label{main-thm} If $R$ is recognizable, there exists a unique semifan $\mathfrak{F}_R$
for which $\oF_M^{\mathfrak{F}_R}\to \oF_M^R$ is the normalization. \end{theorem}

\begin{proof} Recall that $\oF_M^R$ is, by Definition \ref{stable-pair-comp}, the coarse space
of the closure (in $\cP_{N,2d}$) of the stack of pairs parameterized by $U\subset \cF_M$.

We define the {\it interior} of $\oF_M^R$ to be the locus in this closure parameterizing
$M$-polarized ADE K3 surface pairs $(\oX,\epsilon\oR)$. Proposition \ref{contains-int}
implies that this locus is isomorphic to $F_M$.

Let $\mathfrak{G}$ be some regular
fan (cones are standard affine)
and let $u\colon \oF_M^\mathfrak{G} \dashrightarrow \oF_M^R$ be the birational
map which is isomorphism on the interiors. Let
$\sigma=\textrm{span}\{\lambda_1, \dots ,\lambda_d\}$ be a Type III
standard affine cone of $\mathfrak{G}$ of maximal dimension.
Associated to this cone is an analytic, finite morphism from a tubular neighborhood $N(\sigma)$
of  the toric boundary of
$$X(\sigma)=\C^d =\C\lambda_1\oplus \cdots \oplus \C\lambda_d$$
to a neighborhood of the boundary strata of $\oF_M^\mathfrak{G}$
containing the $0$-dimensional stratum associated $\sigma$. The finiteness
arises from quotienting by the ${\rm Stab}_{\Gamma_\delta}(\sigma)$
action on this toric chart. 

Let $u(\sigma) \colon N(\sigma)\dashrightarrow \oF_M^R$ denote the
corresponding meromorphic map. Consider an arc germ $(C,0)\subset (\C^d,0)$ 
with $C^*\subset (\C^*)^d$ contained in the open torus orbit.
Since $\oF_M^R$ is proper, $u(\sigma)$ extends uniquely over $C^*$ to the origin $0$.
By Corollary \ref{cor:slc-type-defined-by-lambda}, the combinatorial type of the stable model 
depends only on the orders $r_i$ of tangency of $(C,0)$ with the coordinate
hyperplanes of $\C^d$, since this determines
the monodromy invariant of $(C,0)$ to be
$\lambda =r_1\lambda_1+\cdots +r_d\lambda_d$.

The meromorphic map $u(\sigma) \colon N(\sigma) \dashrightarrow \oF_M^R$
thus satisfies the following conditions:
\begin{enumerate}
\item There is a stratification (by slc 
combinatorial type) of $\oF_M^R$ for which
the extension of $u(\sigma)$ over any arc $(C,0)$ with fixed
tangency orders $r_i$ to the coordinate hyperplanes of $\C^d$
lies in a fixed slc stratum. 
\item The indeterminacy
locus lies in the coordinate hyperplanes, which map by $u(\sigma)$
into the union of Type III slc strata. \end{enumerate}

No Type III slc stratum contains a
complete curve by Corollary \ref{strata-affine}.
We conclude by
Lemma \ref{toric-resolution} that
there exists a toric blow-up of $X(\sigma)$ eliminating the
indeterminacy of $u(\sigma)$. Further refining,
we may assume this toric blow-up is given by a
${\rm Stab}_{\Gamma_\delta}(\sigma)$-invariant fan.
 Thus, we may refine $\mathfrak{G}$ so that
 $u$ defines a morphism over the refinement of $\sigma$.
Applying this argument to all $\Gamma$-orbits of maximal cones
 $\sigma\in \mathfrak{G}$,
 we may as well have assumed that
 $u: \oF_M^\mathfrak{G}\dashrightarrow \oF_M^R$
 has no indeterminacy over the Type III extension of $F_M$.

In fact, there is no indeterminacy in the Type II ($\lambda^2=0$) locus either:
By Theorem \ref{thm:equiv-rec},
there is a divisor $\lambda$-family
$(\mathcal{X},\mathcal{R})\to S_\lambda$. Consider the
resulting stable $\lambda$-family
$(\overline{\mathcal{X}},\epsilon\overline{\mathcal{R}}) \to S_\lambda$.
The base $S_\lambda$ is an order $k$ branched cover
of a tubular neighborhood of the boundary divisor
in the unipotent quotient $\mathbb{D}_M(I)$. 
There is a natural quotient map $v\colon S_\lambda\to\oF_M^\mathfrak{G}$
by the action of $\Gamma_I$.

The classifying morphism $S_\lambda\to \oF_M^R$ for the stable $\lambda$-family
must factor through $v$ because the
fibers of $v$ not lying in the boundary give isomorphic ADE K3 surfaces with divisor.
Ranging over all $I=\Z\delta\oplus \Z\lambda$, the maps $v$ surject
onto the Type II locus, so $u$ extends to a morphism
over the Type II extension of $F_M$.

Since the Type II and III extensions of $F_M$ cover all of $\oF_M^{\mathfrak{G}}$,
we conclude that there is a morphism $\oF_M^{\mathfrak{G}}\to \oF_M^R$---on the intersection
of the closure of the Type II locus with the Type III locus,
it is a morphism as opposed to just a set-theoretic map because
$\oF_M^{\mathfrak{G}}$ is normal. 

By Lemma \ref{lem:over-bb},
we also have a morphism $(\oF_M^R)^\nu\to \oF_M^\bb$. 
So by Theorem \ref{thm:semi-is-normal},
the normalization of $\oF_M^R$ is semitoroidal
for a unique semifan $\mathfrak{F}_R$.
  \end{proof}

\begin{corollary}\label{tor-to-slc-strata} Suppose $R$ is recognizable.
The normalization map $\oF_M^{\mathfrak{F}_R}\to \oF_M^R$
sends semitoroidal strata to slc strata.
\end{corollary}

\begin{proof} Let $\sigma\in \mathfrak{F}_R$ be any cone and
choose $\lambda$ in the relative 
interior ${\rm int}(\sigma)$. By Corollary
 \ref{cor:slc-type-defined-by-lambda}, the stable limit of any degeneration with 
 monodromy invariant $\lambda$ lies in a fixed slc stratum. 
 Since the natural map $\delta^\perp/\{\delta,\lambda\}\to 
 \delta^\perp/\{\delta,\sigma\}$ is surjective, every point in ${\rm Str}_\sigma$ is the limit of some
   arc with monodromy invariant $\lambda$. So the combinatorial 
   type of the slc stable model at any point in ${\rm Str}_\sigma$ is the same.
   \end{proof}

Corollary \ref{tor-to-slc-strata} implies that there is a well-defined function 
\begin{align*} \mathbb{S}\colon \{\textrm{cones of }\mathfrak{F}_R\textrm{ mod }\Gamma\}&\to \left\{\!\!
\begin{array}{ll} \textrm{combinatorial
types of slc} \\ \textrm{strata which
appear in }\oF^R_M \end{array}\!\!\right\}. \end{align*}

Note that $\mathbb{S}$ may not be injective.
For instance, $\mathbb{S}(\sigma)=\mathbb{S}(\tau)$ if the
corresponding strata are unglued by normalizing. By abuse,
let $\mathbb{S}(\lambda):=\mathbb{S}(\sigma)$ 
where $\lambda\in{\rm int}(\sigma)$. 

\begin{theorem}\label{stratum-function} Let $R$ be a recognizable divisor for $F_M$.
Let $D$ be the decomposition of monodromy invariants 
into loci $\left\{\lambda \in \textstyle 
\coprod_\delta C_\delta^+\cap \delta^\perp/\delta\,\,\big{|}\,\,\mathbb{S}(\lambda)\textrm{ is constant}\right\}$.
Then maximal cones of $\mathfrak{F}_R$ and $D$ are the same. 
 \end{theorem}
 
A {\it maximal cone} of $D$ is a top-dimensional, convex cone in
 $C_\delta^+$ whose {\it integral} interior points lie in a single element of $D$,
 and which is maximal for this property.
 
 \begin{proof} $\mathbb{S}$ is constant on cones of $\mathfrak{F}_R$ by Corollary \ref{tor-to-slc-strata},
 so it suffices to show that $\mathbb{S}$
 cannot take the same value on two maximal dimensional cones
 $\sigma_1,\sigma_2\in\mathfrak{F}_R$ and a codimension $1$ face
 $\tau\subset \sigma_1\cap \sigma_2$ they share.
 If this were the case, the closed
 boundary stratum $\overline{\rm Str}_\tau$ would map to a fixed slc stratum 
 $\mathbb{S}(\sigma_1)=\mathbb{S}(\sigma_2)=\mathbb{S}(\tau)$.
 But the Type III slc strata contain no complete curve by Corollary \ref{strata-affine}.
 So $\overline{\rm Str}_\tau$ would be contracted to a point, contradicting
 finiteness of the normalization $\oF^{\mathfrak{F}_R}_M\to\oF^R_M$. \end{proof}

Theorem \ref{stratum-function} gives a method to compute
the semifan $\mathfrak{F}_R$. Up to taking faces, its cones are sets of
monodromy invariants $\lambda$ which produce a fixed combinatorial slc type.
This is how $\mathfrak{F}_R$ was computed in Examples \ref{deg2-ex}, \ref{ell-ex}
below.

The semifan $\mathfrak{F}_R$ is also functorial under restriction to
Type IV subdomains of $F_M$, i.e. Noether-Lefschetz loci. 
Let $M\subset M'\subset L_{K3}$ be primitive hyperbolic sublattices.
Then there is a natural map of moduli stacks
$\mathcal{F}_{M'}^{\rm q}\to \mathcal{F}_M^{\rm q}$
sending $(X,j)\mapsto(X,j\big{|}_M)$. Let $L\in M$.

\begin{proposition}\label{inclusion} Suppose $R\in |L|$ is recognizable for $\mathcal{F}_M^{\rm q}$. Then its
restriction to $\mathcal{F}_{M'}^{\rm q}$ is also recognizable. Furthermore, $\mathfrak{F}_R({M'})$ is the restriction of the semifan $\mathfrak{F}_R(M)$ to the appropriate
linear subspaces of $C_\delta^+\subset \delta^\perp_{M^\perp}/\delta$. 
\end{proposition}

More precisely, if $\delta\in {M'}^\perp\subset M^\perp$ is an isotropic vector corresponding
to some $0$-cusp of $F_{M'}$, we restrict the decomposition 
$\mathfrak{F}_{R,\delta}(M)$ to the subspace $\delta^{\perp}_{{M'}^\perp}/\delta$.
  
\begin{proof} Proposition \ref{inclusion} follows from the
 the fact that any ${M'}$-quasipolarized Kulikov model is also $M$-quasipolarized, plus the
 functoriality of the stable pair and semitoroidal constructions
 under restriction to Noether-Lefschetz subdomains.\end{proof}

\subsection{Moduli of anticanonical pairs} We prove 
here that Type III slc strata contain no complete curve by
considering the periods of
anticanonical pairs. A useful general reference is
\cite{friedman2015on-the-geometry}.

\begin{definition} Let $(V,D)$ be an anticanonical pair with $D=D_1+\cdots+D_n$ an
oriented, labeled cycle of rational curves. Define
$\Lambda_{(V,D)}:=\{D_1,\dots,D_n\}^\perp\subset H^2(V)$ and
define the {\it period point} $\psi_{(V,D)}\in {\rm Hom}(\Lambda_{(V,D)},\,\C^*)$
to be the restriction map $\gamma\mapsto \gamma\big{|}_D\in {\rm Pic}^0(D)=\C^*$.
\end{definition}

\begin{definition}[{\cite[Def.~5.4]{friedman2015on-the-geometry}}]
The {\it generic ample cone} $A_{\rm gen}\subset H^2(V)$ 
 is the ample cone of a very general topologically trivial 
 deformation of $(V,D)$.\end{definition}

It suffices to take a deformation for which ${\rm ker}(\psi_{(V,D)})=0$.
This is possible because there is a local universal
deformation $(\mathcal{V},\mathcal{D})\to S$ of pairs
for which the assignment $s\mapsto \psi_{(V_s,D_s)}$ is an isomorphism
to an open subset of ${\rm Hom}(\Lambda_{(V,D)},\C^*)$.

\begin{definition}[{\cite[Def.~6.5]{friedman2015on-the-geometry}}]
A {\it Looijenga root} $\beta\in \Lambda_{(V,D)}$ is a class of
square $\beta^2=-2$ which represents a smooth $(-2)$-curve on
some topologically trivial deformation of $(V,D)$,
and for which $\psi_{(V,D)}(\beta)=1$.
\end{definition}

Reflections in Looijenga roots act on $A_{\rm gen}$. The
ample cone $A$ of $(V,D)$ is a fundamental chamber for the
action of the group
$W_{(V,D)}:=\langle r_\beta \colon \beta\textrm{ a Looijenga root}\rangle $
on $A_{\rm gen}$. We can now recall the Torelli theorem for anticanonical pairs:

\begin{theorem}[{\cite[Thm.~8.7]{friedman2015on-the-geometry}}]
Two pairs $(V,D)$ and $(V',D')$ (with oriented, labeled cycle)
are isomorphic if and only if there exists an isometry $\phi\colon H^2(V)\to H^2(V')$
for which $\phi(D_j)=D_j'$, $\phi(A_{\rm gen})=A'_{\rm gen}$, and
$\psi_{(V,D)}=\psi_{(V',D')}\circ \phi$. Furthermore, $\phi=f^*$ is
induced by an isomorphism $f:(V',D')\to (V,D)$ if and only if
$\phi(A)=\phi(A')$. This isomorphism is unique up to the
action of continuous automorphisms ${\rm Aut}^0(V,D)$.\end{theorem}

So the analogue of the Torelli Theorem \ref{torelli-i} holds
nearly verbatim, replacing $\cC$ with $A_{\rm gen}$ (which is notably
not the positive cone), 
 $\cK$ with $A$ (which is  the K\"ahler cone), and $W_X$ with $W_{(V,D)}$.

\begin{definition} Fix a reference lattice $L_{(V,D)}$ isomorphic
to $H^2(V)$. Fix classes $(D_j)^0\in L_{(V,D)}$ and fix a cone
$A_{\rm gen}^0\subset L_{(V,D)}\otimes \R$. A {\it marking} of
$(V,D)$ is an isometry $\sigma:H^2(V)\to L_{(V,D)}$ sending
$\sigma(D_j)=(D_j)^0$ and $\sigma(A_{\rm gen}) = A_{\rm gen}^0$.
Let $\Gamma_{(V,D)}\subset O(L_{(V,D)})$ be the subgroup
fixing all this data. \end{definition}

\begin{theorem}[{\cite[Thm.~8.13]{friedman2015on-the-geometry}}]
Assume ${\rm Aut}^0(V,D)$ is trivial. There is
a fine moduli space $\mathcal{M}_{(V,D)}$ of marked 
anticanonical pairs deformation-equivalent to $(V,D)$.
 It has a period map $$\mathcal{M}_{(V,D)}\to {\rm Hom}(L_{(V,D)},\C^*)$$
  which is generically one-to-one, and whose fibers are torsors
   over a group isomorphic to $W_{(V,D)}$ with the action on
    a fiber given by $(X,\sigma)\mapsto (X,g\circ \sigma)$. \end{theorem}

When ${\rm Aut}^0(V,D)$ is non-trivial, there is still a space
$\mathcal{M}_{(V,D)}$ admitting a family which defines at
every point a universal deformation, and for which every
isomorphism type is represented, but it is not a fine moduli space.

\begin{definition} A {\it quasi-polarized triple} $(V,D,L)$ is an
anticanonical pair $(V,D)$ and a big and nef line bundle
$L\in {\rm Pic}(V)$. A {\it polarized ADE triple} is an image
 $(\oV,\oD,\oL)$ of such under the linear system $\phi_{|nL|}$, 
 $n\gg 0$ (we must add the condition that $\psi_{(V,D)}(L)=1$ when
 $L\in \Lambda$). A {\it divisor triple} $(V,D,R)$ is the extra data of 
 an element $R\in |L|$ such that $R$ contains no nodes of $D$. 
 A {\it stable triple} $(\oV, \oD, \epsilon \oR)$
 is an image of a divisor 
 triple $(V,D,\epsilon R)$ under $\phi_{|nR|}$, $n\gg 0$.\end{definition}

The map $(V,D)\to (\oV,\oD)$ contracts the components of
$D$ for which $L\cdot D_j=0$, together with some
negative-definite ADE configuration of $(-2)$-curves
whose classes lie in $\Lambda_{(V,D)}$.

\begin{theorem} The coarse moduli space of polarized ADE triples
$F_{(\oV,\oD,\oL)}$ of a fixed deformation type is the quotient of
${\rm Hom}(L_{(V,D)},\,\C^*)$ by the finite group 
$\Gamma_{(V,D,L)}:={\rm Stab}_{\Gamma_{(V,D)}}(L)$. \end{theorem}

\begin{proof} The result is analogous to Theorem \ref{smooth-coarse}.
If $L\notin \Lambda$, take the sublocus $\mathcal{M}_{(V,D,L)}\subset \mathcal{M}_{(V,D)}$ 
where $L$ defines a big and nef divisor---this surjects onto the 
period torus with fibers a torsor over the reflection subgroup 
$W_{(V,D,L)}:={\rm Stab}_{W_{(V,D)}}(L)$. When $L\in \Lambda$,
we restrict to the sublocus $\psi_{(V,D)}(L)=1$.
 Now take the relative 
linear system of $nL$, which
simultaneously contracts the ADE configuration in 
$\Lambda_{(V,D)}\cap L^\perp$ and some components of $D$.

The fibers of the period map $\mathcal{M}_{(V,D,L)}\to {\rm Hom}(L_{(V,D)},\,\C^*)$ 
(or ${\rm Hom}(L_{(V,D)}/\Z L,\,\C^*)$ when $L\in \Lambda$)
are identified with distinct resolutions of the contraction, and 
the moduli functor factors through the separated
quotient of $\mathcal{M}_{(V,D,L)}$. Since we have
included $\oL$ as part of the data, our
change-of-markings in $\Gamma_{(V,D)}$ must preserve $L$. The result follows.

We can even identify (when ${\rm Aut}^0(V,D)$ is trivial) the
moduli stack as the separated quotient of
$[\mathcal{M}_{(V,D,L)}:\Gamma_{(V,D,L)}]$. Like in the
K3 case (Rem.~\ref{stack-difference}), its only difference with the quotient stack
$[{\rm Hom}(L_{(V,D)},\,\C^*)\colon \Gamma_{(V,D,L)}]$
is that the inertia groups are locally quotiented by $W_{(V,D,L)}$. \end{proof}

Let $F_{(\oV,\oD,\oR)}$ denote the coarse moduli space
  of stable triples $(\oV,\oD,\epsilon\oR)$ with a fixed
  deformation type of minimal resolution. Here $\epsilon$ is a
  fixed small number.

\begin{lemma}\label{fibers-affine} $F_{(\oV,\oD,\oR)}$
  is a (possibly non-flat)
family of affine varieties
over the coarse moduli space $F_{(\oV,\oD,\oL)}$. \end{lemma}

\begin{proof} On a given polarized ADE triple $(\oV,\oD,\oL)$, we may
choose $\oR\in |\oL|$ arbitrarily, subject to the condition that $\oR$
not contain any nodes of $\oD$. This condition is either not
satisfied by any element of $|L|$, or is the complement of a
non-zero number of hyperplanes, corresponding to
sections which go through some node. Thus, the set of choices
of $\oR$ on a fixed ADE triple forms an affine variety.

The automorphism group of $(\oV,\oD,\oL)$ acts on the set of
such choices $\oR$. So when this automorphism group is finite,
the choices form an affine variety. If the automorphism group contains
a continuous part of dimension $1$ or $2$, we may rigidify by requiring
$\oR$ to go through $1$ or $2$ generically chosen points of $\oV\setminus \oD$.
Then, the coarse moduli space is a finite image of the rigidified
moduli space, which is again affine by the reasoning of the
first paragraph. \end{proof}

\begin{corollary}\label{component-no-complete}
The coarse moduli space $F_{(\oV,\oD,\oR)}$ contains no complete curves. \end{corollary}

\begin{proof} This follows immediately from Lemma \ref{fibers-affine}
  and $F_{(\oV,\oD,\oL)}$ being affine.
\end{proof}

 Let $F_{(\oX,\oR)}$ denote the coarse moduli space
 of stable slc pairs of a fixed combinatorial type,
 as in Definition~\ref{def:slc-type}.

 \begin{remark}
   Semi log canonical singularities are seminormal.
   The seminormality implies 
that the scheme-theoretic structure of a $0$-stratum of 
$\oX$ is unique, since $\oX$ is the direct limit of the diagram of
strata, partially ordered by inclusion.
So moduli is uniquely determined by the 
moduli of components and gluings of double curves.\end{remark} 

\begin{theorem}\label{stable-all-possible} Let $F_{(\oX,\oR)}$
be a coarse moduli space of glued seminormal stable pairs containing a
Type III stable pair degeneration of K3 surfaces.
 Then $F_{(\oX,\oR)}$ contains no complete curve.
 \end{theorem}

\begin{proof} We can construct the coarse moduli space as follows: 
First, take the product of the coarse moduli spaces of each component 
$\prod_i F_{(\oV_i,\oD_i,\oR_i)}$. Let  $\{\mu_{ij}\}\subset \C^*$
be the (possibly empty, but always finite)
set of gluings of $\oD_{ij}$ to $\oD_{ji}$ which identify
the nodes of $\oD_i$ and $\oD_j$ and for which
$\oR_i\cap \oD_{ij}=\oR_j \cap \oD_{ji}$. The space $G_{(\oX,\oR)}$ 
of such glued pairs is $\prod_i F_{(\oV_i,\oD_i,\oR_i)}\times \prod_{i,j} \{\mu_{ij}\}$
which has a finite map to $\prod_i F_{(\oV_i,\oD_i,\oR_i)}$. So by
Corollary \ref{component-no-complete},
$G_{(\oX,\oR)}$ contains no complete curves. 

The space $G_{(\oX,\oR)}$ parameterizes seminormal
pairs $(\oX,\epsilon \oR)$ together with a combinatorial labeling of the dual
complex $\Gamma(\oX)$. Consider the finite group of combinatorial self-maps of
   $\Gamma(\oX)$ preserving the combinatorial types of all
    stable triples. The coarse moduli space $F_{(\oX,\oR)}$
    is the quotient of $G_{(\oX,\oR)}$ by this finite group.
    Since $G_{(\oX,\oR)}$ contains no complete curve, neither does
   $F_{(\oX,\oR)}$.
\end{proof}

\begin{corollary}\label{strata-affine} No Type III stratum of $\oF_M^R$ contains a complete curve.\end{corollary}

\begin{proof} A Type III stratum of $\oF_M^R$ is a sublocus of the coarse
moduli space of pairs $(\oX,\epsilon \oR)$ as in Theorem \ref{stable-all-possible}.
The corollary follows. \end{proof}

\subsection{Other lemmas} We prove the remaining lemmas used in Theorem \ref{main-thm}.
Let $(\C^n,B)$ denote the analytic germ of $B:=\{x_1\cdots x_n=0\}$, the union of the coordinate
hyperplanes, in $\C^n$.

\begin{lemma}\label{toric-resolution}
  Let $V$ be an analytic variety stratified by sub-varieties $V_i$.  Consider a
  meromorphic map $\phi\colon (\mathbb{C}^n,B)\dashrightarrow V$ with locus
  of indeterminacy contained in $B$. Assume
  that the image of the indeterminacy locus is contained in
  $\cup_{i\in I}V_i$ and that no $V_i$ for $i\in I$ contains
  a complete curve.

  Assume that for any arc germ $f\colon (C,0)\to (\bC^n,0)$ with
  $f(C\setminus 0)\subset (\C^*)^n$,
  there exists an extension 
  $g\colon (C,0)\to V$ of $\phi\circ f$.
  Moreover, assume that for any such $f$, the stratum $V_i\ni g(0)$
  depends \emph{only} on the orders of tangency of $C$ to the
  coordinate hyperplanes.

  Then the indeterminacy of $\phi$ can be resolved by toric blow-ups.
\end{lemma}
\begin{proof}
  Fix the standard torus action of $T=(\bC^*)^n$ on $\bC^n$.  By
  Hironaka (see W{\l}odarczyk \cite{wlodarczyk2009resolution} for a careful
  treatment of analytic spaces), there exists a sequence of
  blowups at smooth centers in the indeterminacy loci that resolves
  $\phi$. Let $H$ be the first center which is not $T$-invariant.
  %Thus, $X\to(\bC^n, B)$ is toric and $H\subset X$ is not $T$-invariant.

  Let $O$ be the largest $T$-orbit with $O\cap
  H\ne\emptyset$. By restricting to an open subset, we can assume 
  that $H\subsetneq O$. Consider the toric cross-sections normal to $O$.
  These cross-sections satisfy the conditions of the Lemma,
  and so applying an inductive hypothesis in $n$,
  we resolve the indeterminacy of $\phi$ generically
  along $O$, by a series of toric blowups.

  So we get a rational map $\phi'\colon (X',B') \dashrightarrow V$ from a toric
  variety, a torus orbit $O'$, and a nontoric center of indeterminacy
  $H'\subset O'$ such that $\phi'$ is regular on an open set
  $U\subset X'$ intersecting $O'$. Then $\phi'(U\cap O')$ is contained in a
  single stratum~$V_i$. The stratum containing the limit of an arc in $X'$ again
  depends only on the orders of tangency with the components of $B'$.

  Let $X'\gets Z\to V$ be a resolution of singularities. Then there
  exists $p\in O'\cap H'$ such that for the fiber $Z_p$ of $Z\to X'$ the
  morphism $Z_p\to V$ is non-constant. Since $Z_p$ is proper and
  the strata $V_i$ contain no complete curve, there exist two arcs
  with $f_1(0)= f_2(0) = p$
  and with $g_1(0)$, $g_2(0)$ lying in different strata $V_i$.
  But shifts of these arcs by the torus action have the same tangency
  conditions with the coordinate hyperplanes and satisfy $f(0)\in
  U$. So for them $g(0)$ lie in the same stratum of $V$. Contradiction.
\end{proof}

\begin{lemma}\label{lem:over-bb}  
There is a morphism $(\oF_M^R)^\nu\to \oF_M^\bb$ for
any canonical choice of polarizing divisor $R$ (recognizable or
not). \end{lemma}

This is proved in \cite[Thm.~3.15]{alexeev2019stable-pair} and amounts
to the observation that in Type II, the $j$-invariant of the
corresponding point in the $1$-cusp of $\oF_M^\bb$ can be
recovered from the stable slc pair $(X,\epsilon R)$. Indeed, either
$X$ is nonnormal and every connected component of the double locus is
an elliptic curve $E$ with this $j$-invariant, or $X$ has an elliptic
singularity corresponding to $E$.

\subsection{Examples} Previously known examples
 of recognizable divisors come from
  \cite{alexeev2019stable-pair}, \cite{alexeev2022compactifications-moduli}.

\begin{example}[Degree $2$ K3s]\label{deg2-ex} 
Let $(X,L)$ be a quasipolarized K3 surface of degree $L^2=2$. Let $\oX$
denote the corresponding polarized ADE K3 surface.
Then $\phi_{|L|}$ defines a branched double cover $\oX\to \mathbb{P}^2$ or
$\oX\to \mathbb{F}_4^0$ (the contraction of the Hirzebruch surface $\mathbb{F}_4$
along its negative section). Define $R\in |3L|$ to be the pullback of the
   ramification locus $\oR\subset \oX$.

There is only one $\Gamma$-orbit of primitive isotropic 
vector $\delta\in L^\perp$, and so a semitoroidal compactification is 
determined by a single $\Gamma_\delta$-invariant semifan 
$\mathfrak{F}=\mathfrak{F}_\delta$ in the positive cone of 
$C_\delta^+$. Then \cite{alexeev2019stable-pair} verifies
Theorem \ref{thm:equiv-rec}(5) directly,
by constructing for each monodromy invariant $\lambda$,
a divisor $\lambda$-family with monodromy invariant
in the projective class $[\lambda]$.

These divisor models are constructed by ensuring the involution on the general
fiber $X_t$ of Kulikov model $X\to (C,0)$ extends to the central fiber $X_0$.
Then the fixed locus $R_0\subset X_0$ is the canonical choice of divisor
certifying recognizability. The resulting semifan $\mathfrak{F}_R$ is not a fan. 
The lattice $\delta^\perp/\delta$ is a hyperbolic root lattice with a finite 
covolume Coxeter chamber $\mathfrak{K}$. There is an infinite 
subgroup $W\subset \Gamma_\delta$ for which
 $\mathfrak{F}_R$ is the $\Gamma_\delta$-orbit of a 
 single chamber $\mathfrak{L}:=W\cdot \mathfrak{K}$.
 
 Thus, Theorem \ref{main-thm} {\it cannot} be strengthened 
 by replacing ``semifan" with ``fan." \end{example}

\begin{example}[Elliptic K3s]\label{ell-ex}
Let $(X,j)$ be an $H$-quasipolarized K3 surface, i.e. an elliptic K3 surface, with
fiber class $f$ and section $s$ (so $h$ lies in cone spanned by $f$, $s+2f$). 
Let $R=s+m\sum f_i$ be the section plus the sum of the singular fibers,
with multiplicity, and weighted by $m$. Here $L=s+24mf\in H$ is the relevant
 big and nef class. $\oF_H^R$
 is the same for all $m> \tfrac{1}{3}$.

As in the previous example, \cite{alexeev2022compactifications-moduli} find 
Kulikov models $X\to (C,0)$ for any monodromy invariant which preserve 
the existing structures on the general fiber: $X_0$ admits a fibration
 by genus $1$ curves $\pi_0\colon X_0\to B_0$
 over a chain of rational curves $B_0$, with a section $s_0$.
Finitely many fibers $f_{i,0}=\pi_0^{-1}(b_i)$ not contained in the double 
locus of $X_0$ have more nodes than all analytically nearby fibers.  Counting $f_{i,0}$
 with the correct multiplicity, the recognizable divisor on $X_0$ is $$R_0=s_0+m\sum f_{i,0}.$$

There is a unique $\Gamma$-orbit of primitive isotropic
$\delta\in H^\perp=I\!I_{2,18}$ and $\delta^\perp/\delta=I\!I_{1,17}$
 is a hyperbolic root lattice with a finite covolume Coxeter chamber
  $\mathfrak{K}$. Then $\mathfrak{F}_R$ is the $\Gamma_\delta$-orbit
   of a subdivision of $\mathfrak{K}$ into $9$ subchambers. 
   So $\mathfrak{F}_R$ is a fan. \end{example}

\begin{example} Any choice of divisor $R$ when $\dim\mathbb{D}_M=1$ 
is recognizable: There exists a divisor model in the neighborhood
 of any point $p\in \oF_M$ in the unique toroidal compactification, 
 which also equals the Baily-Borel compactification. \end{example}

\begin{remark} The necessity of normalizing $\oF_M^R$ to
 get a semitoroidal compactification is apparent in both 
 Examples \ref{deg2-ex}, \ref{ell-ex}.
 \cite{alexeev2019stable-pair, alexeev2022compactifications-moduli}
  compute the normalization map explicitly. 
   \end{remark}

\section{The rational curve divisor}
\label{sec:rc-divisor}

Our goal is to now make a canonical choice of divisor for
$F_{2d}$ for any $d>0$, then prove its recognizability. From this,
we can conclude Corollary \ref{cor:some-semitoroidal}: there are KSBA
compactifications of $F_{2d}$ whose normalizations are semitoroidal, for all degrees.
Our divisor is roughly the sum of all rational curves in $|L|$. Its
recognizability is proven below, by showing that the image of a predeformable,
stable, genus zero map to any Kulikov surface is rigid.

\subsection{Definition of $R^{\rm rc}$} Consider the moduli space $F_{2d}^{\rm q}$
of quasipolarized K3 surfaces of degree $2d$. Let $(X,L)\in F_{2d}^{\rm q}$.

\begin{definition} We say that $G\in |L|$ is a {\it rational curve} if the
normalization of every irreducible component of $G$ is $\mathbb{P}^1$.\end{definition}

\begin{theorem}[Yau-Zaslow formula \cite{yau1996bps,beauville1997counting, chen1998rational, chen2000simple}]
\label{yz-formula}
There is a Zariski open subset $U\subset F_{2d}^{\rm q}$ for which any
rational curve $G\in |L|$ for $(X,L)\in U$ is irreducible, nodal, and for
which the number of such rational curves is exactly
\begin{displaymath}
  n_d:=[q^d]\, \frac1{q} \prod_{k\geq 1} \frac{1}{(1-q^k)^{24}}
  = [q^d]\, \frac1{\Delta(q)}
\end{displaymath}
where $\Delta(q)$ is the modular discriminant, and $[q^d]$ denotes the 
$q^d$-coefficient.
\end{theorem}

The integer $n_d$ is the number of $24$-colored partitions of $d+1$.

\begin{definition}\label{naive-rc-def}
The {\it rational curve divisor} is the canonical choice
of polarizing divisor $R^{\rm rc}:=\sum_{G\in |L|\textrm{ rational }} G\in |n_dL|$
defined over the open subset $U\subset F_{2d}^{\rm q}$. \end{definition}

We now 
outline an alternative definition using Gromov-Witten invariants.

\begin{definition} Let $X$ be a smooth complete variety and let 
$\beta\in H_2(X,\Z)$.
The {\it Kontsevich space} $M_g(X,\beta)$ is the
moduli space of stable maps $f:T\to X$ from a genus $g$ nodal
curve $T$, for which $f_*[T]=\beta$.
\end{definition}

The Kontsevich space is a proper Deligne-Mumford stack.
For a surface, $H_2(X,\Z)$ and $H^2(X,\Z)$ are canonically identified
by Poincar\'e duality, so we make no distinction. We will take $L=\beta$.

There is a {\it virtual fundamental cycle}
$[M_g(X,\beta)]^{\rm vir}\in A_{\rm exp.dim}(M_g(X,\beta))$
where ${\rm exp.dim}=(\dim X-3)(1-g)+c_1(T_X)\cdot \beta$
is the expected dimension of the moduli space \cite{behrend1997intrinsic}.
In particular, for
stable genus $0$ maps to a K3 surface,
${\rm exp.dim}=-1$ so $[M_0(X,L)]^{\rm vir}=0$.
Geometrically, this can be explained by the fact that GW invariants
are deformation-invariant, but that a generic
deformation of $X$ has no nontrivial line bundles, so $\beta$
cannot represent an algebraic curve. 

For polarized K3 surfaces $(X,L)$,
there is a {\it reduced virtual fundamental cycle}
$[M_0(X,L)]^{\rm vir,red}\in A_0(M_0(X,L))$, see
\cite{kool2011reduced-classes}.
Roughly, it is built to be invariant 
only under the deformations
of $X$ which stay in $\cF_{2d}^{\rm q}$.
This decreases dimension of the obstruction space
by one, increasing the expected dimension by one.

\begin{lemma}\label{rc-constant-smooth}
 Let $(T,f)$ be a stable map $f:T\to X$ with $T$ a nodal curve of 
 arithmetic genus $0$ and $X$ a smooth K3 surface.
Under any deformation of the stable map $(T,f)\in M_0(X,L)$,
the image divisor $f_*T$ is constant. \end{lemma}

\begin{proof} If the image divisor $f_*T$ moves under a deformation
of $(T,f)$, a restriction of $f$ gives a dominant
map $S=B\times \mathbb{P}^1\xrightarrow{\pi} X$,
for some (possibly incomplete)
curve $B$. Letting ${\rm Ram}\subset S$
be the ramification divisor of the map $S\to X$, the Riemann-Hurwitz
formula gives $K_S= {\rm Ram}$. This contradicts the adjunction 
formula, because restricting to a general fiber $F=\{b\}\times \mathbb{P}^1$ gives
$-2=2g(F)-2 = F\cdot (F+{\rm Ram}) = F\cdot {\rm Ram}\geq 0.$
\end{proof}

Replicating this argument for a Kulikov surface
is the key to proving that $R^{\rm rc}$ is recognizable. 

\begin{definition}\label{coeff-of-rc}
Let $G\in |L|$ be a rational curve. Define $M_G(X)$
to be the union of the connected components of $M_0(X,L)$ for which
$f_*T=G$.
This is well-defined because $f_*T$ is constant
on any connected component by Lemma 
\ref{rc-constant-smooth}. Define
$$n_G:=\textstyle \deg_{M_G(X)} [M_0(X,L)]^{\rm vir,red}\in \Q.$$
\end{definition}

This quantity a priori only lies in $\Q$ because of stack-theoretic
issues.

\begin{proposition} \label{rc-alt-smooth}
For any smooth quasipolarized K3 surface $(X,L)$,
we have $$\textstyle R^{\rm rc}=\sum_{G\in|L|\,\rm{ rational}} n_GG.$$
Furthermore, $n_G$ is a non-negative integer for all rational curves $G\in |L|$.
\end{proposition}

\begin{proof} Chen's theorem \cite{chen2000simple} implies that for a
sufficiently general $(X,L)\in U$, we have $n_G=1$ for all rational curves $G$.
Fix an $(X_0,L_0)\in F_{2d}^{\rm q}$ not in $U$ and consider a 
$1$-parameter deformation $(\mathcal{X} ,\mathcal{L})\to (C,0)$ over an
analytic disc $C$ for which
$X_t\in U$ for all $t\in C^*$. Consider the moduli space
of relative stable maps $M_0(\mathcal{X}, \beta)$ where $\beta$ is the class of
$L_0$ pushed forward to $\mathcal{X}$.

There is a proper morphism
$M_0(\mathcal{X}, \beta)\to C$ sending a curve to the fiber
it is supported on, but the fibers of this family
are in general poorly behaved. For instance, the dimension can and 
often does
suddenly jump at $t=0$. But by assumption, the fiber over
any point $t \in C^*$ is a reduced zero-dimensional scheme
consisting of exactly $n_d$ points. The proposition follows if we can
prove that the scheme-theoretic intersection
$\overline{M_0(\mathcal{X}^*,\beta)}\cap M_0(X_0,L_0)$ represents
the reduced virtual fundamental class, in homology.

The constancy of the reduced Gromov-Witten invariants
$\textstyle n_d = \deg_{M_0(X_t,\beta)} [M_0(X_t,\beta)]^{\rm vir,red}$ as one varies $t$
 follows from the existence of a relative perfect obstruction theory
 \cite[Sec.~7]{behrend1997intrinsic}, \cite[Rem.~3.1]{kool2011reduced-classes}.
 Without going into the details, this is a perfect two-term complex with a morphism
 to the relative cotangent complex, satisfying various axioms.
  
  Now let $W\subset M_0(\mathcal{X},\beta)$ be a connected
  component.
  The restriction of the axioms of a (relative) perfect obstruction theory still
  hold under restricting this two-term complex to $W$. Hence the
  constancy of reduced GW invariants still holds, i.e.
  $\deg_W [M_0(X_0,L_0)]^{\rm vir,red}$ will equal the number of sheets of 
   $M_0(X_t,L_t)$ whose closures over $t=0$ lie in $W$.
   This implies the first statement.
   
Summing these integrals over the components $W$
for which the image curve $f_*T=G$,
we also see that $n_G$ is a non-negative integer.  \end{proof}

\begin{remark}  A priori, the contribution $n_G$ could equal zero.
Perhaps no genus $0$ stable map with image $G$ deforms
to the general fiber $X_t$. Notably, this cannot occur when 
there is a component $W\subset M_G(X)$ of dimension $\dim W=0$,
see \cite[Ch.~13.2.3]{huybrechts2016lectures-on-k3} and
references therein.
\end{remark}

Proposition \ref{rc-alt-smooth} provides us with a definition of $R^{\rm rc}$
on {\it all} $(X,L)\in F_{2d}^{\rm q}$.

 \begin{definition} A quasipolarized K3 surface $(X,L)$ of degree $2d$
  is {\it unigonal} if it is elliptic, with section and fiber classes $s$, $f$
 and $L=s+(d+1)f$. \end{definition}
 
 As a Noether-Lefschetz locus of Picard rank $2$, the unigonal locus forms
 a divisor in $F_{2d}^{\rm q}$ isomorphic to the moduli space $F_H^{\rm q}$ 
 of elliptic K3s.
 
 \begin{proposition}\label{elliptic-correct} On a unigonal K3 surface $(X,L)$, the
 rational curve divisor is $$R^{\rm rc}:=n_d(s+\tfrac{d+1}{24}\textstyle \sum f_i)$$ where
 $f_i$ are the $24$ singular fibers in $|f|$, counted with multiplicity. \end{proposition}

\begin{proof} The proposition follows immediately from Proposition \ref{rc-alt-smooth}
and the main result of \cite{bryan2000enumerative-geometry}, though historically
 \cite{chen2000simple} relies on \cite{bryan2000enumerative-geometry}.
\end{proof}

\subsection{Proof of Theorem~\ref{thm:main-rc-divisor}}
We now prove our second main result:

\begin{theorem}\label{thm:rc-divisor} The rational curve divisor $R^{\rm rc}$ is recognizable
for $F_{2d}$ for all $d>0$. \end{theorem}

\begin{proof}
Take a divisor model $(X,R)\to (C,0)$. We verify
Theorem \ref{thm:equiv-rec}(1) by showing that
the limiting curve $R_0\subset X_0$ satisfies some geometric
property ensuring its rigidity on $X_0$ even as we deform the smoothing
of $X_0$. Take a base change and standard
resolution of $(X,R)\to (C,0)$ so that the irreducible components $G_t$ of $R_t$ are
not permuted by monodromy. Then, Lemmas \ref{rc-rec1} and
\ref{rc-rec2} imply that the limit of any individual rational curve $G_t$
is rigid. \end{proof}

\begin{lemma}\label{rc-rec1}Let $X\to (C,0)$ be a Kulikov model and
let $G\subset X$ be a flat family of curves for which $G_t$ is an
irreducible rational curve for $t\neq 0$, and $G_0$ contains no strata.
Then, after a finite base change and resolution of $X\to (C,0)$, there
is a stable map $f:T\to X_0$ from a nodal,
genus $0$ curve (a tree of $\mathbb{P}^1$s) for which $f_*T=G_0$
and $(T,f)$ is {\rm predeformable} (see Def.~\ref{loc-cartier}). \end{lemma}

\begin{definition}\label{loc-cartier} We say that $(T,f)$ is {\it predeformable}
 \cite[Def.~2.5]{li2001stable} 
if no component of $T$ is contracted into the double locus, and for each
node $p\in T$ with $f(p)\in D_{ij}$, the two arcs
$(T_k,p)$, $(T_\ell,p)$ with $f(T_k)\subset V_i$ and $f(T_\ell)\subset V_j$
satisfy
\begin{align*} &\textrm{the tangency order of }f(T_k,p)\textrm{ to }D_{ij} = 
\textrm{the tangency order of }f(T_\ell,p)\textrm{ to }D_{ji}.\end{align*}
\end{definition}

\begin{proof}[Proof of Lemma \ref{rc-rec1}]
 Because $G_0$ contains no strata, it maps into the complement
 of the triple points of $X_0$, i.e. the union of the non-singular locus
 and the double locus. Then, the result follows 
 from the properness over $(C,0)$ of the space of predeformable
 stable maps \cite[Thm.~3.10]{li2001stable} to varieties
 with only double crossings. \end{proof}

\begin{lemma}\label{rc-rec2} Let $f:T \to X_0\times B$ be a family of stable
maps over a local curve $B$, such that $T_b$ is a tree of
$\mathbb{P}^1$s of fixed combinatorial type for all $b\in B$, and
for which $(T_b, f_b)$ is predeformable. Then the image
curves $f_*(T_b)=G_0$ are constant.\end{lemma}

\begin{proof} Let $T=\cup_k T_k$ be the components of $T$. We have
$T_k\cong \mathbb{P}^1\times B$. Let $N_{k\ell}=T_k\cap T_\ell$ be
the relative nodes over $B$. We label the vertices $\Gamma(T)^{[0]}$ of the
dual complex $\Gamma(T)$ as follows:
\begin{enumerate}
\item[(V0)] $T_k$ is contracted to a point inside  a component.
\item[(V1b)] $T_k$ is contracted along multisections to a curve inside
  a component.
\item[(V1f)] $T_k$ is contracted along fibers of $T_k\to B$ to a curve inside a component.
\item[(V2)] $T_k$ maps generically finitely to a component.
\end{enumerate}
These are the only possibilities, by noting that $T_k\to B$ is
proper and that the image of $f$ contains no triple points. Next,
we label the edges $\Gamma(T)^{[1]}$ of the dual complex $\Gamma(T)$ as follows:
\begin{enumerate}
\item[(v0)] $N_{k\ell}$ maps to a point in the interior of a component.
\item[(v1)] $N_{k\ell}$ maps to a curve in the interior of a component.
\item[(d0)] $N_{k\ell}$ maps to a point in a double curve.
\item[(d1)] $N_{k\ell}$ maps to a curve in a double curve.
\end{enumerate}

Table \ref{adjacencies} records 
the allowable adjacencies for the labeled dual complex $\Gamma(T)$, which
can be verified from predeformability by straightforward geometric arguments.

Let $\Gamma\subset \Gamma(T)$ be a maximal subtree consisting of
only V2-vertices and d1-edges. Let $T_\Gamma\subset T$ be the
sub-family of curves with dual complex $\Gamma$. Consider the restricted
family $f_\Gamma: T_\Gamma\to X_0\times B.$

\def\arraystretch{1.3}
\begin{table}
\begin{tabular}{c|cccccc}
 & V0 & V1b & V1f & V2 \\
 \hline 
V0 & v0 & v0 & --- & v0\\
V1b & & v0/v1/d0 & v1 & v0/v1/d0  \\
V1f & & & v1 & v1  \\
V2 & & & & all
\end{tabular}
\vspace{5pt}
\caption{Allowable adjacencies for the labeled dual complex $\Gamma(T)$.}
\vspace{-15pt}
\label{adjacencies}
\end{table}

The fibers $T_{\Gamma,b}$ may only fail to map in a predeformable
way to $X_0$ at the leaves of $\Gamma$ which are not leaves of 
$\Gamma(T)$. Consider the edges emanating from such a V2 leaf 
which are connected to the rest of $\Gamma(T)$. Disconnecting
 $\Gamma(T)$ at a v-edge does not interfere with the condition of 
 being predeformable, so consider only the d-edges. By Table
  \ref{adjacencies} and maximality of $\Gamma$, such a V2 leaf 
  of $\Gamma$ must connect by a d0-edge.
  
So fix one V2 leaf of $\Gamma$, associated to a component
$T_k\cong \mathbb{P}^1\times B\subset T_\Gamma$
 attached to the rest of $\Gamma(T)$ by
d0-edges. The further restriction $f_k:T_k\to V_i$
is now a map of smooth surfaces.

Each outgoing d0-edge corresponds to a relative node $N_{k\ell}$
of $T$ which maps under $f_k$ to a single point $p_{k\ell}\in D_{ij}$. 
There is at most one remaining relative node $N\subset T_k$
which attaches $T_k$ to the rest of $T_\Gamma$ and for which $f(N)$
is a curve in one boundary component of $V_i$.
Make an interior blow-up $\widetilde{V}_i\to V_i$ at each fixed attaching point $p_{k\ell}$.
Taking the strict transforms of the images of fibers of $T_k\to B$,
we can lift $f_k$ to a map $\widetilde{f}_k\colon T_k \to \widetilde{V}_i.$
If $\widetilde{f}_k$ still sends any $N_{k\ell}$ to a point in the new anticanonical boundary
$\widetilde{D}_{ij}$, we continue to blow up at the fixed attaching points, until
the lifted map satisfies the property $\widetilde{f}_k^{-1}(\widetilde{D}_i)=N$.

Since both $N$ and $\widetilde{D}_i$
are divisors with coefficient $1$, we have by 
Riemann-Hurwitz that
 $$\omega_{T_k}(N)=\widetilde{f}_k^*(\omega_{\widetilde{V}_i}(\widetilde{D}_i))
 \otimes \mathcal{O}({\rm Ram})\otimes \mathcal{O}(\textstyle \sum a_iE_i).$$ Here
 ${\rm Ram}\subset T_k$ is the interior ramification divisor,
 i.e.~the ramification away from 
 the boundary $\widetilde{D}_i$, and $E_i\subset T_k$ are the contracted curves of
 $\widetilde{f}_k$. Note that $a_i\geq 0$ because $\widetilde{V}_i$ is smooth.
 Since $\omega_{\widetilde{V}_i}(\widetilde{D}_i)=\mathcal{O}$ we conclude
 $\omega_{T_k}(N)$ is effective, implying
 $-1=\omega_{T_k}(N)\cdot \mathbb{P}^1 \geq 0$.
Contradiction.
\end{proof}

In analogy with Proposition~\ref{rc-alt-smooth},
Lemma~\ref{rc-rec2} allows
us to define $R^{\rm rc}$ for Kulikov surfaces inherently,
in terms of logarithmic Gromov-Witten invariants 
\cite{chen2014stable, abramovich2014stable, gross2013logarithmic}.

\subsection{The rational curve semifan} We first give some general results concerning
the Baily-Borel compactification of $F_{2d}$, following
\cite{scattone1987on-the-compactification-of-moduli}.

The number of $0$-cusps of
$\oF_{2d}^\bb$ is exactly $\lfloor \tfrac{N+2}{2}\rfloor$ where
$d = N^2d_0$ for a square-free integer $d_0$. As discussed in Section \ref{sec:bb},
they are in bijection with the $\Gamma$-orbits of primitive isotropic lattices $I=\Z\delta$
in the lattice $$L_{2d}:=\langle -2d\rangle \oplus H^{\oplus 2} \oplus E_8^{\oplus 2}
=v^\perp\subset L_{\rm K3}.$$ The $\Gamma$-orbit of a generator $\delta\in I$ is
determined by the following invariant:
$\delta^* = \frac{\delta}{p^*(\delta)}\in \Delta_{2d}:=L^*_{2d}/L_{2d},$
where $p^*$ is by definition the imprimitivity in $L_{2d}^*$. Then $\delta^*$
is an isotropic vector for the quadratic form on $\Delta_{2d}$ valued in $\tfrac{1}{2d}\Z/2\Z\subset \Q/2\Z$.
Identifying the source $\Delta_{2d}=\Z/2d\Z$ and the target with $\Z/4d\Z$, the quadratic form
is given by $x\mapsto x^2$. We must have $\delta^*=2qNd_0\in \Z /2d\Z$ for $q\in\Z/N\Z$.
So $I=\Z\delta$ is determined by $\{\pm q\}$ and we have
$$\delta^\perp/\delta = \langle \tfrac{-2d}{p^*(\delta)}\rangle \oplus H\oplus E_8^{\oplus 2}.$$

A semitoroidal compactification of $F_{2d}$ is determined by a collection
of $\Gamma_\delta$-invariant semifans $\mathfrak{F}_\delta$ decomposing
the rational closures $C_\delta^+$ of the positive cones of each lattice
$\delta^\perp/\delta$ as above, as one ranges over the
$\lfloor \tfrac{N+2}{2}\rfloor$ possible values of $\{\pm \delta^*\}$.
By Theorems \ref{thm:recognizable-semitoroidal} and \ref{thm:main-rc-divisor},
we may define: 

\begin{definition} Let $\mathfrak{F}^{\rm rc}$ be the
semifan for which $\nu\colon \oF_{2d}^{\mathfrak{F}^{\rm rc}}\to \oF_{2d}^{R^{\rm rc}}$
is the normalization map. \end{definition}

Some facts about the combinatorics of $\mathfrak{F}^{\rm rc}$ can be deduced
from Proposition \ref{elliptic-correct} and \cite{alexeev2022compactifications-moduli},
by restricting to the locus of elliptic K3 surfaces.

\begin{theorem}\label{restrict-ell}
Consider a cone $C_\delta^+$ with invariant $p^*(\delta)=1$, that is, 
where $\delta$ is primitive in $L_{2d}^*$
(all such $\delta$ are equivalent under $\Gamma$).
The restriction of $\mathfrak{F}^{\rm rc}$ to
$\langle -2d\rangle^\perp=H\oplus E_8^{\oplus 2}$,
or any $\Gamma_\delta$-orbit of it, is a fan.
Furthermore, $\langle -2d\rangle^\perp\cap C_\delta^+$
is a union of cones of $\mathfrak{F}^{\rm rc}_\delta$. \end{theorem}

We call such hyperplanes {\it unigonal}. The last statement in the theorem
implies that $\mathfrak{F}^{\rm rc}$ refines the
unigonal hyperplane arrangement in $C_\delta^+$.

\begin{proof} Suppose $(X,L)$ is in the unigonal locus, so that $L=s+(d+1)f$.
 The inclusion $\Z L\hookrightarrow H$ induces
an inclusion of moduli spaces $F_H\to F_{2d}$. So
the restriction of $R^{\rm rc}$ to $F_H$ is recognizable.
Suppose that $\delta\in H^\perp$ is primitive isotropic. There is a unique
isometry orbit of such and $\delta$
is primitive in $(L^\perp)^*$. So $\delta^{\perp}_{H^\perp}/\delta$
includes into the lattice $\delta^{\perp}_{L^\perp}/\delta$ corresponding an
isotropic  vector with invariant $\delta^*=0$. Concretely, it is the summand inclusion
$H\oplus E_8^{\oplus 2}\hookrightarrow \langle -2d\rangle\oplus H \oplus E_8^{\oplus 2}$. 

By Proposition \ref{inclusion},
the restriction of $\mathfrak{F}^{\rm rc}$ to $H\oplus E_8^{\oplus 2}$
is the semifan $\mathfrak{F}^{\rm rc}_H$ whose corresponding
semitoroidal compactification normalizes $\oF_H^{R^{\rm rc}}$.
By Proposition \ref{elliptic-correct}, the rational curve divisor, 
as in Definition \ref{naive-rc-def}, when extended to the unigonal locus,
is a multiple of $R = s+m\sum f_i$ for $m=\tfrac{d+1}{24}$.
\cite[Thm.~1.2]{alexeev2022compactifications-moduli} gives an explicit description 
of the fan $\mathfrak{F}^{\rm rc}_H$ modulo the following caveat:
The divisor models described in \cite[Sec.~7B]{alexeev2022compactifications-moduli}
require a threshold value of $m> \tfrac{1}{3}$ for the divisors $R_0\subset X_0$
constructed therein to be nef. The threshold is achieved when $\Gamma(X_0)$ has a
so-called $X_3$ end singularity. So it is automatic that {\it loc.cit} describes the restriction
of $\mathfrak{F}^{\rm rc}$ to the unigonal hyperplane when $d> 7$. For $m\leq \tfrac{1}{3}$ or $d\leq 7$,
the stable models only differ from those in {\it loc.cit.} in a minor way---one might
contract the section on one or both end surfaces. But the stratum function $\bS$ has
the same level sets and so $\fF_H^{\rm rc}$ is the same (Prop.~\ref{stratum-function}).

The fan $\mathfrak{F}_H^{\rm rc}$ consists
of six orbits of maximal cones \cite[Sec.~4C]{alexeev2022compactifications-moduli}.
To prove the final statement of the theorem,
we must show that all six of the $18$-dimensional cones
$\sigma_H\in \mathfrak{F}_H^{\rm rc}$
are themselves cones of $\mathfrak{F}^{\rm rc}$ and not simply slices
of the interior of some  $19$-dimensional cone $\sigma \in \mathfrak{F}^{\rm rc}$.

A maximal cone $\sigma_H$ corresponds to a $0$-stratum of the stable
pair compactification of elliptic K3s and hence to unique Type III elliptic stable K3 pair $(\oX_0,\oR_0)$.
If $(\oX_0,\oR_0)$ deforms out of the unigonal locus
as rational curve K3 pair, keeping the combinatorial type constant,
then $\sigma_H$ must be a cone of $\mathfrak{F}^{\rm rc}$. But if the elliptic
stable K3 pair $(\oX_0,\oR_0)$ is, as a rational curve K3 pair, rigid in its combinatorial type,
then $\sigma_H$ must be the slice of a larger dimensional cone of $\mathfrak{F}^{\rm rc}$.

Let $(X_0,R_0)$ be a Type III divisor model whose stable model is $(\oX_0, \epsilon \oR_0)$, see \cite[Sec.~7A]{alexeev2022compactifications-moduli} for an explicit description.
Let $L_0=\mathcal{O}_{X_0}(R_0)$.
We can deform $(X_0,L_0)$ to a non-elliptic, $d$-semistable Kulikov model $(X_0',L_0')$
by regluing double curves so that $\psi_{X_0'}(f)\neq 1$.
Concretely, comparing to \cite[Def.~7.10]{alexeev2022compactifications-moduli},
it corresponds to when a connected chain of fibers of vertical rulings {\it fails to glue}
to a closed cycle, destroying the elliptic fibration and the Cartierness of $f$.

Since $R^{\rm rc}$ is recognizable, the rational curve divisor on
such a deformed Kulikov model is necessarily a deformation of the curve
$n_d(s+ \tfrac{d+1}{24}\sum f_{i,0})$ (see \ref{ell-ex}) living in the linear system $|L_0'|$.
So the resulting stable model has the same combinatorial type as the elliptic one
$$(\oX_0',\epsilon \oR_0')=\textstyle\bigcup_{i=1}^r (\oV_i,\oD_i,\epsilon\oR_i)$$ for $r=18$, $19$, $20$
depending on the cone $\sigma_H$.

In the elliptic case, the intermediate components $\oV_i$ for $i\neq 1,r$ are
the result of gluing two sections
of $\mathbb{P}^1\times \mathbb{P}^1$ via the isomorphism provided
by the vertical fibration. But when we deform $X_0$ out
of the elliptic locus to the Kulikov surface $X_0'$,
the surface $\oV_i$ also
deforms: The gluing map between the two sections includes a shift exactly equal to
$\psi_{X_0'}(f)$.

Hence $(\oX_0,\epsilon\oR_0)$ is {\it not}
rigid in $\oF_{2d}^{R^{\rm rc}}$ within its slc combinatorial type.
Even forgetting the divisor,
the underlying surface $\oX_0$ is not rigid. We conclude that
$\sigma_H$ is a cone of $\mathfrak{F}^{\rm rc}$.
\end{proof}

\begin{remark}\label{imprim-case} The results of this section hold
for the {\it imprimitive rational curve divisor}
$$R^{\rm rc}(m)=\!\!\!\!\!\!\!\!\sum_{G\in |mL|\textrm{ rational}}\!\!\!\!\!\!\!\! n_GG$$
where the coefficients $n_C$ are defined using reduced GW invariants
as in Definition \ref{coeff-of-rc}. A naive version of Chen's theorem (that generically
all rational curves are nodal) is false: For instance one can take $mG$ for $G\in |L|$
rational. It is not clear whether one can recover Chen's theorem
by subtracting out these and other obvious non-reduced and non-irreducible
contributions to get a divisor $R^{\rm rc}_{\rm prim}(m)$.
Regardless, the above serves as a definition of $R^{\rm rc}(m)$ and produces a
canonical choice of polarizing divisor. The proof of recognizability,
Theorem~\ref{thm:rc-divisor}
applies verbatim because the normalization of any
irreducible component of $R^{\rm rc}(m)$
is $\mathbb{P}^1$.

Thus, there are semifans $\mathfrak{F}^{\rm rc}(m)$ for all $m\geq 1$ which give the normalization of the KSBA compactification associated to $R^{\rm rc}(m)$.  \end{remark}

\bibliographystyle{amsalpha}
%\bibliography{va}

\def\cprime{$'$}
\providecommand{\bysame}{\leavevmode\hbox to3em{\hrulefill}\thinspace}
\providecommand{\MR}{\relax\ifhmode\unskip\space\fi MR }
\providecommand{\MRhref}[2]{%
  \href{http://www.ams.org/mathscinet-getitem?mr=#1}{#2}
}
\providecommand{\href}[2]{#2}

\end{document}